%% file: PrymTautological.tex
\documentclass[11pt]{amsart}

\usepackage{bbm}

\usepackage{comment}
\usepackage{amsmath}							
\usepackage{amssymb}
\usepackage{amsthm}
\usepackage{amscd}
\usepackage{amsfonts}
\usepackage[all]{xy}
\usepackage{colonequals}
\usepackage{upgreek}
\usepackage{graphicx}

\usepackage{hhline}
\usepackage{etoolbox}

\usepackage{afterpage}

\usepackage{centernot}

\usepackage{longtable}
\usepackage{booktabs}

\newcommand{\thickhline}{%
  \noalign{\global\arrayrulewidth=1pt}%
  \hline
  \noalign{\global\arrayrulewidth=0.4pt}%
}
\usepackage{array}
\newcolumntype{P}[1]{>{\arraybackslash}p{#1}}
\usepackage{pdflscape}

\newcounter{subsubsubsection}[subsubsection]

\makeatletter
\DeclareRobustCommand{\cev}[1]{%
  \mathpalette\do@cev{#1}%
}
\newcommand{\do@cev}[2]{%
  \fix@cev{#1}{+}%
  \reflectbox{$\m@th#1\vec{\reflectbox{$\fix@cev{#1}{-}\m@th#1#2\fix@cev{#1}{+}$}}$}%
  \fix@cev{#1}{-}%
}
\newcommand{\fix@cev}[2]{%
  \ifx#1\displaystyle
    \mkern#23mu
  \else
    \ifx#1\textstyle
      \mkern#23mu
    \else
      \ifx#1\scriptstyle
        \mkern#22mu
      \else
        \mkern#22mu
      \fi
    \fi
  \fi
}
\makeatother

\usepackage{caption}
\usepackage{subcaption}
\usepackage{euler}

\usepackage{extarrows}

\usepackage[colorlinks, linktocpage, citecolor = purple, linkcolor = blue,backref=page]{hyperref}
\usepackage{zref-clever}
\zcsetup{
  cap = true,
  abbrev = false,
  nameinlink = false
}

\usepackage{color}
\usepackage{faktor}
\usepackage[table]{xcolor}

\usepackage[margin=1in]{geometry}

\usepackage{tikz}							 
\usepackage{float} 
\usetikzlibrary{matrix}
\usetikzlibrary{patterns}
\usetikzlibrary{positioning}
\usetikzlibrary{decorations.pathmorphing}
\usetikzlibrary{cd}
\definecolor{lilla}{RGB}{199,85,255}
\usepackage{ytableau}
\usepackage[shortlabels]{enumitem}

\newcolumntype{C}[1]{>{\centering\arraybackslash}p{#1}}

\theoremstyle{plain}
\newtheorem{innercustomthm}{Theorem}
\newenvironment{customthm}[1]
{\renewcommand\theinnercustomthm{#1}\innercustomthm}
{\endinnercustomthm}

\newtheorem{innercustomalg}{Algorithm}
\newenvironment{customalg}[1]
{\renewcommand\theinnercustomalg{#1}\innercustomalg}
{\endinnercustomalg}

\newtheorem{theorem}{Theorem}[section]
\AddToHook{env/theorem/begin}{%
  \zcsetup{countertype={theorem=theorem}}%
}

\newtheorem{lemma}[theorem]{Lemma}
\AddToHook{env/lemma/begin}{%
  \zcsetup{countertype={theorem=lemma}}%
}

\newtheorem{proposition}[theorem]{Proposition}
\AddToHook{env/proposition/begin}{%
  \zcsetup{countertype={theorem=proposition}}%
}

\newtheorem{corollary}[theorem]{Corollary} 
\AddToHook{env/corollary/begin}{%
  \zcsetup{countertype={theorem=corollary}}%
}

\newtheorem*{corollary*}{Corollary}

\AddToHook{env/conjecture/begin}{%
  \zcsetup{countertype={theorem=conjecture}}%
}

\AddToHook{env/defprop/begin}{%
  \zcsetup{countertype={theorem=Definition-Proposition}}%
}

\theoremstyle{definition}

\newtheorem{algorithm}[theorem]{Algorithm}
\AddToHook{env/algorithm/begin}{%
  \zcsetup{countertype={theorem=algorithm}}%
}

\newtheorem*{algorithm*}{Algorithm}

\newtheorem{problem}[theorem]{Problem}	
\AddToHook{env/problem/begin}{%
  \zcsetup{countertype={theorem=problem}}%
}

\zcRefTypeSetup{problem}{
  Name-sg = Problem,
  name-sg = problem,
  Name-pl = Problems,
  name-pl = problems
}

\newtheorem{definition}[theorem]{Definition}
\AddToHook{env/definition/begin}{%
  \zcsetup{countertype={theorem=definition}}%
}

\newtheorem{example}[theorem]{Example}
\AddToHook{env/example/begin}{%
  \zcsetup{countertype={theorem=example}}%
}
\newtheorem*{example*}{Example}

\newtheorem{construction}[theorem]{Construction}
\AtEndEnvironment{construction}{\qed}
\AddToHook{env/construction/begin}{%
  \zcsetup{countertype={theorem=construction}}%
}

\zcRefTypeSetup{construction}{
  Name-sg = Construction,
  name-sg = construction,
  Name-pl = Constructions,
  name-pl = constructions
}

\newtheorem{remark}[theorem]{Remark}
\AddToHook{env/remark/begin}{%
  \zcsetup{countertype={theorem=remark}}%
}

\newtheorem*{remark*}{Remark}

\AddToHook{env/question/begin}{%
  \zcsetup{countertype={theorem=question}}%
}

\makeatletter
\def\@tocline#1#2#3#4#5#6#7{\relax
	\ifnum #1>\c@tocdepth % then omit
	\else
	\par \addpenalty\@secpenalty\addvspace{#2}%
	\begingroup \hyphenpenalty\@M
	\@ifempty{#4}{%
		\@tempdima\csname r@tocindent\number#1\endcsname\relax
	}{%
		\@tempdima#4\relax
	}%
	\parindent\z@ \leftskip#3\relax \advance\leftskip\@tempdima\relax
	\rightskip\@pnumwidth plus4em \parfillskip-\@pnumwidth
	#5\leavevmode\hskip-\@tempdima
	\ifcase #1
	\or\or \hskip 1em \or \hskip 2em \else \hskip 3em \fi%
	#6\nobreak\relax
	\dotfill\hbox to\@pnumwidth{\@tocpagenum{#7}}\par
	\nobreak
	\endgroup
	\fi}
\makeatother

\newcommand{\lambdatilde}{\widetilde{\lambda}}
\newcommand{\chitilde}{\widetilde{\chi}}
\newcommand{\Etilde}{\widetilde{\mathbb{E}}}

\newcommand{\Acalbar}{\overline{\mathcal{A}}}

\newcommand{\Mcalbar}{\overline{\mathcal{M}}}
\newcommand{\Mcaltilde}{\widetilde{\mathcal{M}}}
\newcommand{\Rcalbar}{\overline{\mathcal{R}}}
\newcommand{\Rcal}{\mathcal{R}}

\newcommand{\Lcal}{\mathcal{L}}

\newcommand{\CH}{CH}
\newcommand{\PP}{\operatorname{PP}}
\newcommand{\PL}{\operatorname{PL}}
\newcommand{\PPhat}{\widehat{\operatorname{PP}}}

\newcommand{\rt}{\mathrm{rt}}

\newcommand{\dual}{\vee}

\newcommand{\GL}{\operatorname{GL}}

\newcommand{\Span}{\operatorname{Span}}

\newcommand{\ch}{\operatorname{ch}}

\newcommand{\ord}{\operatorname{ord}}

\newcommand{\Irr}{\operatorname{Irr}}

\renewcommand{\log}{\mathsf{log}}

\newcommand{\llbracket}{[\![}
\newcommand{\rrbracket}{]\!]}

\newcommand{\gl}{\operatorname{gl}}

\newcommand{\Acal}{\mathcal{A}}
\newcommand{\Xcal}{\mathcal{X}}
\newcommand{\Hcal}{\mathcal{H}}
\newcommand{\Mcal}{\mathcal{M}}
\newcommand{\Ucal}{\mathcal{U}}

\newcommand{\tw}{\operatorname{tw}}
\newcommand{\orb}{\operatorname{orb}}

\newcommand{\Td}{\operatorname{{Td}}}
\newcommand{\res}{\operatorname{Res}}

\newcommand{\EE}{\mathbb{E}}
\newcommand{\OO}{\mathcal{O}}
\newcommand{\Bcal}{\mathcal{B}}

\newcommand{\Gm}{\mathbb{G}_m}

\newcommand{\R}{\mathbb{R}}
\newcommand{\Z}{\mathbb{Z}}
\newcommand{\Q}{\mathbb{Q}}

\newcommand{\ZZ}{\mathbb{Z}}

\newcommand{\NN}{\mathbb{N}}
\newcommand{\QQ}{\mathbb{Q}}

\newcommand{\Rder}{R^\bullet}

\newcommand{\kfield}{\Bbbk}

\newcommand{\taut}{\operatorname{taut}}

\DeclareMathOperator{\Ker}{Ker}
\DeclareMathOperator{\Coker}{Coker}
\newcommand{\Hom}{\mathsf{Hom}}
\newcommand{\Ext}{\mathsf{Ext}}
\DeclareMathOperator{\rk}{rk}
\DeclareMathOperator{\Sym}{Sym}

\let\Im\relax
\DeclareMathOperator{\Im}{Im}

\newcommand{\ul}[1]{\underline{#1}}

\newcommand{\Ical}{\mathcal{I}}

\DeclareMathOperator{\Id}{Id}

\DeclareMathOperator{\Nm}{Nm}

\DeclareMathOperator{\Prym}{Prym}

\newcommand{\lcm}{\operatorname{lcm}}

\newcommand{\Aut}{\operatorname{Aut}}

\DeclareMathOperator{\Jac}{Jac}

\title[Tautological projections of Prym classes]{Hodge integrals on admissible covers\\and tautological projections of Prym classes}
\author{Yoav Len, Sam Molcho, Navid Nabijou}

\begin{document}

\begin{abstract} 
In the moduli space of abelian varieties, Prym loci parametrise Prym varieties associated to covers of curves with fixed degree and monodromy data. We derive and implement an algorithm computing the tautological projections of classes of Prym loci. The key input is a comparison formula relating three different Hodge bundles on moduli spaces of admissible covers. We present computer and hand calculations for double and triple covers, with the locus of hyperelliptic Jacobians as a special case. Two applications follow. First, we calculate the degrees of the Prym maps in the codimension-zero cases, unifying a series of previously ad hoc results. Second, we prove in low genera the non-divisibility of the Jacobian class by the Prym class.
\end{abstract}

\maketitle

\setcounter{secnumdepth}{4}
\setcounter{tocdepth}{1}
\tableofcontents

%%%%%%%%%%%%%%%%%%%%%%%%%%%%%%%%%%%%%%%%%%%%%%%%%%%%%%

\newpage

\section*{Introduction}

\subsection{Background}\label{sec: tautological ring introduction}

\subsubsection{Moduli of abelian varieties and the tautological ring}  The moduli space $\Acal_g$ of principally polarised abelian varieties is a smooth Deligne--Mumford stack of dimension $\binom{g+1}{2}$. It admits a system of toroidal compactifications
\[ \Acalbar_g^\Sigma \supseteq \Acal_g \]
indexed by suitable fan-theoretic data $\Sigma$ \cite{AMRT,FaltingsChai}. The universal abelian variety extends to a universal semiabelian variety over each such toroidal compactification, whose cotangent space at the identity element defines a rank-$g$ vector bundle
\[ \EE \to \Acalbar_g^\Sigma \]
called the Hodge bundle. The lambda classes are the Chern classes of the Hodge bundle:
\[
\lambda_i \colonequals c_i(\EE) \in CH^i(\Acalbar_g^\Sigma).
\]
They play a fundamental role in the intersection theory of $\Acal_g$ and its compactifications. In particular, the tautological ring
\[ R^\star(\Acalbar_g) \subseteq CH^\star(\Acalbar_g^\Sigma) \]
is the subring generated by the lambda classes. It is independent of the choice of $\Sigma$ and admits a simple presentation as a graded ring \cite{vdG,EW}:
\begin{equation} \label{eqn: presentation taut introduction} R^\star(\Acalbar_g) = \Q[\lambda_1,\ldots,\lambda_g]/ \big( (1+\lambda_1+\cdots+\lambda_g)(1-\lambda_1+\cdots+(-1)^g \lambda_g)=1 \big).\end{equation}
The tautological ring is in particular much simpler than its counterpart for the moduli space of curves. This simplicity leads to new structures which, in turn, lead to new questions. Among the most significant of these structures is the tautological projection.

\subsubsection{Tautological projection} From the presentation \eqref{eqn: presentation taut introduction} we deduce the Gorenstein property of $R^\star(\Acalbar_g)$: the Poincar\'e pairing remains perfect when restricted from the Chow ring to the tautological subring. The tautological projection
\begin{equation} \taut \colon CH^\star(\Acalbar_g^\Sigma) \to R^\star(\Acalbar_g) \end{equation}
is then defined by taking an arbitrary Chow class $\gamma$, computing its intersection against a basis of complementary-dimensional tautological classes, and passing through the restricted Poincar\'e pairing to obtain a tautological class $\taut(\gamma)$.

\subsubsection{Investigating the tautological projection} The tautological projection plays a fundamental role in the intersection theory of the moduli space of abelian varieties. Its compatibility with the ring structure is the subject of an open conjecture \cite{IribarLopez_NoetherLefschetz,HomomorphismConjecture} and it provides a key tool for investigating tautological relations, both on the moduli space of abelian varieties and on the moduli space of curves \cite{TautNonTaut,CanningLarsonSchmitt}.

A key problem is to compute tautological projections of naturally-occurring classes. The fundamental example is the Torelli class: the class of (the closure of) the locus of Jacobians. This is equal to the pushforward of the fundamental class along the Torelli map
\begin{align} \label{eqn: Torelli introduction} \operatorname{Tor} \colon \Mcalbar_g & \to \Acalbar^\Sigma_g \\
\nonumber C & \mapsto \Jac(C), \end{align}
and so computing its tautological projection amounts to computing Hodge integrals on $\Mcalbar_g$. Faber provides an algorithm \cite{Faber_algorithms} for such integrals, based on Mumford's formula \cite{MumfordTowards}.

%It is not known whether the Torelli class itself is tautological, in which case it would coincide with its tautological projection.  
%The combinatorial complexity of the algorithm is dampened by exploiting the invariance of the Hodge bundle along maps forgetting markings. 

Besides the Torelli class, the tautological projection has been computed for product classes and, more generally, Noether--Lefschetz classes \cite{IribarLopez_NoetherLefschetz}. For adjacent results see \cite{CMOP_tautologicalProjection, TautNonTaut, DrakengrenSelf, DrakenGrenFibre, IPT, HomomorphismConjecture}.

\subsection{Algorithm and implementation} This paper studies another collection of naturally-occurring classes: the Prym classes. These arise from moduli spaces parametrising covers 
\[ f \colon C \to B \]
of Riemann surfaces with fixed degree, number of branch points, and ramification profiles. Given such a cover, the associated Prym variety
\[ \Prym(f) \]
is the dual of the quotient $\Jac(C)/f^\star \Jac(B)$. Prym varieties typically sweep out a larger locus inside the moduli space than the locus of Jacobians, see \zcref{sec: divisibility} for a direct comparison. The associated \textbf{Prym class}
\[ [\Prym] \in CH^\star(\Acalbar_{g,\delta}^\Sigma) \]
is the class of the closure of the locus of Prym varieties (strictly speaking this description holds only if the Prym map is generally injective, but this is true in most cases; see \zcref{sec: Prym} for the definition in general). 

Because Prym classes depend on a large number of parameters (including the degree, number of branch points, and ramification profiles) they form a rich and highly structured collection of classes in the moduli space of abelian varieties. They include as a special case the class of the locus of Jacobians of $d$-gonal curves.

We present our first main result:
\begin{customalg}{A}[\zcref{algorithm main}] \label{algorithm introduction} There is a recursive algorithm computing the tautological projection of the Prym class, for any fixed degree, number of branch points, and ramification profiles.\footnote{We in fact compute Prym classes associated to the more refined data of a global type, see \zcref{sec: introduction higher degree covers}. Appropriate sums of these then give Prym classes associated to a fixed degree, number of branch points, and ramification profiles.} The algorithm hinges on a comparison formula for Hodge classes (Theorem~\ref{thm: comparison introduction}) together with a recursive property of those classes (\zcref{prop: recursive structure of Prym Hodge}).\end{customalg}

We computer-implement this algorithm for $d=2$ and arbitrary ramification in ancillary \texttt{Sage} code. Tables of the resulting calculations are presented in \zcref{appendix}, including tautological projections of classes of hyperelliptic Jacobians. Examples can be found overleaf.

Beyond double covers, we also carry out several $d=3$ computations by hand: see Sections~\ref{sec: degree 3 covers non-cyclic}~and~\ref{sec: degree 3 covers cyclic} and Theorem~\ref{thm: degree Prym introduction} below.

% EXAMPLE TABLES ON NEXT FREE PAGE.
\afterpage{
    \clearpage
    \input{IntroductionTables}
    \clearpage
    }

\subsection{Tautological projection for $\Acal_g$} \label{rmk: taut projection interior introduction} Remarkably, the non-proper moduli space $\Acal_g$ also admits a tautological projection \cite{CMOP_tautologicalProjection}:
\[ \taut \colon CH^\star(\Acal_g) \to R^\star(\Acal_g). \]
Given a class $\gamma \in CH^\star(\Acalbar_g^\Sigma)$ we then have (see \eqref{eqn: commuting square taut projection interior}):
\[ \taut(\gamma|_{\Acal_g}) = \taut(\gamma)|_{\lambda_g=0}.\]
Therefore our algorithm also computes the tautological projections of the Prym classes in $\Acal_g$, by simply setting $\lambda_g=0$ in the final formulae. Since the computation on the compactified moduli space contains more information, we present our calculations in this setting throughout. However, we expect the projection on the interior to have better structural properties.

\subsection{Applications} We deduce two results unconnected to the tautological ring. The first concerns the degree of the Prym map, the second the non-divisibility of the Jacobian class by the Prym class.

\subsubsection{Degrees of Prym maps} \label{sec: degree Prym introduction} Suppose the parameters are chosen so that the moduli spaces of covers and abelian varieties have the same dimension. Then the Prym class has codimension zero. In codimension zero the tautological and Chow rings coincide, and so
\begin{equation*} \taut [\Prym] = [\Prym] = \deg(\Prym) \cdot \mathbbm{1}, \end{equation*}
where $\deg(\Prym)$ is the degree of the Prym map sending a cover to the associated Prym variety. Our algorithm therefore computes the degree of the Prym map in this case.

In \zcref{sec: degree of Prym} we enumerate the various cases for which the Prym class has codimension zero, and then use our algorithm to calculate the degree of the associated Prym map in each case. This gives a uniform intersection-theoretic proof of a corpus of results established by various authors over many years \cite{DonagiSmith_StructurePrym,NagarajRamanan,BardelliCilibertoVerra,MarcucciNaranjo,LangeOrtegaTriple}.

\begin{customthm}{B}[Propositions~\ref{thm: 27}--\ref{prop: deg 3 non-cyclic}] \label{thm: degree Prym introduction} Let $\Prym_2(g,n)$ denote the map of moduli spaces sending a double cover with genus $g$ base curve and $n$ branch points to the associated Prym variety. Then we have:
\begin{itemize}
\item $\deg(\Prym_2(6,0)) = 27$ (originally \cite[Theorem~2.1]{DonagiSmith_StructurePrym}). \medskip
\item $\deg(\Prym_2(3,4)) = 72$ (originally \cite[Theorem~5.11]{BardelliCilibertoVerra}, \cite[Theorem~9.14]{NagarajRamanan}). \medskip
\item $\deg(\Prym_2(1,6)) = 720$ (originally \cite[Theorem~1.1]{MarcucciNaranjo}). \medskip
\item $\deg(\Prym_2(0,4)) = 6$.\medskip
\item $\deg(\Prym_2(0,6)) = 720$.\smallskip
\end{itemize}
Similarly, let $\Prym_3(2,0)$ denote the analogous map for non-cyclic \'etale triple covers with genus $2$ base curve. Then:
\begin{itemize}
\item $\deg(\Prym_3(2,0)) = 10$ (originally \cite[Theorem~5.1]{LangeOrtegaTriple}).
\end{itemize}
\end{customthm}

It follows in particular that the Prym map is dominant in all of the above cases. Dominance is usually established in advance of, and separately to, the degree calculation, but with the intersection-theoretic approach the two arise in tandem.

\subsubsection{Non-divisibility of the Jacobian class by the Prym class} Working now in the non-proper moduli space $\Acal_{g-1}$, consider the locus of Jacobians associated with curves of genus $g-1$, and the locus of Prym varieties associated with \'etale double covers with base curve of genus $g$. Replacing these two loci by their closures, there is an inclusion (see \cite{BeauvillePrym})
\[ \Jac_{g-1} \subseteq \Prym_{g} \]
inside $\Acal_{g-1}$. We prove:

\begin{customthm}{C}[\zcref{prop: Jacobian Prym not divisible}] \label{thm: non divisible introduction} For $g \in \{7,8,9,10\}$, we have
\[ [\Prym_{g}] \nmid [\Jac_{g-1}] \]
inside $\CH^\star(\Acal_{g-1})$. That is, there is no $\gamma \in \CH^\star(\Acal_{g-1})$ such that
\[ [\Prym_{g}] \cdot \gamma = [\Jac_{g-1}]. \]
In particular, there is no substack $V \subseteq \Acal_{g-1}$ which intersects $\Prym_{g}$ transversely and is such that
\[ \Prym_{g} \cap V = \Jac_{g-1}. \]
\end{customthm}

\begin{remark}
Divisibility fails trivially for $g \leqslant 5$ (the Prym class is zero) and holds trivially for $g=6$ (the Prym class is $27 \cdot \mathbbm{1}$). 
\end{remark}

To prove Theorem~\ref{thm: non divisible introduction}, we first demonstrate non-divisibility at the level of tautological projections by computing both sides. We then appeal to Borel's homological stability (specifically, Ta\"ibi's improved bound) and the properties of the tautological projection to deduce the result in cohomology, hence a fortiori in Chow.

\subsection{Calculation scheme} Returning to Algorithm~\ref{algorithm introduction}, we outline the necessary steps for computing the tautological projection of the Prym class. These culminate in the comparison formula of Theorem~\ref{thm: comparison introduction}, on which the algorithm hinges.

\subsubsection{\'Etale double covers} \label{sec: double covers introduction} We begin with the classical case of \'etale double covers. Let $\Rcal_g$ denote the Hurwitz space of \'etale double covers $f \colon C \to B$ where $g_B=g$. The associated Prym variety $\Prym(f)$ has dimension $g_C-g_B = g-1$ and there is a morphism
\begin{align} 
\label{eqn: Prym map introduction} \Prym \colon \Rcal_g & \to \Acal_{g-1} \\
\nonumber (f \colon C \to B) & \mapsto \Prym(f)	
\end{align}
known as the \textbf{Prym map}. 
The moduli space of admissible covers provides a normal crossings compactification of the Hurwitz space:
\[ \Rcalbar_g \supseteq \Rcal_g .\]
The Prym map does not extend to any of the standard compactifications of $\Acal_{g-1}$ \cite{ABH_PrymDegenerations,CMGHL_Prym,FriedmanSmith_Prym}, but for any choice of $\Sigma$ indexing a toroidal compactification, there is a diagram
\begin{equation} \label{eqn: compactified Prym map introduction}
\begin{tikzcd}
\Rcalbar_g^\Sigma \ar[d,"b"] \ar[r,"\Prym"] & \Acalbar_{g-1}^\Sigma \\
\Rcalbar_g	
\end{tikzcd}
\end{equation}
where $b$ is a suitable toroidal modification, and $\Prym$ extends the Prym map \eqref{eqn: Prym map introduction} (see \cite{Zakharov_ResolutionPrymGenus4} for the minimal modification in genus 4). 
The \textbf{Prym class} is defined as
\[ \Prym_\star [\Rcalbar_g^\Sigma] \in CH^\star (\Acalbar_{g-1}^\Sigma), \]
and our goal is to compute its tautological projection:
\[ \taut \Prym_\star [\Rcalbar_g^\Sigma] \in R^\star (\Acalbar_{g-1}). \]
We note that this does not depend on $\Sigma$. By the projection formula, it is equivalent to compute the integrals
\begin{equation} \label{eqn: first integral introduction} \int_{\Rcalbar_g^\Sigma} F(\Prym^\star \lambda_1,\ldots,\Prym^\star \lambda_{g-1}) \end{equation}
for arbitrary monomials $F$  (in fact, we may assume squarefree). To describe the pullbacks $\Prym^\star \lambda_i$, we appeal to the universal admissible cover
\[
\begin{tikzcd}
C \ar[rr,"f"] \ar[rd] & & B \ar[ld]	\\
& \Rcalbar_g
\end{tikzcd},
\]
from which we obtain three universal semiabelian varieties: $\Jac(C)$, $\Jac(B)$, $\Prym(f)$. To each is associated a Hodge bundle: $\EE_C$, $\EE_B$, $\Etilde$ with corresponding Chern classes $\lambda_i^C$, $\lambda_i^B$, $\lambdatilde_i$. With reference to \eqref{eqn: compactified Prym map introduction} we then have
\[ 
\Prym^\star \lambda_i = b^\star \lambdatilde_i,
\]
 so the integral \eqref{eqn: first integral introduction} becomes:
\begin{equation} \label{eqn: second integral introduction} \int_{\Rcalbar_g} F(\lambdatilde_1,\ldots,\lambdatilde_{g-1}). \end{equation}
We thus dispense with the auxiliary blowup $\Rcalbar_g^\Sigma$ and work directly on the moduli space $\Rcalbar_g$. The problem remains to compute the \textbf{Prym Hodge integral} \eqref{eqn: second integral introduction}. There is an identity in K-theory
\[ \Etilde = \EE_C - \EE_B \]
and the second term $\EE_B$ is pulled back along the finite target map $\Rcalbar_g \to \Mcalbar_g$. Our comparison formula (Theorem~\ref{thm: comparison introduction}) expresses the first term $\EE_C$ in terms of $\EE_B$ and correction terms. The algorithm then proceeds by applying the projection formula to the target map, recursively reducing the calculation to Hodge integrals on $\Mcalbar_g$. The combinatorial complexity is reduced by exploiting the recursive property of the Prym Hodge bundle (\zcref{prop: recursive structure of Prym Hodge}).

\subsubsection{Arbitrary covers} \label{sec: introduction higher degree covers} Passing to the general case, we now consider covers of arbitrary degree, number of branch points, and ramification profiles. After normalisation, the space of admissible covers coincides with the space of orbifold stable maps to the classifying stack, namely,
\begin{equation} \label{eqn: all maps to BSd introduction} \Mcalbar_{g,n}(\Bcal S_d).\end{equation}
Given an element $(f \colon C \to B)$ of this moduli space, the associated Prym variety $\Prym(f)$ is typically not principally polarised. However this is mostly immaterial: the moduli space
\begin{equation} \label{eqn: polarised moduli space introduction} \Acal_{\tilde{g},\delta} \end{equation}
of polarised abelian varieties of fixed dimension $\tilde{g}$ and polarisation type $\delta=(d_1,\ldots,d_{\tilde{g}})$ satisfies all the key properties outlined in \zcref{sec: tautological ring introduction}, in particular the existence of a tautological projection (see \zcref{sec: moduli of polarised abelian varieties} for details).

There is, however, a complication. Fixing the degree, number of branch points and ramification profiles is not enough to determine the polarisation type $\delta$ of the Prym variety, which is sensitive to more subtle monodromy data associated to the cover. For example, connected \'etale triple covers have different polarisation types depending on whether the monodromy group is cyclic or non-cyclic. This poses a problem, as we wish to study the Prym locus inside a single moduli space \eqref{eqn: polarised moduli space introduction} of abelian varieties.

The solution is to take a more granular decomposition of the moduli space of covers. The correct notion is that of a \textbf{global type} $\Theta$ (\zcref{def: decomposition by type}), which  contains the data of the ramification profiles over the branch points, but also  more refined monodromy data such as the global monodromy group. The moduli space \eqref{eqn: all maps to BSd introduction} admits a decomposition
\[ \Mcalbar_{g,n}(\Bcal S_d) = \coprod_{\Theta} \Mcalbar_{g,n}^{\Theta}(\Bcal S_d) \]
into clopen substacks indexed by global types $\Theta$. The polarisation type $\delta$ of the Prym variety is now constant on each piece of this decomposition, at least in all cases of interest to us; see \zcref{sec: uniform polarisation type} and \zcref{prop: cases with uniform polarisation type} for details.

The need to decompose by global types requires us to rework much of the theory of orbifold stable maps to $\Bcal S_d$. This is carried out in \zcref{sec: covers}. In particular, we emphasise the role of the orbifold fundamental group of a balanced twisted nodal curve (\zcref{def: orbifold pi1}), and develop a splitting formalism (\zcref{prop: splitting}) which is used in Algorithm~\ref{algorithm introduction}.

Fixing a global type $\Theta$ we have the following diagram, generalising\footnote{Fixing a global type already occurs implicitly when $d=2$: we exclude disconnected covers, since these lead to Prym varieties of dimension $g$ rather than $g-1$.} \eqref{eqn: compactified Prym map introduction}:
\[
\begin{tikzcd}
\Mcalbar_{g,n}^{\Theta}(\Bcal S_d)^\Sigma \ar[r,"\Prym"] \ar[d,"b"] & \Acalbar_{\tilde{g},\delta}^\Sigma \\
\Mcalbar_{g,n}^{\Theta}(\Bcal S_d).
\end{tikzcd}
\]
The discussion of \zcref{sec: double covers introduction} then applies verbatim, with the problem reduced to computing the Prym Hodge integrals
\begin{equation} \label{eqn: Prym Hodge integral introduction} \int_{\Mcalbar_{g,n}^{\Theta}(\Bcal S_d)} F(\lambdatilde_1,\ldots,\lambdatilde_{\tilde{g}}).\end{equation}
As before, the goal is to express $\Etilde$ recursively in terms of $\EE_B$. Once this is achieved, the above integral is computed recursively using the projection formula for the finite target map
\[ t \colon \Mcalbar_{g,n}^{\Theta}(\Bcal S_d) \to \Mcalbar_{g,n}\]
whose degree is an appropriate Hurwitz number (and which admits a character formula, see \zcref{prop: charformula}).

\subsection{Comparison formula}

\subsubsection{Statement} We now state the formula comparing $\Etilde$ and $\EE_B=t^\star \EE$, which allows us to compute the Prym Hodge integrals \eqref{eqn: Prym Hodge integral introduction}. Since this equivalently compares $\EE_C$ and $t^\star \EE$ (\zcref{thm: comparison}), it also allows us to compute source Hodge integrals.

\begin{customthm}{D}[\zcref{thm: Etilde in terms of EB}] \label{thm: comparison introduction}
Fix a global type $\Theta$. Then
\begin{align*}
	\ch^\vee(\Etilde) = \ & (d-1) \cdot t^\star \ch^\vee(\EE) + b_0(C) - d + \dfrac{1}{2} \sum_{i=1}^n (d-l_i) \ + \\
	& - \sum_{i=1}^n \bigg( \sum_{k \geqslant 1} \dfrac{B_{2k}}{(2k)!} \sum_{j=1}^{l_i} \dfrac{m_{ij}^{2k}-1}{m_{ij}^{2k-1}} \bigg) t^\star \psi_i^{2k-1} + \\	
	& \sum_{(\Gamma,\Theta_{V(\Gamma)})} (j_{(\Gamma,\Theta_{V(\Gamma)})})_\star \left( \sum_{k \geqslant 1} \bigg( \dfrac{B_{2k}}{(2k)!} \cdot \lcm(m_e) \cdot \sum_{i=1}^{l_e}  \dfrac{m_{ei}^{2k}-1}{m_{ei}^{2k-1}} \bigg) \cdot (t^\star \psi_{\vec{e}}^{2k-2} - t^\star ( \psi_{\vec{e}}^{2k-3} \psi_{\cev{e}}) + \cdots + t^\star \psi_{\cev{e}}^{2k-2} ) \right)
\end{align*}
as Chow classes on $\Mcalbar^{\Theta}_{g,n}(\Bcal S_d)$.
\end{customthm}

The notation is as follows (see \zcref{sec: statement of the comparison} for more details). The map $t$ is the target map
\[ t \colon \Mcalbar_{g,n}^{\Theta}(\Bcal S_d) \to \Mcalbar_{g,n}.\]
The second line is a sum over psi classes $t^\star \psi_i$ supported at branch points of the base curve, with associated ramification profile
\[ m_i = (m_{i1},\ldots,m_{il_i}) \vdash d \]
of length $l_i$. The third line is a sum over boundary divisors in the moduli space, with $j_{(\Gamma,\Theta_{V(\Gamma)})}$ denoting the inclusion of the boundary divisor and 
\[ m_e = (m_{e1},\ldots,m_{el_e}) \vdash d \]
denoting the ramification profile over the node in the base curve.

\subsubsection{Proof strategy} Theorem~\ref{thm: comparison introduction} has several cousins, which appear in the literature in various guises and levels of generality (see \zcref{remark:comparison to existing formulae}). Our proof follows similar lines: we apply Grothendieck--Riemann--Roch to the universal curves $C$ and $B$ and compare the resulting formulae. One innovation is our use of the language of tropical curves and piecewise polynomials. More than an expositional preference, it allows us to recognise the invariance under subdivisions of a key Todd class formula (\zcref{prop: Todd is res}). This means that it is not necessary to pass to a crepant resolution of the universal curve, considerably simplifying the proof.

\subsubsection{Algorithm} Algorithm~\ref{algorithm introduction} now proceeds recursively, using Theorem~\ref{thm: comparison introduction}. At each step, a Prym lambda class $\lambdatilde_i$ is traded for a target lambda class $\lambda_i^B = t^\star \lambda_i$ together with psi and boundary terms. The base of the recursion occurs when there are no Prym lambda classes remaining, in which case the projection formula along $t$ reduces the calculation to an integral of lambda and psi classes on the moduli space of curves.

While the algorithm works in principle, it is computationally expensive due to the appearance of  products of boundary strata. Computer implementations struggle even in the simplest case of \'etale double covers. To obtain an effective algorithm, one final trick is required: utilising the invariance of the Prym Hodge bundle under pullbacks to boundary strata (\zcref{prop: recursive structure of Prym Hodge}). This allows us to never have to compute products of boundary strata directly, leading to a significantly more effective algorithm in practice. See \zcref{appendix} for tables of calculations.

\subsection{Prospects} \label{sec: introduction future directions}
The large number of parameters used to define the Prym classes make them a particularly rich and varied collection. We outline several directions for future study.

\subsubsection{General formulae} While our algorithm effectively computes the tautological projection for many examples, we lack a general closed form. In fact, general formulae for Hodge integrals are notoriously elusive. Nevertheless, it is reasonable to hope for general formulae at least for certain coefficients, see e.g. \cite[Conjectures~1~and~2]{Faber_algorithms} and \cite{JPT}.

\subsubsection{Tautological intersection theory of the moduli space of admissible covers} The moduli space of admissible covers admits a tautological ring \cite{LianHTaut} about which very little is known. We expect the Prym map to play an important role here, similar to the role played by the Torelli map for the moduli space of curves \cite{CanningLarsonSchmitt}.

\subsubsection{Degrees of Prym maps over loci with additional automorphisms} \label{sec: future directions additional auts introduction} Extending the calculations in \zcref{sec: degree Prym introduction}, we can consider situations in which the Prym class has positive codimension but the Prym map is still finite onto a strict sublocus of the moduli space of abelian varieties, typically a locus of abelian varieties admitting additional automorphisms. We write this as:
\[ \Prym \colon \Mcal_{g,n}^{\Theta}(\Bcal S_d) \to \Bcal_{\tilde{g},\delta} \subseteq \Acal_{\tilde{g},\delta} \]
We can then compute the degree (onto its image) of the Prym map, by comparing the tautological projections of
\[ \Prym_\star \left[\Mcalbar_{g,n}^{\Theta}(\Bcal S_d)^\Sigma \right] \qquad \text{and} \qquad [\Bcal_{\tilde{g},\delta}]. \]
In fact, it is sufficient to compare a single coefficient. In certain cases it is possible to identify $\Bcal_{\tilde{g},\delta}$ with a Noether--Lefschetz locus (see e.g. \cite{LangeOrtegaSeven}), in which case its tautological projection is known by \cite{IribarLopez_NoetherLefschetz}.

\subsection{Outline of the paper} The paper proceeds as follows:
\begin{itemize}
\item \textbf{\zcref{sec: covers}.} We rework the theory of admissible covers in the form we require, decomposing the moduli space by global types (\zcref{def: decomposition by type}) and establishing a splitting formalism in this context (\zcref{prop: splitting}). 
\item \textbf{\zcref{sec: abelian varieties}.} We recall the theory of moduli spaces of abelian varieties with fixed polarisation type. We describe the associated tautological ring (\zcref{thm: tautological ring presentation}) and tautological projection (\zcref{sec: taut projection}), and explain the role the polarisation type plays in the proportionality constant (\zcref{sec: proportionality constant}).
\item \textbf{\zcref{sec: Prym}.} We recall Prym varieties associated to covers with arbitrary degree and monodromy data (\zcref{sec: Prym varieties}). We then discuss in detail their induced polarisations (\zcref{sec: Prym polarisation}). Finally, after imposing the condition that the polarisation type is uniform (\zcref{sec: uniform polarisation type}), we introduce the Prym map, and discuss its extensions to the compactified moduli spaces (\zcref{sec: Prym map construction}).
\item \textbf{\zcref{sec: tautological}.} We introduce and compare the three different Hodge bundles on the moduli space of covers (\zcref{sec: Hodge bundles on moduli space}), and explain how computing the tautological projection reduces to computing Prym Hodge integrals (\zcref{problem: second formulation}).
\item \textbf{\zcref{sec: comparison}.} We state and prove the key comparison formula (\zcref{thm: Etilde in terms of EB}). The proof uses Grothendieck--Riemann--Roch and hinges on a tropical description of the Todd class of the residue sheaf (\zcref{prop: Todd is res}).
\item \textbf{\zcref{sec: algorithm}.} We establish the recursive property of the Prym Hodge bundle, namely its compatibility with the splitting formalism (\zcref{prop: recursive structure of Prym Hodge}). Using this, we describe the recursive algorithm for computing the tautological projection of the Prym class (\zcref{algorithm main}).
\item \textbf{\zcref{sec: applications}.} We present two applications. First, we compute the degrees of the Prym maps in the codimension-zero cases (\zcref{sec: degree of Prym}). This includes several hand-calculations in degree~$3$ (Sections~\ref{sec: degree 3 covers non-cyclic}~and~\ref{sec: degree 3 covers cyclic}). Second, we establish the non-divisibility of the Jacobian class by the Prym class (\zcref{prop: Jacobian Prym not divisible}).
\item \textbf{\zcref{appendix}.} We collect tables of tautological projections computed via our algorithm. This includes Prym classes with fixed numbers of branch points (\zcref{appendix Prym}) and hyperelliptic Jacobian classes (\zcref{appendix hyperelliptic Jacobians}). For reference, we also include tables of Torelli classes (\zcref{appendix Jacobians}).
\end{itemize}

\subsection*{Acknowledgements} We thank Thomas~Blomme, Andrei~Bud, Sam Canning, Renzo~Cavalieri, Alessandro~Chiodo, Gabi~Farkas, Sam Grushevsky, Gabriele~Mondello, Rahul~Pandharipande, and Riccardo~Salvati~Manni for helpful conversations. We especially thank Aitor~Iribar~Lopez for numerous invaluable discussions and detailed comments on a preliminary draft. The second author wishes to thank Alex~Abreu and Nicola~Pagani for numerous conversations on the Grothendieck--Riemann--Roch theorem for curves over the years.

We thank Johannes~Schmitt for invaluable technical assistance with \texttt{admcycles} \cite{admcycles}, used throughout the course of this work to perform toy calculations and consistency checks. The final implementation of Algorithm~\ref{algorithm introduction} uses for its base cases the \texttt{admcycles} implementation of Faber's algorithm, authored primarily by Zheming~Sun.

We have also used Carel~Faber's original \texttt{Maple} implementation of his algorithm \cite{Faber_algorithms}, and we thank him for providing us with a table of certain $\Mcalbar_g$ integrals which, due to our inability to use his code properly, we could not produce ourselves.

\subsection*{Institutional support and funding} 
Parts of this work were carried out at the Universities of Durham, Cambridge, St Andrews, and Rome (Sapienza). We thank them for excellent working conditions. YL was supported by an EPSRC New Investigator Award (grant number EP/X002004/1). 

\subsection*{Conventions and notation} We work over $\kfield = \mathbb{C}$. Given a positive integer $n \geqslant 1$ we write $[n]$ for the finite set $\{1,\ldots,n\}$. We collect the most commonly-used notation in the following table.

\renewcommand{\arraystretch}{1.5} % increase row height locally
\begin{longtable}{| >{\centering\arraybackslash}p{3cm} | p{8cm} | >{\centering\arraybackslash}p{3cm} |}
% Heading for first page
%\caption{My long table}
%#\label{tab:my-long-table}
%\\
\hline
\cellcolor{gray!20} \textbf{Notation} & \cellcolor{gray!20} \centering{\textbf{Description}} & \cellcolor{gray!20} \textbf{Location} \\
\thickhline
\endfirsthead

% Heading for subsequent pages
\multicolumn{3}{l}{Continued from previous page} \\
\hline
\cellcolor{gray!20} \textbf{Notation} & \cellcolor{gray!20} \centering{\textbf{Description}} & \cellcolor{gray!20} \textbf{Location} \\
\thickhline
\endhead

% Footer for non-final pages
\hline
\multicolumn{3}{r}{Continued on next page} \\
\endfoot

% Footer for final page
\hline
\endlastfoot
$\Bcal$ & balanced twisted curve & \zcref{sec: Sd covers} \\ \hline
$P \to \Bcal$ & $S_d$-cover & \zcref{sec: Sd covers} \\ \hline
$\theta$ & global image list & \zcref{sec: setup}\\ \hline
$\Theta$ & global type & \zcref{sec: setup} \\ \hline
$\Mcalbar_{g,n}^{\Theta}(\Bcal S_d)$ & moduli space of $S_d$-covers of global type $\Theta$ & \zcref{def: decomposition by type} \\ \hline
$\Gamma$ & stable graph & \zcref{def: stable graph} \\ \hline
$\Mcalbar_\Gamma$ & stratum in the space of curves & \zcref{def: stable graph} \\ \hline
$(\Gamma,r)$ & twisted stable graph & \zcref{def: twisted stable graph} \\ \hline
$\Mcalbar^{\tw}_{(\Gamma,r)}$ & stratum in the space of twisted curves & \zcref{def: twisted stable graph} \\ \hline
$\pi_1^{\orb}(\Gamma,r)$ & orbifold fundamental group & \zcref{def: orbifold pi1} \\ \hline
$\rho$ & monodromy representation $\pi_1^{\orb}(\Gamma,r) \to S_d$ & \zcref{prop: G cover same as monodromy rep of pi1orb} \\ \hline
$\Theta_v$ & vertex type & \zcref{def: vertex image list} \\ \hline
$ \Theta_{V(\Gamma)}$ & vertexwise type & \zcref{def: vertexwise type} \\ \hline
$(\Gamma,\Theta_{V(\Gamma)})$ & vertexwise $S_d$-graph & \zcref{def: vertexwise Sd graph} \\ \hline
$\Mcalbar_{(\Gamma,\Theta_{V(\Gamma)})}^{\Theta}(\Bcal S_d)$ & vertexwise boundary stratum & \zcref{def: boundary strata vertex decorated graph 2} \\ \hline
$m_e$ & ramification profile associated to $e \in E(\Gamma)$ & \zcref{sec: ramification profile boundary stratum} \\ \hline
$f \colon C \to B$ & degree-$d$ cover associated to $S_d$-cover & \zcref{sec: G-cover to d-cover} \\ \hline
$\Prym(f)$ & Prym variety & \zcref{sec: Prym varieties} \\ \hline
$\tilde{g}$ & dimension of $\Prym(f)$ & \zcref{eqn: dim of Prym} \\ \hline
$\delta$ & polarisation type of Prym variety & \zcref{sec: Prym polarisation}\\ \hline
$\Acal_{\tilde{g},\delta}$ & moduli space of abelian varieties & \zcref{sec: open moduli abelian varieties} \\ \hline
$\Sigma$ & admissible decomposition & \zcref{sec: compact moduli abelian varieties} \\ \hline
$\Acalbar_{\tilde{g},\delta}^\Sigma$  & compactified space of abelian varieties & \zcref{sec: compact moduli abelian varieties} \\  \hline
$\Prym$ & Prym map $\Mcalbar_{g,n}^{\Theta}(\Bcal S_d)^\Sigma \to \Acalbar_{\tilde{g},\delta}^{\Sigma}$ & \zcref{sec: Prym map construction} \\ \hline
$\Ucal_{g,\delta}^\Sigma$ & universal semiabelian variety & \zcref{eqn: universal semiabelian variety} \\ \hline
$\EE$ & Hodge bundle on $\Acalbar_{\tilde{g},\delta}^\Sigma$ & \zcref{sec: Hodge bundle} \\
 \hline
$R^\star(\Acalbar_{\tilde{g},\delta})$ & tautological ring & \zcref{def: taut ring} \\ \hline
$\taut$ & Tautological projection $\CH^\star(\Acalbar_{\tilde{g},\delta}^\Sigma) \to R^\star(\Acalbar_{\tilde{g},\delta})$ & \zcref{sec: taut projection} \\ \hline
$\Etilde$ & Prym Hodge bundle on $\Mcalbar_{g,n}^{\Theta}(\Bcal S_d)$ & \zcref{sec: Hodge bundles on moduli space} \\ \hline
$\lambdatilde_i$ & Chern class of Prym Hodge bundle & \zcref{sec: Hodge bundles on moduli space} \\ \hline
$\chitilde_i$ & Chern character of Prym Hodge bundle & \zcref{sec: Hodge bundles on moduli space} \\ \hline
$p \colon C \to S$ & family of logarithmic curves & \zcref{sec: tropical Todd} \\ \hline
$R_p$ & residue sheaf & \zcref{eqn: exact sequence dualising and residue} \\ \hline
$\Sigma_p \colon \Sigma_C \to \Sigma_S$ & family of tropical curves & \zcref{construction: residue series} \\ \hline
$\res$ & residue series & \zcref{construction: residue series} \\ \hline
$\Phi$ & operator $\widehat{\PP}^\star(\Sigma_X) \to CH^\star(X)$ & \zcref{sec: piecewise} \\ \hline
$\Rcalbar_{g,n}$ & moduli space of admissible double covers & \zcref{sec: comparison degree 2} \\ \hline
\end{longtable}

\section{Covers} \label{sec: global G-covers} \label{sec: covers}

\noindent Our goal is to study moduli of Prym varieties for covers of arbitrary degree, ramification, and monodromy structure. 
The theory of orbifold stable maps to the classifying stack $\Bcal S_{d}$ furnishes compact moduli of such covers \cite{AbramovichVistoli, ACV, AbramovichGraberVistoli,AbramovichLectures}. %We outline this theory in the form we require.
While this section primarily serves as a recap, we provide a detailed treatment of many aspects (global types, orbifold fundamental groups, splitting formalisms) which to the best of our knowledge do not appear in the literature. The key definitions and results are:
\begin{itemize}
\item Global types (\zcref{def: global type smooth cover}) and decompositions according to type (\zcref{def: decomposition by type}).
\item Orbifold fundamental groups (\zcref{def: orbifold pi1}) and their role in $S_d$-covers (\zcref{prop: G cover same as monodromy rep of pi1orb}).
\item Vertexwise boundary strata (\zcref{def: boundary strata vertex decorated graph 2}).
\item Forgetful degrees for boundary strata (\zcref{lem: degree target map vertexwise stratum}).
\item Splitting formalism (\zcref{prop: splitting}).
\end{itemize}
A summary is provided at the end of the section (\zcref{sec: covers summary}).

\subsection{Setup} \label{sec: setup}
The following data will be fixed throughout the paper. We choose a degree $d \geqslant 1$, a base genus $g \geqslant 0$, and a number $n \geqslant 0$ of branch points, assuming always the stability condition $2g-2+n >0$. We choose also a subgroup
\[ G \leqslant S_d \]
and $G$-conjugacy classes associated to the branch points:
\[ [g_1],\ldots,[g_n] \subseteq G. \]
We write the above group-theoretic data as
\[ \theta \colonequals (G,[g_1],\ldots,[g_n]) \]
and refer to it as a \textbf{global image list}. There is an action of $S_d$ on the set of global image lists by simultaneous conjugation, and we write $\Theta$ for the equivalence class of $\theta$. We refer to $\Theta$ as a \textbf{global type}.

\begin{remark}
The classical setup for Prym varieties is $d=2$ and $n \geqslant 0$ even, with $G=S_2$ and $[g_i]=[(12)]$ the non-identity conjugacy class for all $i \in [n]$. The majority of our explicit calculations will be in this setting (see \zcref{sec: codim zero degree 2} and \zcref{appendix}). However, see \zcref{sec: degree 3 covers non-cyclic} for a calculation with $d=3$ and $G=S_3 \leqslant S_3$, and \zcref{sec: degree 3 covers cyclic} for a calculation with $d=3$ and $G=A_3 \leqslant S_3$.
\end{remark}

\subsection{$S_d$-covers} \label{sec: Sd covers}

\subsubsection{Covers with smooth base} \label{sec: G-covers} An \textbf{$S_d$-cover with smooth base} is a representable morphism
\[ \Bcal \to \Bcal S_d \]
where $\Bcal$ is a smooth twisted curve of genus $g$. This means that $\Bcal$ is obtained as a root stack \cite[Definition~2.2.1]{CadmanStacks} of a marked genus-$g$ Riemann surface $(B,b_1,\ldots,b_n)$ along the markings $b_1,\ldots,b_n$. The morphism $\Bcal \to \Bcal S_d$ is equivalent to the data of a principal $S_d$-bundle:
\[ P \to \Bcal. \]
The moduli space of $S_d$-covers with smooth base is denoted:
\[ \Mcal_{g,n}(\Bcal S_d).\]
It is a smooth, non-proper Deligne--Mumford stack of dimension $3g-3+n$. It depends only on $d,g,n$ and not yet on the group-theoretic data $\Theta$. In particular, no conditions are currently imposed on the rooting indices or the conjugacy classes at the markings.

\subsubsection{Covers with nodal base}
The moduli space $\Mcal_{g,n}(\Bcal S_d)$ admits a normal crossings compactification
\begin{equation} \label{eqn: space admissible covers} \Mcalbar_{g,n}(\Bcal S_d) \supseteq \Mcal_{g,n}(\Bcal S_d) \end{equation}
given by the moduli space of \textbf{admissible $S_d$-covers} (or \textbf{$S_d$-covers} for short). This is defined as the moduli space of twisted stable maps \cite{ACV,AbramovichGraberVistoli} to the orbifold $\Bcal S_d$, parametrising representable morphisms
\[ \Bcal \to \Bcal S_d \]
where $\Bcal$ is a balanced twisted nodal curve in the sense of \cite[Section~4]{AbramovichVistoli}. %This compactification refines to give compactifications of the moduli spaces of covers of type $\Theta$: we will describe these after introducing the language of orbifold fundamental groups.

\subsubsection{From $S_d$-covers to $d$-covers} \label{sec: G-cover to d-cover} Consider the faithful action $S_d \curvearrowright [d] \colonequals \{1,\ldots,d\}$. The \textbf{mixing construction} associates to a principal $S_d$-bundle $P \to \Bcal$ an \'etale cover $C \to \Bcal$ of degree $d$, given by the quotient
\[ C \colonequals [(P \times [d])/S_d] \to [P/S_d] = \Bcal \]
where $S_d$ acts anti-diagonally on $P \times [d]$, so that $(gp,i) \sim (p,gi)$. Since the morphism $\Bcal \to \Bcal S_d$ is representable it follows that $C$ is a scheme. The induced morphism to the coarse space of $\Bcal$,
\[ C \to B \]
is a degree~$d$ cover, \'etale away from the markings and nodes of $B$. We caution that this cover may not be Galois, meaning the deck group may not act transitively on the fibres. We write the universal $d$-cover arising from the above construction as:
\begin{equation} \label{eqn: universal d cover}
\begin{tikzcd}
    C \ar[rr,"f"] \ar[rd,"p_C"{xshift=-10pt,yshift=-12pt}] & & B \ar[ld,"p_B"] \\
    & \Mcalbar_{g,n}(\Bcal S_d).
\end{tikzcd}
\end{equation}
We emphasise that $B$ is the (relative) coarse space of $\Bcal$, so both $p_B$ and $p_C$ are representable.

Up to normalisation, $\Mcalbar_{g,n}(\Bcal S_d)$ is the Harris--Mumford moduli space \cite{BeauvilleAdmissible,HarrisMumford} of admissible $d$-covers \cite[Proposition~4.2.2]{ACV}. We thus often abuse notation and write $(f \colon C \to B)$ instead of $(P \to \Bcal)$ for a point in this moduli space.

\subsection{Decomposition by type}

The moduli space of $S_d$-covers introduced in the previous section decomposes into pieces indexed by the group-theoretic data of the global type $\Theta$  (\zcref{sec: setup}).

\subsubsection{Decomposition for open moduli} We begin with $S_d$-covers $P \to \Bcal$ with smooth base. Given the pointed coarse curve $(B,b_1,\ldots,b_n)$, the data of the cover $P \to \Bcal$ is equivalent to the data of a monodromy representation
\[ \rho \colon \pi_1(B \setminus \{b_1,\ldots,b_n\}) \to S_d, \]
while isomorphisms between $S_d$-covers are given by the action of $S_d$ on the set of monodromy representations by post-conjugation.

\begin{definition}[Global image list associated to a monodromy representation] \label{def: global image list} Fix a monodromy representation $\rho$. We define associated images
\[ G  \colonequals \rho(\pi_1(B \setminus \{b_1,\ldots,b_n\})) \leqslant S_d \]
and
\[ [g_j]  \colonequals [\rho(\gamma_j)] \subseteq G \]
where $\gamma_j$ denotes the class of a small loop around the puncture $b_j$. While $\gamma_j$ is only well-defined up to conjugation in $\pi_1(B \setminus \{b_1,\ldots,b_n\})$, the $G$-conjugacy class $[\rho(\gamma_j)]$ does not depend on this choice. We define the \textbf{global image list associated to $\rho$} as the data
\[ \theta_\rho \colonequals (G,[g_1],\ldots,[g_n]).\]
\end{definition}

\begin{definition}[Global type associated to a cover] \label{def: global type smooth cover}
Fix an $S_d$-cover $P \to \Bcal$. The associated monodromy representation $\rho$ is well-defined up to post-conjugation, and therefore the global image list $\theta_\rho$ is well-defined up to the action of $S_d$ by simultaneous conjugation. We write $\Theta_{P \to \Bcal}$ for the equivalence class and refer to it as the \textbf{global type} of the $S_d$-cover.
\end{definition}

\begin{definition}[Decomposition by types, open moduli] Given a fixed global type $\Theta$, let
\begin{equation} \label{eqn: stack covers of type theta smooth} \Mcal_{g,n}^{\Theta}(\Bcal S_d) \subseteq \Mcal_{g,n}(\Bcal S_d) \end{equation}
denote the moduli space of covers such that $\Theta_{P \to \Bcal} = \Theta$. Note in particular that over this locus, the rooting index at each marking is fixed at $r_i \colonequals \ord(g_i)$. 

Since the global type $\Theta_{P \to \Bcal}$ is locally-constant over the moduli space, the inclusion \eqref{eqn: stack covers of type theta smooth} is a clopen substack. We thus obtain a decomposition
\[ \Mcal_{g,n}(\Bcal S_d) = \coprod_{\Theta} \Mcal_{g,n}^{\Theta}(\Bcal S_d) \]
into clopen substacks indexed by global types.
\end{definition}

\subsubsection{Decomposition for compact moduli} The above decomposition extends to compact moduli by taking closures.

\begin{definition}[Decomposition by types, compact moduli] \label{def: decomposition by type} Given a global type $\Theta$ we define
\begin{equation} \label{eqn: stack covers of type theta general} \Mcalbar_{g,n}^{\Theta}(\Bcal S_d) \subseteq \Mcalbar_{g,n}(\Bcal S_d) \end{equation}
as the closure of $\Mcal_{g,n}^{\Theta}(\Bcal S_d)$ inside $\Mcalbar_{g,n}(\Bcal S_d)$. Since \eqref{eqn: stack covers of type theta smooth} is clopen and \eqref{eqn: space admissible covers} is normal crossings, it follows that \eqref{eqn: stack covers of type theta general} is also clopen, so we have a decomposition:
\begin{equation} \label{eqn: decomposition according to type} \Mcalbar_{g,n}(\Bcal S_d) = \coprod_{\Theta} \Mcalbar_{g,n}^{\Theta}(\Bcal S_d). \end{equation}
Moreover, for each $\Theta$ the compactification
\[ \Mcalbar_{g,n}^{\Theta}(\Bcal S_d) \supseteq \Mcal_{g,n}^{\Theta}(\Bcal S_d) \]
is normal crossings. Once the language of orbifold fundamental groups has been developed (\zcref{sec: orbifold fundamental group}), the substack \eqref{eqn: stack covers of type theta general} can be defined without resort to a closure operation (\zcref{lem: decomposition by type without closure}).
\end{definition}

\begin{remark} The decomposition by types is crucial in order to guarantee that the polarisation type of the Prym variety is constant in the cases of interest to us: see Sections~\ref{sec: Prym polarisation}~and~\ref{sec: uniform polarisation type}.
\end{remark}

\begin{remark} \label{rmk: essential covers decompose moduli space examples} The decomposition  \eqref{eqn: decomposition according to type} unifies and generalises the common practice of restricting to connected, cyclic, or non-cyclic covers. For instance when $n=0$, the choice of a global type amounts to a choice of conjugacy class of subgroup $G \leqslant S_d$. Letting $C \to B$ denote the associated $d$-cover (via the construction in \zcref{sec: G-cover to d-cover}) we have the following important examples:
\begin{itemize}
\item For $d=2$ the valid subgroups are $G=S_2$ and $G=\{e\}$. The associated moduli spaces parametrise respectively connected and disconnected double covers $C \to B$.
\item For $d=3$ the valid subgroups are:
\begin{itemize}
	\item $G=S_3$: $C \to B$ connected and non-cyclic. Such covers are considered e.g. in \cite{LangeOrtegaTriple}.
    \item $G \leqslant S_3$ cyclic of order $3$: $C \to B$ connected triple cover with cyclic monodromy group.
    \item $G \leqslant S_3$ cyclic of order $2$: $C \to B$ triple cover with two connected components.
    \item $G \leqslant S_3$ trivial: $C \to B$ triple cover with three connected components.
\end{itemize}
\end{itemize}
\end{remark}

\begin{remark} \label{remark: decomposition not connected}
We do not claim that the components of the decomposition are connected. We simply take the decomposition that, together with the vertexwise boundary strata discussed in \zcref{sec: vertexwise strata}, admit a splitting formalism (\zcref{prop: splitting}), a character formula (Propositions~\ref{prop: charformula} and \ref{prop: boundary char formula}), and remains sufficiently fine to state the comparison result (\zcref{thm: comparison}).

The problem of determining the connected components of the boundary strata, which is equivalent to determining the tropicalisation of the moduli space, has very recently been solved \cite{GlynnHigherGenus,GlynnThesis} (see also \cite{CMR_Admissible, Glynn, BrandtChanKannan} for partial results). However for us the formalism of global types is more convenient, see \zcref{rmk: connected components} for further discussion. We note that other intermediate stratifications are also available, see e.g. \cite{BertinRomagny, SchmittVanZelm, LianHTaut}.
\end{remark}

\begin{remark} \label{rmk: compare to usual decomposition}
We recall the usual decomposition of the moduli space $\Mcalbar_{g,n}(\Bcal S_d)$ arising from the theory of orbifold stable maps, and explain how it relates to the decomposition by global types given in \eqref{eqn: decomposition according to type}. Fix $S_d$-conjugacy classes
\[ \Omega_1,\ldots,\Omega_n \subseteq S_d \]
and let $\Omega_{\ul{n}} \colonequals (\Omega_1,\ldots,\Omega_n)$. There is then a clopen substack
\begin{equation} \label{eqn: substack fixed conjugacy classes at markings} \Mcalbar_{g,\Omega_{\ul{n}}}(\Bcal S_d) \subseteq \Mcalbar_{g,n}(\Bcal S_d) \end{equation}
and associated decomposition
\[ \Mcalbar_{g,n}(\Bcal S_d) = \coprod_{\Omega_{\ul{n}}} \Mcalbar_{g,\Omega_{\ul{n}}}(\Bcal S_d) \]
running over lists $\Omega_{\ul{n}}$ of $S_d$-conjugacy classes at the markings. On the locus where $\Bcal$ is smooth, the clopen substack \eqref{eqn: substack fixed conjugacy classes at markings} parametrises covers whose associated monodromy representation satisfies $\rho(\gamma_j) \in \Omega_j$ for all $j \in [n]$ (this description extends to nodal $\Bcal$ using the orbifold fundamental group, see \zcref{sec: orbifold fundamental group}). There is no condition on the image $G$ of the monodromy representation.

Given a global type $\Theta$ with representative $\theta=(G,[g_1],\ldots,[g_n])$, the $S_d$-conjugacy class $\Omega_j \colonequals [g_j]_{S_d} \subseteq S_d$ is independent of the choice of $\theta$ and the choice of representative $g_j$ of the $G$-conjugacy class $[g_j]$. There is thus a well-defined association:
\[ \Theta \mapsto \Omega_{\ul{n}} \colonequals ([g_1]_{S_d},\ldots,[g_n]_{S_d}). \]
For a fixed list $\Omega_{\ul{n}}$ of $S_d$-conjugacy classes, we then have
\[ \Mcalbar_{g,\Omega_{\ul{n}}}(\Bcal S_d) = \coprod_{\Theta \mapsto \Omega_{\ul{n}}} \Mcalbar_{g,n}^{\Theta}(\Bcal S_d). \]
Our decomposition is thus finer than the usual one, as we have the additional choice of subgroup $G \leqslant S_d$ recording the image of the monodromy representation.

When defining the Prym class (\zcref{sec: taut projection of the Prym class}) we always restrict to the moduli space associated to a fixed global type $\Theta$. This is crucial in order to ensure that the associated polarisation type is constant in the cases of interest to us (see \zcref{sec: uniform polarisation type}). At certain points we may wish to study Prym classes arising from a coarser decomposition of the moduli space, in which case we simply sum the contributions of several global types.
\end{remark}

\subsubsection{Degree of the target map} \label{sec: degree target map global} Consider now the target map remembering the coarse base curve:
\begin{align*} t \colon \Mcalbar_{g,n}^{\Theta}(\Bcal S_d) & \to \Mcalbar_{g,n} \\
(P \to \Bcal) & \mapsto B.
\end{align*}
The degree of this map is given by an associated Hurwitz number. We let
\begin{equation} \pi_1(S_{g,n}) \end{equation}
denote the abstract group with generators $\alpha_1,\beta_1,\ldots,\alpha_g,\beta_g,\gamma_1,\ldots,\gamma_n$ and the single relation:
\[ \prod_{i=1}^g [\alpha_i,\beta_i] \cdot \prod_{j=1}^n \gamma_j = 1. \]
We then have an isomorphism
\[ \pi_1(S_{g,n}) \xrightarrow{\cong} \pi_1(B \setminus \{b_1,\ldots,b_n\})\] 
which sends $\gamma_j$ to a small loop around $b_j$. This isomorphism is not unique, but the image of each $\gamma_j$ is well-defined up to conjugation. We let
\begin{equation} \label{eqn: Theta monodromy rep} \Hom^{\Theta}(\pi_1(S_{g,n}),S_d) \end{equation}
denote the set of monodromy representations whose global image list (\zcref{def: global image list}) belongs to the equivalence class $\Theta$.

\begin{lemma} \label{lem: degree target map global} We have:
\[ \deg \big( \Mcalbar_{g,n}^{\Theta}(\Bcal S_d) \to \Mcalbar_{g,n} \big) = \dfrac{\big|\Hom^{\Theta}(\pi_1(S_{g,n}),S_d)\big|}{d!}.\]
\end{lemma}

\begin{proof} This follows from the definition of $\Mcalbar_{g,n}^\Theta(\Bcal S_d)$ and the well-known identification between Hurwitz numbers and counts of monodromy representations, see e.g. \cite[Theorem~7.3.2]{CavalieriMiles}.
\end{proof}

\subsubsection{Character formula} The degree in \zcref{lem: degree target map global} can be calculated via a character formula. This is similar in spirit to the Dijkgraaf--Witten formula (see e.g. \cite[Theorem~3]{ZagierFiniteGroups}) with the additional complexity that we specify the image of the monodromy representation.

Recall first the M\"obius function $\mu=\mu_G$ of $G$, which assigns an integer $\mu(H,K)$ to any pair of subgroups $H \leqslant K \leqslant G$, and is defined inductively via:
\[
\mu(H,H)=1 \qquad \text{and} \qquad \sum_{H \leqslant L \leqslant K} \mu(H,L) = 0.
\]
If $H \nleqslant K$ we sometimes set $\mu(H,K) = 0$ by convention. Choose now a representative $\theta_0=(G,[g_1],\ldots,[g_n])$ of $\Theta$ and define its normaliser:
\begin{align} \label{eqn: normaliser group}
N_{S_d}(\theta_0) & \colonequals \left\{ \sigma \in S_d: \sigma G \sigma^{-1} = G, \,\sigma [g_j] \sigma^{-1} = [g_j] \text{ for all } j \in [n] \right\}.
\end{align}
By the orbit-stabiliser theorem, we have $|N_{S_d}(\theta_0)| = d!/|\Theta|$.

\begin{proposition} \label{prop: charformula} Let $\Irr(H)$ denote the set of irreducible characters of a group $H$. Then 
 \[
 \deg \big( \Mcalbar_{g,n}^{\Theta}(\Bcal S_d) \to \Mcalbar_{g,n} \big) = \dfrac{1}{|N_{S_d}(\theta_0)|} \sum_{H \leqslant G} \mu(H,G)|H|^{2g-1} \sum_{\chi \in Irr(H)} \dfrac{\prod_{j=1}^n \sum_{c \in [g_j] \cap H} \chi(c)}{{\chi(1)^{2g+n-2}} }.
 \]
\end{proposition}

\begin{proof}
    Fix a base curve $(B,b_1,\ldots,b_n) \in \Mcal_{g,n}$. We must count monodromy representations $\rho \colon \pi_1(B \setminus \{b_1,\ldots,b_n\}) \to S_d$ weighted by their automorphisms. This data consists of elements
\[ a_i,b_i,c_j \in S_d \]
satisfying the commutator relation
\[
\prod_{i=1}^g [a_i,b_i] \prod_{j=1}^n c_j = 1. 
\]
Isomorphisms are given by elements of $S_d$ acting by conjugation. By hypothesis, our covers can be conjugated so that the associated global image list is $\theta_0$, meaning that
\[ a_i,b_i,c_j \in G, \quad \left \langle a_i,b_i,c_j \right \rangle_{i,j} = G,\quad c_j \in [g_j]. \]
After this reduction, isomorphisms correspond to conjugation by  elements of $S_d$ that fix $\theta_0$, or in other words,  by  elements of the normaliser $N_{S_d}(\theta_0)$. Setting
\[
\Hcal_{\theta_0} \colonequals \big\{a_i,b_i,c_j \in G : \prod_{i=1}^g [a_i,b_i] \prod_{j=1}^n c_j = 1, \langle a_i,b_i,c_j \rangle_{i,j} = G, c_j \in [g_j] \big\},
\]
we see that the fibre of $\Mcal_{g,n}^\Theta(\Bcal S_d) \to \Mcal_{g,n}$ over $(B,b_1,\ldots,b_n)$ is the groupoid $[\Hcal_{\theta_0}/N_{S_d}(\theta_0)]$ and so the degree of the map is
\[ |\Hcal_{\theta_0}|/|N_{S_d}(\theta_0)|.\]
We now calculate $|\Hcal_{\theta_0}|$. For every subgroup $H \leqslant G$, consider 
\[
S(H) \colonequals \big| \big\{ a_i,b_i,c_j \in H : \prod_{i=1}^g [a_i,b_i] \prod_{j=1}^n c_j = 1, c_j \in [g_j] \cap H \big\} \big|.
\]
In other words, $S(H)$ is  the number of elements as above which are in $H$ but which are not required to generate $H$. The number $S(H)$ can be calculated via character theory. Let $\kfield[H]$ be the group algebra of $H$ and let $\{e_g\}_{g \in H}$ denote the standard basis. Consider the elements
\[
K \colonequals \sum_{a,b \in H} e_{{[a,b]}}, \qquad C_j \colonequals \sum_{c \in [g_j] \cap H} e_c.
\] 
Since these elements are conjugation invariant, they belong to the centre of $\kfield[H]$. Therefore, by Schur's lemma, they act on every irreducible representation of $H$ by a scalar.

Let $V$ be such an irreducible representation. Taking the trace produces the associated character $\chi=\chi_V$, and the scalar by which a central element $g \in \kfield[H]$ acts on $V$ is given by 
\[
\chi(g)/\chi(1).
\] 
For $K$, the Frobenius commutator formula gives $\chi(K) = |H|^2/\chi(1)$ and so $K$ acts on $V$ by the scalar:
\begin{equation} \label{eqn: scalar action of K}
\frac{|H|^2}{\chi(1)^2}. 
\end{equation}
For $C_j$ the action on $V$ is given by the scalar:
\begin{equation} \label{eqn: scalar action of Ci}
\frac{\sum_{c \in [g_j] \cap H} \chi(c)}{\chi(1)}.
\end{equation}
We now claim that $S(H)$ is the coefficient of the identity element $1 \in \kfield[H]$ of
\[
P=K^g \prod_{j=1}^n C_j \in \kfield[H].
\]
Indeed, expanding this product and writing it in terms of the  standard basis of $\kfield[H]$ gives
\[
P = \sum \left( \prod_{i=1}^g e_{[a_i,b_i]} \prod_{j=1}^n e_{c_j} \right) = \sum_{g \in H} \eta_g e_g,
\]
where the sum is over all possible $a_i,b_i,c_j\in H$, and the coefficient $\eta_g$ is the number of $a_i,b_i,c_j$ satisfying $\prod_{i=1}^g [a_i,b_i] \prod_{j=1}^n c_j = g$. It follows that $\eta_1 = S(H)$ as claimed.

On the regular representation $\kfield[H]$, any $g \neq 1$ acts by a traceless matrix. Therefore, the trace of $P$ acting on $\kfield[H]$ is 
\begin{equation} \label{eqn: trace via group counts}
|H| \cdot S(H).
\end{equation}

On the other hand, by \eqref{eqn: scalar action of K} and \eqref{eqn: scalar action of Ci}, it follows that $P$ acts on any irreducible representation $V$ by the scalar:
\[
\frac{|H|^{2g}}{\chi(1)^{2g}} \prod_{j=1}^n \dfrac{\sum_{c \in [g_j] \cap H} \chi(c)}{\chi(1)}.
\]
Decomposing the regular representation into its irreducible factors
\[
\kfield[H] = \oplus_{\chi \in \Irr(H)} (V_{\chi})^{\chi(1)}
\]
we find that the trace of $P$ acting on $\kfield[H]$ is also given by:
\begin{equation} \label{eqn: trace via irreps}
\sum_{\chi \in \Irr(H)} (\chi(1)^2)\frac{|H|^{2g}}{\chi(1)^{2g}} \prod_{j=1}^n \dfrac{\sum_{c \in [g_j] \cap H} \chi(c)}{\chi(1)}.
\end{equation}
Combining \eqref{eqn: trace via group counts} and \eqref{eqn: trace via irreps} we find 
\[
S(H) = \sum_{\chi \in \Irr(H)} \frac{|H|^{2g-1}}{\chi(1)^{2g+n-2}} \prod_{j=1}^n \sum_{c \in [g_j] \cap H} \chi(c).
\]
To calculate $|\mathcal{H}_{\theta_0}|$ we now use the inclusion-exclusion principle based on the subgroup of $G$ generated by $a_i,b_i,c_j$. The M\"obius function $\mu$ of $G$ formalizes this: we have 
\[
|\mathcal{H}_{\theta_0}| = \sum_{H \leqslant G} \mu(H,G) S(H).
\]
Therefore, we find 
\[
|\mathcal{H}_{\theta_0}| = \sum_{H \leqslant G} \mu(H,G)|H|^{2g-1} \sum_{\chi \in Irr(H)} \dfrac{\prod_{j=1}^n \sum_{c \in [g_j] \cap H} \chi(c)}{{\chi(1)^{2g+n-2}} }.\qedhere
\]
\end{proof}

\begin{remark}
    While finding the irreducible characters of all subgroups $H \leqslant G$ and computing the M\"obius function $\mu(H,G)$ can be tedious, the character formula is computationally advantageous (compared to counting the monodromy representations by hand) because it is essentially independent of the genus $g$. The character-theoretic quantities need only be calculated once, and then the resulting formula is valid for all $g$.
\end{remark}

\begin{remark} In \zcref{prop: charformula}, the term
\[ |H|^{2g-1} \sum_{\chi \in Irr(H)}\dfrac{\prod_{j=1}^n \sum_{c \in [g_j] \cap H} \chi(c)}{{\chi(1)^{2g+n-2}} }\]
corresponding to $H \leqslant G$ counts monodromy representations which have type $\Theta$ except that their image is contained in $H$ (instead of equalling $G$).

A subtlety is that $[g_j] \cap H$ may not be an $H$-conjugacy class, but rather a union thereof. This subtlety is responsible for the sum over all representatives $c \in [g_j] \cap H$. Focusing on the $H=G$ term, we have a genuine conjugacy class $\Omega_j \colonequals [g_j]$, and since $\chi$ is conjugation invariant, the term simplifies to the count given by the \textbf{Dijkgraaf--Witten} formula (see e.g. \cite[Theorem~3]{ZagierFiniteGroups}):
    \[
|G|^{2g-1} \sum_{\chi \in Irr(G)}\dfrac{\prod_{j=1}^n |\Omega_j| \cdot \chi(\Omega_j)}{\chi(1)^{2g+n-2}} = |G|^{2g-1}|\Omega_1|\cdots|\Omega_n| \sum_{\chi \in Irr(G)}\dfrac{\chi(\Omega_1)\cdots\chi(\Omega_n)}{\chi(1)^{2g+n-2}}.
\]
\end{remark}

\begin{example} \label{example: char formula global}
	We calculate the degree of the target morphism
	\[ \Mcalbar_{g,0}^{\Theta}(\Bcal S_3) \to \Mcalbar_{g,0}, \]
	where $\Theta$ is the global type associated to the global image list $\theta_0 = (G=S_3 \leqslant S_3)$. In this case,  the normaliser is equal to the entire group: $N_{S_3}(\theta_0) = S_3$. The subgroups of $S_3$ are
	\[
	1, C_2^1, C_2^2, C_2^3, A_3, S_3,
	\]
	where the $C_2^i$ are the order-$2$ cyclic subgroups generated by the transpositions. The M\"obius function is given by
\[
	\mu(1,C_2^{i}) = \mu(1,A_3) = -1, \mu(1,S_3) = 3, \\
	\mu(C_2^{i},S_3) = -1, \\
	\mu(A_3,S_3) = -1
\]
	and $\mu(H,H)=1$ for all $H \leqslant S_3$. As there are no markings, \zcref{prop: charformula} simplifies to give:
	\begin{align*}
	\deg \big( \Mcalbar_{g,0}^{\Theta}(\Bcal S_3) \to \Mcalbar_{g,0} \big) &  = \dfrac{1}{3!} \sum_{H \leqslant G} \mu(H,G)|H|^{2g-1} \sum_{\chi \in \Irr(H)} \dfrac{1}{\chi(1)^{2g-2}} \\
	& = \dfrac{1}{3!} \sum_{H \leqslant G} \mu(H,G)|H|^{2g-1} \sum_{V} \dfrac{1}{\rk(V)^{2g-2}},
	\end{align*}
	where the sum in the second line is over irreducible representations $V$ of $H$. We enumerate these irreducible representations and their ranks:
	\begin{itemize}
	\item $H=1$: $1$ representation with rank $1$.
	\item $H=C_2^i$: $2$ representations with ranks $1,1$.
	\item $H=A_3$: $3$ representations with ranks $1,1,1$.
	\item $H=S_3$: $3$ representations with ranks $1,1,2$.	
	\end{itemize}
	We therefore obtain:
	\begin{align*}
	\deg \big( \Mcalbar_{g,0}^{\Theta}(\Bcal S_3) \to \Mcalbar_{g,0} \big) & = \dfrac{1}{6} \left( 6^{2g-1}(1+1+1/2^{2g-2}) - 3^{2g-1}(1+1+1) - 3 \cdot 2^{2g-1}(1+1) + 3 \right) \\
    & = 3^{2g-2} - 2^{2g-1} + \tfrac{1}{2} \big( 1 + 3^{2g-2}(2^{2g} - 3) \big).
	\end{align*}
	Specialising to $g=2$, we get 
	\[
	\deg \big( \Mcalbar_{2,0}^{\Theta}(\Bcal S_3) \to \Mcalbar_{2,0} \big)  = 60.
	\]
	We will use this calculation in \zcref{sec: degree 3 covers non-cyclic}.
\end{example}

\subsubsection{$G$-covers versus $S_d$-covers} The moduli space of $S_d$-covers of global type $\Theta$ can be realised as the image of a moduli space of $G$-covers. In this paper we prefer to work within the framework of $S_d$-covers throughout, but we explain the connection to $G$-covers for completeness.

Fix a global type $\Theta$ and choose a representative
\begin{equation} \label{eqn: G covers to Sd covers theta0} \theta_0 = (G,[g_1],\ldots,[g_n]) \end{equation}
of the equivalence class $\Theta$.  Let 
\[ \Mcalbar_{g,n}^{[\theta_0]}(\Bcal G) \subseteq \Mcalbar_{g,n}(\Bcal G) \]
denote the clopen substack given as the closure of the locus of $G$-covers with smooth base, such that the associated monodromy representation 
\[ \rho \colon \pi_1(B \setminus \{b_1,\ldots,b_n\}) \to G \]
is surjective and satisfies $\rho(\gamma_i) \in [g_i]$ for all $i \in [n]$, where $\gamma_i$ denotes a small loop around $b_i$ (these conditions are invariant under post-conjugation by $G$). We have a morphism:
\begin{equation} \label{eqn: map G cover to Sd cover} \Mcalbar_{g,n}^{[\theta_0]}(\Bcal G) \to \Mcalbar_{g,n}^{\Theta}(\Bcal S_d). \end{equation}
Recall the definition \eqref{eqn: normaliser group} of the normaliser group $N_{S_d}(\theta_0)$.

\begin{proposition} \label{prop: G covers to Sd covers} The morphism \eqref{eqn: map G cover to Sd cover} is surjective, of degree:
\[ \dfrac{|N_{S_d}(\theta_0)|}{|G|}. \]
\end{proposition}

\begin{proof}
Consider the diagram:
\begin{equation} \label{eqn: diagram G cover to Sd cover}
\begin{tikzcd}
    \Mcalbar_{g,n}^{[\theta_0]}(\Bcal G) \ar[rr,"f"] \ar[rd,"q" below] && \Mcalbar_{g,n}^{\Theta}(\Bcal S_d) \ar[ld,"p"] \\
    & \Mcalbar_{g,n}.& 
\end{tikzcd}
\end{equation}
Recall from \zcref{sec: degree target map global} the definition of the abstract group $\pi_1(S_{g,n})$ and the set of homomorphisms \eqref{eqn: Theta monodromy rep}. We let
\begin{equation}
\label{eqn: ThetaG monodromy rep} \Hom^{\theta_0}(\pi_1(S_{g,n}),G) \end{equation}
denote the set of homomorphisms $\rho \colon \pi_1(S_{g,n}) \to G$ which are surjective and such that $\rho(\gamma_i) \in [g_i]$ for all $i \in [n]$. 

Given a smooth curve $(B,b_1,\ldots,b_n) \in \Mcal_{g,n}$ we choose an isomorphism $\pi_1(S_{g,n}) \cong \pi_1(B \setminus \{b_1,\ldots,b_n\})$. A choice of lift along $q$ or $p$ is then equivalent to a choice of monodromy representation in \eqref{eqn: ThetaG monodromy rep} or \eqref{eqn: Theta monodromy rep} respectively, up to conjugation. Since every element of \eqref{eqn: Theta monodromy rep} is $S_d$-conjugate to an element of \eqref{eqn: ThetaG monodromy rep}, we conclude that $f$ is dominant, hence surjective. To calculate its degree, we write
\[ \deg(f) = \dfrac{\deg(q)}{\deg(p)} \]
and note (see \zcref{lem: degree target map global}) that
\[ \deg(q) = \dfrac{ |\Hom^{\theta_0}(\pi_1(S_{g,n}),G) |}{|G|}, \qquad \deg(p) = \dfrac{ | \Hom^{\Theta}(\pi_1(S_{g,n}),S_d) | }{d!}, \]
which gives
\begin{equation} \label{eqn: proof G covers to Sd covers deg f as fraction} \deg(f) = \dfrac{ |\Hom^{\theta_0}(\pi_1(S_{g,n}),G) |}{|\Hom^{\Theta}(\pi_1(S_{g,n}),S_d)|} \cdot \dfrac{d!}{|G|}.\end{equation}
We now simplify. Given $\theta \in \Theta$ we let
\[ \Hom^{\theta}(\pi_1(S_{g,n}),S_d) \]
denote the set of monodromy representations whose associated global image list coincides with $\theta$ (rather than simply being $S_d$-conjugate to $\theta$). By definition we have:
\begin{equation} \label{eqn: proof G covers to Sd covers decomposition of hom set} \Hom^{\Theta}(\pi_1(S_{g,n}),S_d) = \bigsqcup_{\theta \in \Theta} \Hom^{\theta}(\pi_1(S_{g,n}),S_d).\end{equation}
Given $\theta_1,\theta_2 \in \Theta$ we have by definition an element $\sigma \in S_d$ such that $\sigma \theta_1 \sigma^{-1} = \theta_2$. The action of $\sigma$ by post-conjugation then defines a bijection
\[ \sigma \colon \Hom^{\theta_1}(\pi_1(S_{g,n}),S_d) \xrightarrow{\cong} \Hom^{\theta_2}(\pi_1(S_{g,n}),S_d).\]
It follows that every piece of the decomposition \eqref{eqn: proof G covers to Sd covers decomposition of hom set} has the same size. Taking our preferred representative \eqref{eqn: G covers to Sd covers theta0} then gives
\[ |\Hom^{\Theta}(\pi_1(S_{g,n}),S_d)| = |\Theta| \cdot |\Hom^{\theta_0}(\pi_1(S_{g,n}),S_d)| = \dfrac{d! \cdot |\Hom^{\theta_0}(\pi_1(S_{g,n}),G)|}{|N_{S_d}(\theta_0)|} \] 
where the second equality follows from the orbit-stabiliser theorem. Plugging back into \eqref{eqn: proof G covers to Sd covers deg f as fraction} gives the result.
\end{proof}

\subsection{Boundary strata in moduli spaces of curves} \label{sec: boundary G-covers}

From now on we fix a global type $\Theta$ as in \zcref{sec: setup}. For each marking index $i \in [n]$, we define the rooting parameter
\[ r_i \colonequals \ord(g_i), \]
and we let $r_{\ul{n}} \colonequals (r_1,\ldots,r_n)$. There is then a tower of forgetful morphisms
\begin{align}\label{eqn: sequence of forgetful maps} \Mcalbar_{g,n}^{\Theta}(\Bcal S_d) & \to \Mcalbar^{\operatorname{tw}}_{g,r_{\ul{n}}} \to \Mcalbar_{g,n} \\
\nonumber (P \to \Bcal) & \mapsto \quad \Bcal \ \ \ \ \mapsto \ \ B \end{align}
where $\Mcalbar^{\operatorname{tw}}_{g,r_{\ul{n}}}$ is the moduli space of balanced twisted nodal curves with isotropy $r_{\ul{n}}$ at the markings. We are interested in studying the restriction of this tower to the boundary. We begin by describing the respective indexing sets for the boundary strata in the three moduli spaces.

\subsubsection{Stable graphs}

\begin{definition} A \textbf{graph} $\Gamma$ consists of a finite set of vertices $V(\Gamma)$, a finite set $H(\Gamma)$ of half-edges, a root map $H(\Gamma) \to V(\Gamma)$, and an involution $H(\Gamma) \to H(\Gamma)$.	
\end{definition}

The fixed points of the involution are referred to as \textbf{legs}: their set is denoted $L(\Gamma)$. The orbits of size two are referred to as \textbf{edges}: their set is denoted $E(\Gamma)$. Given an edge $e$, the associated half-edges are denoted $\vec{e}$ and $\cev{e}$ and are also referred to as \textbf{oriented edges}. The set of oriented edges is denoted $\vec{E}(\Gamma)$ so that we have $H(\Gamma) = \vec{E}(\Gamma) \sqcup L(\Gamma)$. For a vertex $v \in V(\Gamma)$ we let $H_v(\Gamma)$ denote the set of half-edges rooted at $v$.

\begin{definition} \label{def: stable graph} A \textbf{stable graph} is a graph equipped with genus and marking labels:
\[ g \colon V(\Gamma) \to \ZZ_{\geqslant 0}, \qquad L(\Gamma) \xrightarrow{\cong} [n] \colonequals \{1,\ldots,n\}, \]
 subject to the stability condition $2g_v - 2 + |H_v(\Gamma)| > 0$ for every $v \in V(\Gamma)$. We will abuse notation and continue to write $\Gamma$ to indicate a stable graph.
\end{definition}

To each stable graph $\Gamma$ there is an associated boundary stratum
\[ \Mcalbar_\Gamma \subseteq \Mcalbar_{g,n} \]
parametrising stable curves whose dual graphs are (degenerations of) $\Gamma$.

\subsubsection{Twisted stable graphs}\label{sec: twisted stable graphs}

\begin{definition} \label{def: twisted stable graph} A \textbf{twisted stable graph} is a stable graph further equipped with isotropy labels
\[ r \colon E(\Gamma) \sqcup L(\Gamma) \to \ZZ_{\geqslant 1}. \]
\end{definition}

To each twisted stable graph $(\Gamma,r)$ there is an associated boundary stratum
\[ \Mcalbar^{\tw}_{(\Gamma,r)} \subseteq \Mcalbar^{\operatorname{tw}}_{g,r_{\ul{n}}} \]
parametrising twisted stable curves whose dual graphs are (degenerations of) $\Gamma$ and which have stabiliser $\mu_{r_e}$ along the node corresponding to the edge $e \in E(\Gamma)$. The isotropy labels $r_{\ell}$ for $\ell \in L(\Gamma)$ are required to coincide with $r_{\ul{n}}$ under the marking labelling $L(\Gamma) \xrightarrow{\cong} [n]$. The second map of \eqref{eqn: sequence of forgetful maps} restricts to a map of boundary strata
\[ \Mcalbar_{(\Gamma,r)}^{\tw} \to \Mcalbar_\Gamma.\]

This describes the indexing sets for boundary strata in the moduli spaces of stable curves and twisted stable curves. To do the same for the moduli space of covers, we must first introduce the orbifold fundamental group (\zcref{def: orbifold pi1}) and explain its relationship to $S_d$-covers (\zcref{prop: G cover same as monodromy rep of pi1orb}).

\subsection{Orbifold fundamental groups} \label{sec: orbifold fundamental group}

Fix a balanced twisted nodal curve $\Bcal$ with dual graph $(\Gamma,r)$. We wish to count the number of principal $S_d$-bundles $P \to \Bcal$. The answer is given by counting monodromy representations of the orbifold fundamental group:
\[ \rho \colon \pi_1^{\orb}(\Bcal) \to S_d.\] 
This orbifold fundamental group depends only on the underlying twisted stable graph $(\Gamma,r)$, and can be described using the language of Bass--Serre theory \cite{SerreTrees}.

\begin{definition} \label{def: orbifold pi1} Fix a twisted stable graph $(\Gamma,r)$. The \textbf{orbifold fundamental group}
\[ \pi_1^{\orb}(\Gamma,r) \] 
is defined as follows. Choose an auxiliary spanning tree $T \subseteq E(\Gamma)$ and define $\pi_1^{\orb}(\Gamma,r)$ via the following sets of generators and relations:
\begin{itemize}
\item \textbf{Generators.}
\begin{itemize}
\item For each $v \in V(\Gamma)$: \[\alpha_1^{(v)}, \beta_1^{(v)}, \ldots, \alpha_{g_v}^{(v)}, \beta_{g_v}^{(v)}.\]
\item For each $h \in H(\Gamma)$: \[ \gamma_h.\]
\item For each $e \in E(\Gamma) \setminus T$: \[ \delta_e.\]
\end{itemize}	
\item \textbf{Relations.}
\begin{itemize}
\item For each $v \in V(\Gamma)$: \[\prod_{i=1}^{g_v} [\alpha_i^{(v)}, \beta_i^{(v)}] \cdot \prod_{h \in H_v(\Gamma)} \gamma_h = 1.\]
\item For each $h \in H(\Gamma)$: \[\gamma_h^{r_h}=1.\]
\item For each $e \in T$: \[ \gamma_{\vec{e}} = (\gamma_{\cev{e}})^{-1}.\]
\item For each $e \in E(\Gamma) \setminus T$: \[ \delta_e \gamma_{\vec{e}} \delta_e^{-1} = (\gamma_{\cev{e}})^{-1}.\]
\end{itemize}
\end{itemize}
\end{definition}

\begin{proposition} \label{prop: G cover same as monodromy rep of pi1orb} Fix a balanced twisted curve $\Bcal$ with dual graph $(\Gamma,r)$. There is an isomorphism
\begin{equation} \label{eqn: isomorphism pi1 of curve and graph}  \pi_1^{\orb}(\Gamma,r) \xrightarrow{\cong} \pi_1^{\orb}(\Bcal) \end{equation}
sending $\gamma_i$ to a small loop around $b_i$. This isomorphism is not unique, but the image of each $\gamma_i$ is well-defined up to conjugation by the generators associated to the vertex supporting the $i$th marked leg. Once the isomorphism is fixed, the data of a principal $S_d$-bundle $P \to \Bcal$ is equivalent to the data of a monodromy representation
\[ \rho \colon \pi_1^{\orb}(\Gamma,r) \to S_d .\]
Moreover, isomorphisms between $S_d$-covers (including automorphisms) are given by the action of $S_d$ on the set of monodromy representations by post-conjugation.
\end{proposition}

\begin{proof} We first establish the isomorphism. The case of a smooth curve is given e.g. in \cite[Section~5]{BehrendNoohi}. For the general case, the orbifold van~Kampen theorem \cite[Theorem~5.10]{NoohiFibrations} allows one to reduce from the balanced twisted curve to the associated graph of groups, and then \cite[Section~19.5]{NoohiStacks1} demonstrates the equivalence with Bass--Serre theory.

Once the isomorphism is established, a general fact (see e.g. \cite[Theorem~18.19]{NoohiStacks1}) provides the correspondence between principal bundles and monodromy representations of the orbifold fundamental group.
\end{proof}

This framework allows us to remove the closure operation appearing in \zcref{def: decomposition by type}. We first generalise Definitions~\ref{def: global image list}~and~\ref{def: global type smooth cover} from smooth to nodal bases.

\begin{definition}[Global image list associated to a monodromy representation of the orbifold fundamental group] \label{def: global type nodal cover} Given a monodromy representation
\begin{equation} 
\label{eqn: monodromy rep orbifold pi1} \rho \colon \pi_1^{\orb}(\Gamma,r) \to S_d,
\end{equation}
we define associated images:
\begin{align*}
    G & \colonequals \rho(\pi_1^{\orb}(\Gamma,r)) \leqslant S_d \text{\, and } \\
    [g_i] & \colonequals [\rho(\gamma_i)] \subseteq G.
\end{align*}
We define the \textbf{global image list associated to $\rho$} as the data:
\[ \theta_\rho \colonequals (G,[g_1],\ldots,[g_n]).\]
\end{definition}

\begin{definition}[Global type associated to a cover] Fix an $S_d$-cover $P \to \Bcal$. Choosing an isomorphism $\pi_1^{\orb}(\Gamma,r) \cong \pi_1^{\orb}(\Bcal)$ determines a monodromy representation
\[ \rho \colon \pi_1^{\orb}(\Gamma,r) \to S_d\]
up to post-conjugation. The associated global image list $\theta_\rho$ does not depend on the choice of isomorphism \eqref{eqn: isomorphism pi1 of curve and graph}, and its equivalence class under the action of $S_d$ does not depend on the choice of monodromy representation. We thus obtain a global type
\[ \Theta_{P \to \Bcal} \]
which we refer to as the \textbf{global type} of the $S_d$-cover.
\end{definition}

With these definitions, the following is immediate:

\begin{lemma} \label{lem: decomposition by type without closure} \label{rmk: definition type locus without closure} Fix a global type $\Theta$. The substack
\[ \Mcalbar_{g,n}^{\Theta}(\Bcal S_d) \subseteq \Mcalbar_{g,n}(\Bcal S_d) \]
constructed in \zcref{def: decomposition by type} via a closure operation, can equivalently be defined as the locus of $S_d$-covers such that $\Theta_{P \to \Bcal} = \Theta$.
\end{lemma}

\subsection{Vertexwise boundary strata} \label{sec: vertexwise strata}

We now   study boundary strata in the moduli space of $S_d$-covers of global type $\Theta$.

\subsubsection{Vertexwise $S_d$-graphs} We begin by introducing the relevant indexing set. As seen in \zcref{prop: G cover same as monodromy rep of pi1orb}, a point in the boundary determines (up to appropriate choices) a monodromy representation of the orbifold fundamental group:
\[ \rho \colon \pi_1^{\orb}(\Gamma,r) \to S_d. \]
This will allow us to associate a global type $\Theta_v$ to each vertex $v \in V(\Gamma)$. The collection of this data will be referred to as a vertexwise $S_d$-graph
\[ (\Gamma,\Theta_{V(\Gamma)}) \]
and these will index the boundary strata in the moduli space.

\begin{definition} \label{def: vertex image list} Fix a stable graph $\Gamma$. For each vertex $v \in V(\Gamma)$, a \textbf{vertex image list} is given by a tuple
% simply the data of a global image list of shape $(v,H_v(\Gamma))$:
\[ \theta_v = (G_v \leqslant S_d, [g_h] \subseteq G_v \mid h \in H_v(\Gamma)).\] 
A \textbf{vertex type} $\Theta_v$ is an equivalence class of a vertex image list $\theta_v$ under the action of $S_d$ by simultaneous conjugation.
\end{definition}

\begin{definition} \label{def: vertexwise type}
Fix a stable graph $\Gamma$. A \textbf{vertexwise image list}
\[ \theta_{V(\Gamma)} = (\theta_v)_{v \in V(\Gamma)} \]
is a collection of vertex image lists, one for each vertex, such that the representatives $g_{\vec{e}}$ and $(g_{\cev{e}})^{-1}$ at each edge $e \in E(\Gamma)$ are $S_d$-conjugate to each other. 

The associated \textbf{vertexwise type} $\Theta_{V(\Gamma)}$ is the equivalence class of $\theta_{V(\Gamma)}$ under the action of $(S_d)^{V(\Gamma)}$ by vertexwise conjugation. Equivalently, it is the list of all vertex types
\[ \Theta_{V(\Gamma)} = (\Theta_v)_{v \in V(\Gamma)}\]
with one copy of $S_d$ acting independently for each vertex. 
\end{definition}

\begin{definition} \label{def: vertexwise Sd graph} A \textbf{vertexwise $S_d$-graph} is a pair
\[ (\Gamma,\Theta_{V(\Gamma)}) \]
consisting of a stable graph $\Gamma$ and a vertexwise type $\Theta_{V(\Gamma)}$.
\end{definition}

\begin{definition} \label{def: vertexwise graph to twisted graph}    
A vertexwise $S_d$-graph $(\Gamma,\Theta_{V(\Gamma)})$ determines a twisted stable graph $(\Gamma,r)$ by setting
\[ r_e \colonequals \ord(g_{\vec{e}}) \]
for every  $e \in E(\Gamma) \sqcup L(\Gamma)$. The order does not depend on the choice of orientation since $g_{\vec{e}}$ and $(g_{\cev{e}})^{-1}$ are conjugate in $S_d$.
\end{definition}

\subsubsection{Vertexwise type associated to a cover} 
We now explain how to construct a vertexwise $S_d$-graph from an $S_d$-cover. Recall the presentation of the orbifold fundamental group $\pi_1^{\orb}(\Gamma,r)$ from \zcref{def: orbifold pi1}.

\begin{definition} \label{def: vertex subgroup}
Fix a twisted stable graph $(\Gamma,r)$. Given $v \in V(\Gamma)$, let
\[ \pi_1^{\orb}(v,r_{H_v(\Gamma)}) \leqslant \pi_1^{\orb}(\Gamma,r) \]
denote the subgroup generated by the following subset of the generators:
\[ (\alpha^{(v)}_i,\beta^{(v)}_i \mid i =1,\ldots,g_v) \sqcup (\gamma_h \mid h \in H_v(\Gamma)).\]
\end{definition}

\begin{definition}[Vertexwise image list associated to a monodromy representation] \label{def: vertexwise image list associated to rep} 
Fix a twisted stable graph $(\Gamma,r)$ and a monodromy representation $\rho \colon \pi_1^{\orb}(\Gamma,r) \to S_d$. For $v \in V(\Gamma)$ we define the images
\begin{alignat*}{3}
G_v & \colonequals \rho(\pi_1^{\orb}(v,r_{H_v(\Gamma)})) \leqslant S_d, \\
[g_h] & \colonequals [\rho(\gamma_h)] \subseteq G_v,
\end{alignat*}
as $h$ varies over the elements in $H_v(\Gamma)$. The \textbf{associated vertex image list} is given by $\theta_v \colonequals (G_v,[g_h] \mid h \in H_v(\Gamma))$ and the \textbf{vertexwise image list associated to $\rho$} is given by:
\[ \theta_{V(\Gamma)} \colonequals (\theta_v)_{v \in V(\Gamma)}. \]
\end{definition}

\begin{definition}[Vertexwise $S_d$-graph associated to a cover] \label{def: vertexwise type associated to cover}
Fix an $S_d$-cover $P \to \Bcal$ and let $(\Gamma,r)$ be the twisted dual graph of $\Bcal$. The associated monodromy representation
\[ \rho \colon \pi_1^{\orb}(\Bcal) \to S_d \]
is well-defined up to post-conjugation. Choosing an isomorphism \eqref{eqn: isomorphism pi1 of curve and graph} we obtain a monodromy representation $\rho \colon \pi_1^{\orb}(\Gamma,r) \to S_d$. The vertexwise image list $\theta_{V(\Gamma)}$ associated to $\rho$ does not depend on the choice of isomorphism \eqref{eqn: isomorphism pi1 of curve and graph}, and its equivalence class $\Theta_{V(\Gamma)}$ does not depend on the choice of monodromy representation. The pair
\[ (\Gamma,\Theta_{V(\Gamma)}) \]
is referred to as the \textbf{vertexwise $S_d$-graph associated to the cover $P \to \Bcal$}.
\end{definition}

\subsubsection{Vertexwise boundary strata} In order to  define the boundary stratum associated to a vertexwise $S_d$-graph, we  proceed in two stages.

\begin{definition} \label{def: boundary strata vertex decorated graph 1} Fix a vertexwise $S_d$-graph $(\Gamma,\Theta_{V(\Gamma)})$. The \textbf{unrestricted vertexwise locus}
\begin{equation} \label{eqn: VGamma stratum} \Mcalbar_{(\Gamma,\Theta_{V(\Gamma)})}(\Bcal S_d) \subseteq \Mcalbar_{g,n}(\Bcal S_d) \end{equation}
parametrises $S_d$-covers $P \to \Bcal$ such that the associated vertexwise $S_d$-graph (\zcref{def: vertexwise type associated to cover}) is equal to $(\Gamma,\Theta_{V(\Gamma)})$.
\end{definition}

\emph{Crucially}, the unrestricted vertexwise locus may intersect several different pieces of the decomposition \eqref{eqn: decomposition according to type}. This is because the image $G$ of the monodromy representation contains amongst its generators the elements $\rho(\delta_e)$ for $e \in E(\Gamma) \setminus T$, and the vertexwise type contains no information on these. In order to restrict exclusively to the moduli space 
 \[ \Mcalbar_{g,n}^{\Theta}(\Bcal S_d) \]
 associated to our chosen global type $\Theta$, we thus make the following key definition.

\begin{definition} \label{def: boundary strata vertex decorated graph 2} Fix a global type $\Theta$ and a vertexwise $S_d$-graph $(\Gamma,\Theta_{V(\Gamma)})$. The associated (restricted) \textbf{vertexwise stratum} is denoted and defined
\begin{equation} \label{eqn: VGamma theta stratum} \Mcalbar_{(\Gamma,\Theta_{V(\Gamma)})}^{\Theta}(\Bcal S_d) \colonequals 
\Mcalbar_{(\Gamma,\Theta_{V(\Gamma)})}(\Bcal S_d) \cap \Mcalbar_{g,n}^{\Theta}(\Bcal S_d).\end{equation}
Note that this is a clopen substack of the unrestricted vertexwise locus defined above.  

\end{definition}

We then have the following diagram
\[
\begin{tikzcd} \label{eqn: cartesian diagram boundary decomposition 2}
\coprod_{\Theta_{V(\Gamma)}} \Mcalbar_{(\Gamma,\Theta_{V(\Gamma)})}^{\Theta}(\Bcal S_d) \ar[r,hook] \ar[d] \ar[rd,phantom,"\square"] & \Mcalbar_{g,n}^{\Theta}(\Bcal S_d) \ar[d] \\
\Mcalbar_{(\Gamma,r)}^{\tw} \ar[r,hook] & \Mcalbar_{g,r_{\ul{n}}}^{\tw}
\end{tikzcd}
\]
where the union is over vertexwise types $\Theta_{V(\Gamma)}$ such that $(\Gamma,\Theta_{V(\Gamma)}) \mapsto (\Gamma,r)$ (\zcref{def: vertexwise graph to twisted graph}).  Note that, similarly to \zcref{remark: decomposition not connected}, the vertexwise strata are in general not connected. 

\subsubsection{Ramification profiles along edges} \label{sec: ramification profile boundary stratum}

Fix a vertexwise $S_d$-graph $(\Gamma,\Theta_{V(\Gamma)})$. For each vertex $v \in V(\Gamma)$ and half-edge $h \in H_v(\Gamma)$, the associated group element
\[ g_h \in G_v \]
is well-defined up to conjugation in $G_v$. Taking its image in $S_d$ hence determines a partition of $d$. The condition that $g_{\vec{e}}$ and $(g_{\cev{e}})^{-1}$ are conjugate in $S_d$ ensures that the partitions associated to $\vec{e}$ and $\cev{e}$ coincide for each edge $e$ . We thus obtain partitions associated to the edges and legs of the graph,
\begin{align} \nonumber E(\Gamma) \sqcup L(\Gamma) & \to \operatorname{Part}(d) \\
\label{eqn: ramification profile vertexwise stratum} e & \mapsto m_e = (m_{e1},\ldots,m_{el_e}),
\end{align}
which we refer to as the \textbf{ramification profiles}.

\subsubsection{Forgetful degrees for vertexwise boundary strata} \label{sec: forgetful degree graphical boundary stratum}

Recall (\zcref{def: vertexwise graph to twisted graph}) that there is an operation
\[ (\Gamma,\Theta_{V(\Gamma)}) \mapsto (\Gamma,r) \]
sending a vertexwise $S_d$-graph to a twisted stable graph. With $m_e$ as in \zcref{sec: ramification profile boundary stratum}, we have:
\begin{equation} \label{eqn: group order as lcm} r_e = \ord(g_e) = \lcm(m_e) = \lcm(m_{e1},\ldots,m_{el_e}).\end{equation}

The tower \eqref{eqn: sequence of forgetful maps} now restricts to a tower of strata:
\[ \Mcalbar_{(\Gamma,\Theta_{V(\Gamma)})}^{\Theta}(\Bcal S_d) \to \Mcalbar_{(\Gamma,r)}^{\tw} \to \Mcalbar_\Gamma \]
where each stratum has codimension $|E(\Gamma)|$ inside its respective moduli space. We calculate the degrees of these forgetful maps.

\begin{definition} \label{def: group homs for VGamma} Define
\[ \Hom_{\Theta_{V(\Gamma)}}^{\Theta}(\pi_1^{\orb}(\Gamma,r),S_d) \]
to be the set of homomorphisms $\rho \colon \pi_1^{\orb}(\Gamma,r) \to S_d$ such that the associated global type (\zcref{def: global type nodal cover})  coincides with $\Theta$ and the associated vertexwise type (\zcref{def: vertexwise image list associated to rep}) coincides with $\Theta_{V(\Gamma)}$.
\end{definition}

\begin{proposition} \label{lem: degree target map vertexwise stratum}  We have
\begin{align*} \deg \big( \Mcalbar^{\Theta}_{(\Gamma,\Theta_{V(\Gamma)})}(\Bcal S_d) \to \Mcalbar_{(\Gamma,r)}^{\tw} \big) = \dfrac{\big| \Hom_{\Theta_{V(\Gamma)}}^{\Theta}(\pi_1^{\orb}(\Gamma,r),S_d)  \big|}{d!}, \\[0.2cm]
\deg \big( \Mcalbar^{\Theta}_{(\Gamma,\Theta_{V(\Gamma)})}(\Bcal S_d) \to \Mcalbar_\Gamma \big) = \dfrac{\big| \Hom_{\Theta_{V(\Gamma)}}^{\Theta}(\pi_1^{\orb}(\Gamma,r),S_d)  \big|}{d! \prod_{e \in E(\Gamma)} \lcm(m_e)},	
\end{align*}
where $m_e$ is the ramification profile \eqref{eqn: ramification profile vertexwise stratum}.
\end{proposition}

\begin{proof} The first equality follows from \zcref{prop: G cover same as monodromy rep of pi1orb} and the definition of the boundary stratum. The second equality then follows from
\[ \deg \big( \Mcalbar^{\tw}_{(\Gamma,r)} \to \Mcalbar_\Gamma \big) =  \dfrac{1}{\prod_{e \in E(\Gamma)} r_e} \]
(see \cite[Theorem~1.10]{OlssonLogTwisted}) combined with the identity \eqref{eqn: group order as lcm}.
\end{proof}

\subsubsection{Boundary character formula} As in \zcref{prop: charformula}, we derive a character formula for the count of monodromy representations appearing in \zcref{lem: degree target map vertexwise stratum}:
\[ \big| \Hom_{\Theta_{V(\Gamma)}}^{\Theta}(\pi_1^{\orb}(\Gamma,r),S_d)  \big| \] Fix therefore a global type $\Theta$ and a vertexwise $S_d$-graph $(\Gamma, \Theta_{V(\Gamma)})$. Choose also an auxiliary spanning tree $T \subseteq E(\Gamma)$. 

We begin by establishing the necessary notation. The vertexwise type consists of the data of vertex types $\Theta_v$ for each $v \in V(\Gamma)$. Each $\Theta_v$ is an equivalence class of vertex image list, and we write $\theta_v$ for a chosen representative, consisting of the data:
\[ \theta_v = (G_v \leqslant S_d, [g_h] \subseteq G_v \mid h \in H_v(\Gamma)) \]
Given a fixed subgroup $L \leqslant S_d$ we let
\begin{equation} \label{eqn: definition H-realisable types} \operatorname{Real}(\Theta_v,L) \subseteq \Theta_v \end{equation}
denote the set of representatives $\theta_v$ such that $G_v \leqslant L$. We refer to such representatives as \textbf{$L$-realisable}. We also define an indicator function
\begin{equation} \label{eqn: definition indicator function} I(g_1,g_2) \colonequals \begin{cases} 1 \qquad & \text{if $g_1=g_2^{-1}$} \\ 0 \qquad & \text{otherwise} \end{cases} \end{equation}
for $g_1,g_2 \in S_d$. For any subgroup $L \leqslant S_d$ we define a counting function
\begin{equation}  \label{eqn: definition centraliser counting function} \gl_L(g_1,g_2) \colonequals \begin{cases} |C_L(g_1)| \qquad & \text{if $[g_1]_L = [g_2^{-1}]_L$} \\ 0 \qquad & \text{otherwise} \end{cases} \end{equation}
where $[-]_L$ indicates the $L$-conjugacy class and $C_L(g_1) \leqslant L$ is the centraliser. Note that this is symmetric in $g_1,g_2$. Finally choose a representative
\[ \theta_0 = (G,[g_1],\ldots,[g_n]) \]
of $\Theta$ and recall from \eqref{eqn: normaliser group} the definition of the normaliser group $N_{S_d}(\theta_0)$.

\begin{proposition} \label{prop: boundary char formula}
We have:
\begin{align*}
\big| \Hom_{\Theta_{V(\Gamma)}}^\Theta(\pi_1^{orb}(\Gamma,r),S_d) \big| & = \dfrac{d!}{|N_{S_d}(\theta_0)|} \sum_{L \leqslant G} \mu(L,G) \sum_{\theta_v \in \operatorname{Real}(\Theta_v,L)} \sum_{L_v \leqslant G_v} \sum_{c_h \in [g_h] \cap L_v} \\[0.2cm]
& \qquad \prod_{e \in T} I(c_{\vec{e}},c_{\cev{e}}) \prod_{e \in E(\Gamma) \setminus T} gl_L(c_{\vec{e}},c_{\cev{e}}) \\[0.2cm]
& \qquad \prod_{v \in V(\Gamma)} \mu(L_v,G_v) |L_v|^{2g_v-1} \sum_{\chi \in Irr(L_v)}\dfrac{\chi(\prod_{h \in H_v(\Gamma)} c_h)}{\chi(1)^{2g_v-1}}.
\end{align*}
On the first line, the second sum is over choices of representatives $\theta_v \in \operatorname{Real}(\Theta_v,L)$ for every $v \in V(\Gamma)$, the third sum is over choices of subgroups $L_v \leqslant G_v$ for every $v \in V(\Gamma)$, and the fourth sum is over choices of representatives $c_h \in [g_h] \cap L_v$ of each conjugacy class $[g_h] \subseteq G_v$, for every vertex $v \in V(\Gamma)$ and half-edge $h \in H_v(\Gamma)$.
\end{proposition}

\begin{proof}
We wish to count homomorphisms up to conjugation
\[
\rho \colon \pi_1^{\orb}(\Gamma,r) \to S_d
\]
whose associated global type (\zcref{def: global type nodal cover}) is $\Theta$ and whose associated vertexwise type (\zcref{def: vertexwise image list associated to rep}) is $\Theta_{V(\Gamma)}$. As in the proof of \zcref{prop: charformula} we choose a representative of the global type and denote it:
\[ \theta_0 = (G,[g_1],\ldots,[g_n]). \]

Referring to the presentation of the orbifold fundamental group (\zcref{def: orbifold pi1}), the homomorphisms we wish to count consist of the data of generators:
\begin{itemize}
	\item For each vertex $v \in V(\Gamma)$:
    \[ a_1^{(v)},b_1^{(v)},\cdots,a_{g_v}^{(v)},b_{g_v}^{(v)} \in S_d.\]
	\item For each half-edge $h \in H(\Gamma)$:
    \[ c_h \in S_d.\]
	\item For each edge $e \in E(\Gamma) \setminus T$:
    \[ d_e \in S_d.\]
\end{itemize}
This data is required to satisfy the obvious group relations, ensuring that the homomorphism is well-defined:
\begin{itemize}
    \item For each $v \in V(\Gamma)$:
    \[ \prod_{i=1}^{g_v} [a_i^{(v)},b_i^{(v)}] \prod_{h \in H_v(\Gamma)} c_h = 1.\]
    \item For each $e \in T$:
    \[ c_{\vec{e}} = (c_{\cev{e}})^{-1}. \]
    \item For each $e \in E(\Gamma) \setminus T$:
    \[ d_e c_{\vec{e}} d_e^{-1} = (c_{\cev{e}})^{-1}. \]
\end{itemize}
It is also required to satisfy the following type conditions, ensuring that the homomorphism has the desired types $\Theta_{V(\Gamma)}$ and $\Theta$:
\begin{itemize}
    \item \textbf{Vertex type.} For each $v \in V(\Gamma)$ consider the subgroup
    \[ \rho(\pi_1^{\orb}(v,r_{H_v(\Gamma)})) \colonequals \langle a_i^{(v)}, b_i^{(v)}, c_h \mid h \in H_v(\Gamma) \rangle \leqslant S_d. \]
    For each $h \in H_v(\Gamma)$ consider the conjugacy classes $[c_h]$ within this subgroup. We then require that the associated vertex image list
    \[ (\rho(\pi_1^{\orb}(v,r_{H_v(\Gamma)}), [c_h] \mid h \in H_v(\Gamma) ) \]
    belongs to the vertex type $\Theta_v$.\smallskip
    \item \textbf{Global type.} Consider the total image of the monodromy representation, i.e. the subgroup of $S_d$ generated by all the chosen generators $a_i^{(v)}, b_i^{(v)}, c_h, d_e$. Equivalently:
    \[ \rho(\pi_1^{\orb}(\Gamma,r)) \colonequals \langle \rho(\pi_1^{\orb}(v,r_{H_v(\Gamma)})), d_e \mid v \in V(\Gamma), e \in E(\Gamma) \setminus T \rangle. \]
    For each leg $\ell \in L(\Gamma)$ consider the conjugacy class $[c_\ell]$ within the above subgroup. We then require that the associated global image list
    \[ (\rho(\pi_1^{\orb}(\Gamma,r)), [c_\ell] \mid \ell \in L(\Gamma) ) \]
    belongs to the vertex type $\Theta$. This means that it is $S_d$-conjugate to the representative $\theta_0$ chosen above.\smallskip
\end{itemize}

We now enumerate the choices of such data. For our chosen representative $\theta_0$ of $\Theta$ we let
\begin{equation} \label{eqn: homomorphisms with theta and ThetaVGamma} \Hom^{\theta_0}_{\Theta_{V(\Gamma)}}(\pi_1^{\orb}(\Gamma,r),G) \subseteq \Hom^{\Theta}_{\Theta_{V(\Gamma)}}(\pi_1^{\orb}(\Gamma,r),S_d)\end{equation}
denote the set of monodromy representations whose global image list is \emph{equal} to $\theta_0$ (rather than simply being $S_d$-conjugate to it). As in the proof of \zcref{prop: G covers to Sd covers} we have
\begin{equation} \label{eqn: boundary character formula first reduction} |\Hom^{\Theta}_{\Theta_{V(\Gamma)}}(\pi_1^{\orb}(\Gamma,r),S_d)| = \dfrac{d! \cdot |\Hom^{\theta_0}_{\Theta_{V(\Gamma)}}(\pi_1^{\orb}(\Gamma,r),G)|}{|N_{S_d}(\theta_0)|} \end{equation}
where $N_{S_d}(\theta_0)$ is the normaliser group \eqref{eqn: normaliser group}. This is our first reduction. It thus remains to compute the cardinality of \eqref{eqn: homomorphisms with theta and ThetaVGamma}. As in the global case, we begin by counting monodromy representations whose image is contained in (but not necessarily equal to) a given subgroup $L \leqslant G$.

Fix therefore a subgroup $L \leqslant G$ and let
\[ \Hom_{\Theta_{V(\Gamma)}}^{[g_1],\ldots,[g_n]}(\pi_1^{\orb}(\Gamma,r),L) \]
denote the set of homomorphisms whose image is contained in (but not necessarily equal to) $L$, whose individual vertex types agree with the $\Theta_v$, and for which the $G$-conjugacy classes at the legs are $[g_1],\ldots,[g_n]$ as specified by the global image list $\theta_0$. The only difference between this and \eqref{eqn: homomorphisms with theta and ThetaVGamma} is that we do not impose that the total image of the monodromy representation is $G$. Imposing this last condition, we obtain by the M\"obius inversion formula:
\begin{equation} \label{eqn: boundary character formula second reduction}
\big| \Hom^{\theta_0}_{\Theta_{V(\Gamma)}}(\pi_1^{\orb}(\Gamma,r),G) \big| = \sum_{L \leqslant G} \mu(L,G) \big|\Hom_{\Theta_{V(\Gamma)}}^{[g_1],\ldots,[g_n]}(\pi_1^{\orb}(\Gamma,r),L) \big|.
\end{equation}
This is our second reduction. On the right-hand side, for each vertex $v \in V(\Gamma)$ we require the associated vertex type to belong to $\Theta_v$. Since the image of the monodromy representation is contained in $L$, it follows that the associated vertex image list must belong to the subset of $L$-realisable image lists
\[ \operatorname{Real}(\Theta_v,L) \subseteq \Theta_v \]
defined in \eqref{eqn: definition H-realisable types}. We thus arrive at our third reduction:
\begin{equation} \label{eqn: boundary character formula third reduction} \big| \Hom^{[g_1],\ldots,[g_n]}_{\Theta_{V(\Gamma)}}(\pi_1^{\orb}(\Gamma,r),L) \big| = \sum_{\theta_{V(\Gamma)} \in \operatorname{Real}(\Theta_{V(\Gamma)},L)} \big| \Hom^{[g_1],\ldots,[g_n]}_{\theta_{V(\Gamma)}}(\pi_1^{\orb}(\Gamma,r),L) \big| \end{equation}
where the sum is over choices of $\theta_v \in \operatorname{Real}(\Theta_v,L)$ for all $v \in V(\Gamma)$. We henceforth fix $L$-realisable image lists $\theta_v$ for each $v \in V(\Gamma)$ and write them as
\[ \theta_v = (G_v \leqslant L, [g_h] \subseteq G_v \mid h \in H_v(\Gamma)).\]
Our goal is to compute the count
\begin{equation} \label{eqn: count fixed vertexwise image list} 
\big| \Hom^{[g_1],\ldots,[g_n]}_{\theta_{V(\Gamma)}}(\pi_1^{\orb}(\Gamma,r),L)\big|.
\end{equation}
Here, while the total image is not required to equal $L$, the local image associated to each vertex is required to equal $G_v$. We will thus apply M\"obius inversion again, this time separately for each vertex. Fix therefore a $v \in V(\Gamma)$ and a subgroup
\[ L_v \leqslant G_v.\]
At this point we note a key difference with the global case. We will later need to glue the vertex counts along connecting edges, and when we do so, conditions on the actual \emph{representatives} of the conjugacy classes must be imposed. We therefore need a finer count of homomorphisms, in order to discard incompatible choices of representatives along connecting edges. We thus fix, for each $h \in H_v(\Gamma)$, a representative
\[
c_h \in [g_h] \cap L_v
\]
of the $G_v$-conjugacy class $[g_h]$, which also belongs to the chosen subgroup $L_v \leqslant G_v$. We package this data as:
\[
c_{H_v(\Gamma)} \colonequals (c_h)_{h \in H_v(\Gamma)}, \qquad c_{H(\Gamma)} \colonequals (c_h)_{h \in H(\Gamma)}.
\]
We arrive at our fourth reduction
\begin{equation} \label{eqn: boundary character formula fourth reduction}
\big| \Hom_{\theta_{V(\Gamma)}}^{[g_1],\ldots,[g_n]}(\pi_1^{\orb}(\Gamma,r),L)\big| = \sum_{c_{H(\Gamma)}} \big| \Hom^{c_H(\Gamma)}_{\theta_{V(\Gamma)}}(\pi_1^{\orb}(\Gamma,r),L)\big|
\end{equation}
where in each term on the right-hand side we fix the $c_h$ (and not just their $G_v$-conjugacy classes). It remains to compute
\begin{equation} \label{eqn: count fixing class representatives}
\big| \Hom^{c_H(\Gamma)}_{\theta_{V(\Gamma)}}(\pi_1^{\orb}(\Gamma,r),L)\big|.
\end{equation}
We first proceed vertex-by-vertex. Here we use character theory. We wish to count:
\[
\Hom^{c_{H_v(\Gamma)}}(\pi_1^{\orb}(v,r_{H_v(\Gamma)}),L_v) \colonequals \big\{ a_1,b_1,\cdots,a_{g_v},b_{g_v} \in L_v \colon \prod_{i=1}^{g_v} [a_i,b_i] \prod_{h \in H_v(\Gamma)} c_h = 1 \big\}.
\]
Consider the group algebra $\kfield[L_v]$ with standard basis $\{e_g\}_{g \in L_v}$. Define the element: 
\[
K_v = \sum_{x,y \in L_v} e_{[x,y]}.
\] 
As in the proof of \zcref{prop: charformula}, $K_v$ is central and acts on any irreducible representation $V$ of $L_v$ by the scalar 
\[
|L_v|^2/\chi(1)^2
\]
where $\chi$ is the associated irreducible character. Moreover, the count of monodromy representations is equal to the coefficient of the basis element
$e_{(\prod_{h \in H_v(\Gamma)} c_h)^{-1}}$ in the group algebra element $K_v^{g_v} \in \kfield[L_v]$:
\begin{equation} \label{eqn: boundary character formula count via coeff}
\big| \Hom^{c_{H_v(\Gamma)}}(\pi_1^{\orb}(v,r_{H_v(\Gamma)}),L_v) \big| = \operatorname{Coeff}_{e_{(\prod_{h \in H_v(\Gamma)} c_h)^{-1}}}(K_v^{g_v}).
\end{equation}
To calculate this, we use Fourier inversion; we claim that for $s \in L_v$, the coefficient of $e_s$ in an arbitrary group algebra element $z \in \kfield[L_v]$ is given by
\begin{equation} \label{eqn: coefficient via trace}
\operatorname{Coeff}_{e_s}(z) = \dfrac{1}{|L_v|} \sum_{\chi \in Irr(L_v)} \chi(1) Tr(\rho_\chi(z)\rho_\chi(s^{-1}))  
\end{equation}
where $\rho_\chi$ is the corresponding representation. To see this, we write $z=\sum_{g \in L_v} a_g e_g$ and plug it into the above formula to obtain
\[
\sum_{g \in L_v} a_g \dfrac{1}{|L_v|}  \sum_{\chi \in Irr(L_v)} \chi(1) \chi(gs^{-1}).
\] 
From this, the factor
\[
\dfrac{1}{|L_v|} \sum_{\chi \in Irr(L_v)} \chi(1) \chi(gs^{-1})
\]
is equal to the trace of the action of the element $gs^{-1}$ in the regular representation, which is zero unless $gs^{-1}=1$. This collapses the above sum to $a_s$, establishing \eqref{eqn: coefficient via trace}. Applying this back to \eqref{eqn: boundary character formula count via coeff} gives:
\[
\big| \Hom^{c_{H_v(\Gamma)}}(\pi_1^{\orb}(v,r_{H_v(\Gamma)}),L_v) \big| = \dfrac{1}{|L_v|} \sum_{\chi \in Irr(L_v)} \chi(1) Tr \big( \rho_{\chi}(K_v^{g_v}) \rho_{\chi}(\Pi_{h \in H_v(\Gamma)} c_h) \big).
\]
Since $K_v$ acts on each irreducible representation $\rho_\chi$ diagonally by the scalar $|L_v|^2/\chi(1)^2$ we have 
\[
\big| \Hom^{c_{H_v(\Gamma)}}(\pi_1^{\orb}(v,r_{H_v(\Gamma)}),L_v) \big| = |L_v|^{2g_v-1} \sum_{\chi \in Irr(L_v)}\dfrac{\chi(\prod_{h \in H_v(\Gamma)} c_h)}{\chi(1)^{2g_v-1}}.
\]
M\"obius inversion then gives the number of homomorphisms whose image is equal to $G_v$ as:
\begin{align*}
\big| \Hom_{\theta_v}^{c_{H_v(\Gamma)}}(\pi_1^{\orb}(v,r_{H_v(\Gamma)}),G_v) \big| & = \sum_{L_v \leqslant G_v} \mu(L_v,G_v) \big|\Hom^{c_{H_v(\Gamma)}}(\pi_1^{\orb}(v,r_{H_v(\Gamma)}),L_v)\big| \\
& = \sum_{L_v \leqslant G_v} \mu(L_v,G_v)|L_v|^{2g_v-1} \sum_{\chi \in Irr(L_v)}\dfrac{\chi(\prod_{h \in H_v(\Gamma)} c_h)}{\chi(1)^{2g_v-1}}.
\end{align*}
Combining the choices for the individual vertices therefore gives:
\begin{equation} \label{eqn: boundary character formula vertexwise contribution}
\prod_{v \in V(\Gamma)} \big| \Hom_{\theta_v}^{c_{H_v(\Gamma)}}(\pi_1^{\orb}(v,r_{H_v(\Gamma)}),G_v) \big| = \prod_{v \in V(\Gamma)} \sum_{L_v \leqslant G_v} \mu(L_v,G_v)|L_v|^{2g_v-1} \sum_{\chi \in Irr(L_v)}\dfrac{\chi(\prod_{h \in H_v(\Gamma)} c_h)}{\chi(1)^{2g_v-1}}.
\end{equation}
We will use this vertexwise count to obtain a formula for the desired count \eqref{eqn: count fixing class representatives}. We must amend \eqref{eqn: boundary character formula vertexwise contribution} in two ways:
\begin{enumerate}
    \item \textbf{Impose the final two group relations.} For an edge $e \in T$ (respectively an edge $e \in E(\Gamma) \setminus T$) we must exclude choices of representatives such that
    \[ c_{\vec{e}} \qquad \text{and} \qquad (c_{\cev{e}})^{-1} \]
    are not equal (respectively are not conjugate inside $L$).
    \item \textbf{Add the choices of $d_e$.} For each $e \in E(\Gamma) \setminus T$ we must multiply by the choices of possible elements $d_e \in L$ that conjugate $c_{\vec{e}}$ to $(c_{\cev{e}})^{-1}$.
\end{enumerate}

Recall from \eqref{eqn: definition indicator function} and \eqref{eqn: definition centraliser counting function} the definitions of the indicator function $I$ and the counting function $\gl_L$. These functions precisely impose the changes outlined above, integrating information from the different vertices across the edges of the graph. We arrive at:
\begin{equation} \label{eqn: boundary character formula fifth reduction}
\big|\, \Hom^{c_{H_v(\Gamma)}}_{\theta_{V(\Gamma)}}(\pi_1^{\orb}(\Gamma,r),L) \big|
 = \bigg(\prod_{v \in V(\Gamma)} \big| \Hom_{\theta_v}^{c_{H_v(\Gamma)}}(\pi_1^{\orb}(v,r_{H_v(\Gamma)}),G_v) \big| \bigg) \prod_{e \in T} I(c_{\vec{e}},c_{\cev{e}}) \prod_{e \in E(\Gamma) \setminus T} \gl_{L}(c_{\vec{e}},c_{\cev{e}})
\end{equation}
where we note that the product over vertices is given by the character formula \eqref{eqn: boundary character formula vertexwise contribution}.

Working forward again through the sequence of reductions, we combine \eqref{eqn: boundary character formula first reduction}, \eqref{eqn: boundary character formula second reduction}, \eqref{eqn: boundary character formula third reduction}, \eqref{eqn: boundary character formula fourth reduction}, \eqref{eqn: boundary character formula fifth reduction} to arrive at the desired formula.
\end{proof}

%\begin{comment}
\begin{example} \label{example: char formula boundary} As in \zcref{example: char formula global}, let $g \geqslant 2$ and $n=0$ and consider the global type:
\[ \Theta = [G=S_3 \leqslant S_3].\]
Let  $\Gamma$ be the graph consisting of a single vertex $v$ of genus $g-1$ supporting a single loop $e$ (this corresponds to the boundary divisor $\Mcalbar_\Gamma = \delta_0 \subseteq \Mcalbar_{g,0}$). We claim that there are precisely two vertexwise types
\[ \Theta_v = [G_v,c_{\vec{e}},c_{\cev{e}}] \]
whose associated strata intersect nontrivially with $\Mcalbar^{\Theta}_{g,0}(\Bcal S_3)$ and  have nontrivial ramification profile over the loop, namely:
\[
\begin{tikzpicture}
% D1
\draw[fill=black] (0,0) circle[radius=3pt];
\draw (0,0) node[right]{$g\!-\!1$};
\draw (0,0) node[left]{$S_3$};
\draw (1,0) circle[radius=1];
\draw (2,0) node[right]{$(3)$};
\draw (1,-1) node[below]{$[S_3,C_2,C_2]$};

% D2
\draw[fill=black] (4,0) circle[radius=3pt];
\draw (4,0) node[right]{$g\!-\!1$};
\draw (4,0) node[left]{$S_3$};
\draw (5,0) circle[radius=1];
\draw (6,0) node[right]{$(2,1)$};
\draw (5,-1) node[below]{$[S_3,C_3,C_3]$};
\end{tikzpicture}
\]
where $C_2,C_3 \subseteq S_3$ denote respectively  the conjugacy classes of a transposition and a $3$-cycle. 

We justify that these are the only such vertexwise types. Since the ramification is recorded by the cycle shape (\zcref{sec: ramification profile boundary stratum}), we can only have a nontrivial ramification profile over the loop if
the elements $c_{\vec{e}}$ and $c_{\cev{e}}$ are either both $2$-cycles or both  $3$-cycles.

We begin with the case where  $c_{\vec{e}}$ and $c_{\cev{e}}$ are $2$-cycles and suppose for the sake of contradiction that $G_v \neq S_3$. Then $c_{\vec{e}}=c_{\cev{e}}$ and $G_v = \langle c_{\vec{e}} \rangle$ is cyclic of order $2$. But then the equation
\[ d_e c_{\vec{e}} d_e^{-1} = c_{\cev{e}}^{-1} = c_{\vec{e}} \]
implies that $d_e \in G_v$ also, and so $\langle G_v,d_e \rangle \neq S_3$ which means that the associated stratum intersects trivially with the moduli space indexed by the global type $\Theta$.

Next, consider the case where $c_{\vec{e}}$ and $c_{\cev{e}}$ are $3$-cycles and suppose for a contradiction that $G_v \neq S_3$. Then $G_v=A_3=\langle c_{\vec{e}} \rangle$. Since $A_3$ is abelian, the relation
	\[
	\prod_{i=1}^{g-1} [a_i, b_i] \cdot c_{\vec{e}} c_{\cev{e}} = 1
	\]
	implies $c_{\cev{e}}^{-1} = c_{\vec{e}}$. We thus obtain
	\[
	d_e c_{\vec{e}} d_e^{-1} = c_{\cev{e}}^{-1} = c_{\vec{e}}.
	\]
	Since there is no transposition which fixes a $3$-cycle, we have $d_e \in A_3$, and so $\langle G_v,d_e \rangle \neq S_3$ which again means that the associated stratum intersects trivially with the moduli space indexed by the global type $\Theta$. This completes the classification.
	
	We now apply \zcref{prop: boundary char formula} to compute the degrees of these two strata over the boundary divisor $\Mcalbar_\Gamma = \delta_0 \subseteq \Mcalbar_g$. First note that we have:
	\[ G_v = G = S_3. \]
	Given a subgroup $L \leqslant G$, it follows that no vertex image list is $L$-realisable unless $L=G$ (in which case all vertex image lists are $L$-realisable). We therefore set $L=S_3$ and replace $\operatorname{Real}(\Theta_v,L)$ by $\Theta_v$. We note further that since $G_v=S_3$, $\Theta_v$ in fact has a single representative $\theta_v$. We thus obtain
	\[
\deg \big( \Mcalbar^{\Theta}_{(\Gamma,\Theta^i_{V(\Gamma)})}(\Bcal S_3) \to \delta_0 \big) = \dfrac{1}{6 \cdot \lcm(C_i)} \sum_{L_v \leqslant S_3} \sum_{c_{\vec{e}},c_{\cev{e}} \in C_i \cap L_v} |C_{S_3}(c_{\vec{e}})| \cdot \mu(L_v,S_3)| \cdot |L_v|^{2g-3} \sum_{\chi \in Irr(L_v)} \dfrac{\chi(c_{\vec{e}} c_{\cev{e}})}{\chi(1)^{2g-3}}
	\]  
	where $\lcm(C_i)=i$, recording whether $C_i$ is a $2$-cycle or a $3$-cycle.\medskip
	
\noindent \underline{\smash{Case $1$:}} $C_2$. The subgroups $L_v \leqslant S_3$ which intersect the conjugacy class nontrivially are the three cyclic subgroups of order $2$, and $S_3$ itself. Beginning with the former, $c_{\vec{e}}=c_{\cev{e}}$ must both equal the generator of the group. In particular $\chi(c_{\vec{e}} c_{\cev{e}}) = \chi(1)$. Moreover $\mu(L_v,S_3) = -1$ and $|C_{S_3}(c_{\cev{e}})|=2$. These three subgroups taken together therefore contribute:
\[ \dfrac{3 \cdot 2 \cdot (-1)}{6 \cdot 2} \cdot 2^{2g-3} \sum_{\chi \in \Irr(\mu_2)} \dfrac{1}{\chi(1)^{2g-4}} = -2^{2g-3}.\]
On the other hand, when the subgroup is $L_v=S_3$ there are $3$ independent choices for each of $c_{\vec{e}}$ and $c_{\cev{e}}$. If they are equal then $\chi(c_{\vec{e}} c_{\cev{e}})=\chi(1)$, while if they are distinct then $\chi(c_{\vec{e}}c_{\cev{e}}) = 1,1,-1$ for the three irreducible characters of $S_3$. We thus obtain the contribution:
\[
\dfrac{6^{2g-3} \cdot 2}{6 \cdot 2} \left( 3 \cdot (1+1+\dfrac{1}{2^{2g-4}}) + 6 \cdot (1+1+\dfrac{-1}{2^{2g-3}} ) \right).
\]
Combining these two contributions, we obtain the all-genus formula:
\[ \deg \big( \Mcalbar^{\Theta}_{(\Gamma,\Theta^1_{V(\Gamma)})}(\Bcal S_3) \to \delta_0 \big) = 2^{2g-3} (3^{2g-2} - 1).\]
Specialising to $g=2$ gives the degree as $16$. This will be used in \zcref{sec: degree 3 covers non-cyclic}.\medskip

\noindent \underline{\smash{Case 2:}} $C_3$. The subgroups $L_v \leqslant S_3$ which intersect the conjugacy class nontrivially are $A_3$ and $S_3$. The argument proceeds similarly to the previous case. The final result is:
\[ \deg \big( \Mcalbar^{\Theta}_{(\Gamma,\Theta^2_{V(\Gamma)})}(\Bcal S_3) \to \delta_0 \big) = 3^{2g-4} (2^{2g-1} - 2).\]
Specialising to $g=2$ gives the degree as $6$. This will be used in \zcref{sec: degree 3 covers non-cyclic}.
\end{example}
 
\subsection{Splitting formalism}
We now develop a splitting formalism for the vertexwise boundary strata. This will be used in \zcref{sec: algorithm} during the recursive algorithm which calculates the tautological projection of the Prym class.

Fix a global type $\Theta$ and a vertexwise $S_d$-graph $(\Gamma,\Theta_{V(\Gamma)})$ indexing a vertexwise boundary stratum (\zcref{def: boundary strata vertex decorated graph 2}):
\[ \Mcalbar_{(\Gamma,\Theta_{V(\Gamma)})}^{\Theta}(\Bcal S_d) \subseteq \Mcalbar_{g,n}^{\Theta}(\Bcal S_d). \]
We begin by defining certain sets of group homomorphisms whose cardinalities appear in the splitting formalism. These can be computed easily in any particular example, and simplify dramatically in certain cases (see e.g. \zcref{sec: double covers}).

\begin{definition} \label{def: group homs for v} Given $v \in V(\Gamma)$ recall the subgroup $\pi_1^{\orb}(v,r_{H_v(\Gamma)}) \leqslant \pi_1^{\orb}(\Gamma,r)$ from \zcref{def: vertex subgroup}. Analogously to \zcref{sec: degree target map global} we define
\[ \Hom^{\Theta_v}(\pi_1^{\orb}(v,r_{H_v(\Gamma)}),S_d) \]
to be the set of group homomorphisms $\rho_v \colon \pi_1^{\orb}(v,r_{H_v(\Gamma)}) \to S_d$ such that the image list
\[ (\rho_v(\pi_1^{\orb}(v,r_{H_v(\Gamma)})), [\rho_v(\gamma_h)] \mid h \in H_v(\Gamma)) \]
belongs to the equivalence class $\Theta_v$. This means in particular that the image of $\rho_v$ is conjugate in $S_d$ to the subgroup $G_v$.
\end{definition}

\begin{definition} An \textbf{automorphism} of a stable graph $\Gamma$ consists of bijections $V(\Gamma) \to V(\Gamma)$ and $H(\Gamma) \to H(\Gamma)$ which commute with the root map and the involution. We write $\Aut(\Gamma)$ for the automorphism group.
\end{definition}

We are now ready to state the splitting formalism.

\begin{proposition} \label{prop: splitting} Fix a global type $\Theta$ and a vertexwise $S_d$-graph $(\Gamma,\Theta_{V(\Gamma)})$. Consider the associated boundary stratum from \zcref{def: boundary strata vertex decorated graph 2}:
\[ \Mcalbar_{(\Gamma,\Theta_{V(\Gamma)})}^{\Theta}(\Bcal S_d) \subseteq \Mcalbar_{g,n}^{\Theta}(\Bcal S_d) \]
There is then a recursive correspondence
\begin{equation} \label{eqn: splitting correspondence}
\begin{tikzcd}[column sep=tiny]
& \Mcaltilde_{(\Gamma,\Theta_{V(\Gamma)})}^{\Theta}(\Bcal S_d) \ar[rd] \ar[ld] & \\
\prod_{v \in V(\Gamma)} \Mcalbar_{g_v,H_v(\Gamma)}^{\Theta_v}(\Bcal S_d) && \Mcalbar_{(\Gamma,\Theta_{V(\Gamma)})}^{\Theta}(\Bcal S_d) 
\end{tikzcd}
\end{equation}
where the $v$th factor in the bottom-left is the moduli space corresponding to $\Theta_v$, viewed as a global type (\zcref{def: decomposition by type}). We write this correspondence schematically as:
\[
\begin{tikzcd} 
\prod_{v \in V(\Gamma)} \Mcalbar_{g_v,H_v(\Gamma)}^{\Theta_v}(\Bcal S_d) \ar[r,decorate,decoration={zigzag, amplitude=1mm, segment length=3.25mm},"g" {yshift=0.1cm}] & \Mcalbar^{\Theta}_{(\Gamma,\Theta_{V(\Gamma)})}(\Bcal S_d). 
\end{tikzcd}
\]
The overall degree of the correspondence is equal to
\[ \deg(g) = D^{\Theta}_{(\Gamma,\Theta_{V(\Gamma)})} \colonequals \dfrac{|\Aut(\Gamma)| \cdot \prod_{e \in E(\Gamma)} \lcm(m_e) \cdot \prod_{v \in V(\Gamma)} \big| \Hom^{\Theta_v}(\pi_1^{\orb}(v,r_{H_v(\Gamma)}),S_d) \big|}{(d!)^{|V(\Gamma)|-1} \cdot \big| \Hom^{\Theta}_{\Theta_{V(\Gamma)}}(\pi_1^{\orb}(\Gamma,r),S_d) \big|} \]
where the sets of monodromy representations are given in Definitions~\ref{def: group homs for v}~and~\ref{def: group homs for VGamma}.
\end{proposition}

\begin{proof}
We begin with a conceptual observation. Comparing Definitions~\ref{def: boundary strata vertex decorated graph 1}~and~\ref{def: boundary strata vertex decorated graph 2} we see that in general the (restricted) vertexwise boundary stratum is a strict subset of its unrestricted counterpart:
\[ \Mcalbar^{\Theta}_{(\Gamma,\Theta_{V(\Gamma)})}(\Bcal S_d) \subsetneq \Mcalbar_{(\Gamma,\Theta_{V(\Gamma)})}(\Bcal S_d) \]
This is because the latter locus includes pieces for which the associated monodromy group is not conjugate to $G$. The general theory of orbifold stable maps produces a splitting formalism for the latter locus, which we will then refine to produce our desired splitting formalism for the former.

The splitting formalism for orbifold stable maps proceeds by fibring over the rigidified inertia stack of the target \cite[Section~3]{AbramovichGraberVistoli}. For $\Bcal S_d$ this takes the following form \cite[Section~4.2]{AbramovichLectures}:
\[ \overline{\Ical}(\Bcal S_d) = \coprod_{[\omega]} \Bcal (C_{S_d}(\omega)/\langle \omega \rangle).\]
The union is over the conjugacy classes $[\omega]$ of $S_d$, with $C_{S_d}(\omega)$ and $\langle \omega \rangle$ denoting respectively the centraliser of, and the cyclic subgroup generated by, an auxiliary representative $\omega \in [\omega]$. 

Fix now a vertexwise $S_d$-graph $(\Gamma,\Theta_{V(\Gamma)})$. We apply the decomposition in \zcref{def: decomposition by type} to each vertex. For $v \in V(\Gamma)$ we thus consider the clopen substack
\[ \Mcalbar_{g_v,H_v(\Gamma)}^{\Theta_v}(\Bcal S_d) \subseteq \Mcalbar_{g_v,H_v(\Gamma)}(\Bcal S_d)\]
associated to the vertex type $\Theta_v$, viewed as a global type with $g=g_v$ and $n=|H_v(\Gamma)|$. The general splitting formalism for orbifold stable maps \cite[Section~5]{AbramovichGraberVistoli} pulled back to these substacks gives rise to the following diagram
\begin{equation} \label{eqn: essential splitting fibre diagram}
\begin{tikzcd}
\Mcalbar_{(\Gamma,\Theta_{V(\Gamma)})}(\Bcal S_d) & \Mcaltilde_{(\Gamma,\Theta_{V(\Gamma)})}(\Bcal S_d) \ar[l,"h" above] \ar[r,"b"] \ar[d] \ar[rd,phantom,"\square"] & \prod_{v \in V(\Gamma)} \Mcalbar_{g_v,H_v(\Gamma)}^{\Theta_v}(\Bcal S_d) \ar[d] \\
& \prod_{e \in E(\Gamma)} \Bcal( C_{S_d}(g_e)/\langle g_e \rangle ) \ar[r] & \prod_{\vec{e} \in \vec{E}(\Gamma)} \Bcal ( C_{S_d}(g_{\vec{e}})/ \langle g_{\vec{e}} \rangle)
\end{tikzcd}
\end{equation}
where $\Mcaltilde_{(\Gamma,\Theta_{V(\Gamma)})}(\Bcal S_d)$ is defined as the fibre product, and $\Mcalbar_{(\Gamma,\Theta_{V(\Gamma)})}(\Bcal S_d)$ is the unrestricted vertexwise boundary locus from \zcref{def: boundary strata vertex decorated graph 1}. In the bottom-right entry, $g_{\vec{e}} \in S_d$ is obtained by choosing representatives $\theta_v$ of $\Theta_v$ and $g_{\vec{e}}$ of $[g_{\vec{e}}]$ (where $v \in V(\Gamma)$ is the root of $\vec{e} \in H(\Gamma)$). In the bottom-left entry, we arbitrarily choose an orientation $\vec{e}$ of $e$ and let $g_e = g_{\vec{e}}$. Up to isomorphism, the bottom-left entry and the bottom morphism do not depend on this choice. The gluing maps in the above diagram satisfy
\begin{align*}
\deg h & = |\Aut(\Gamma)|, \\[0.2cm]
\deg b & = \prod_{e \in E(\Gamma)} |C_{S_d}(\omega_e)/\langle \omega_e \rangle | = \prod_{e \in E(\Gamma)} \dfrac{|\Aut(m_e)| \cdot \prod_{i=1}^{l_e} m_{ei}}{\lcm(m_{e1},\ldots,m_{el_e})},
\end{align*}
where $m_e$ is the ramification profile over the edge (\zcref{sec: ramification profile boundary stratum}). The gluing map fit into a correspondence
\begin{equation} \label{eqn: essential splitting first correspondence}
\begin{tikzcd}
& \Mcaltilde_{(\Gamma,\Theta_{V(\Gamma)})}(\Bcal S_d) \ar[ld,"b" above] \ar[rd,"h"] & \\
\prod_{v \in V(\Gamma)} \Mcalbar_{g_v,H_v(\Gamma)}^{\Theta_v}(\Bcal S_d) && \Mcalbar_{(\Gamma,\Theta_{V(\Gamma)})}(\Bcal S_d).
\end{tikzcd}
\end{equation}
We now restrict to covers of type $\Theta$ by replacing $\Mcalbar_{(\Gamma,\Theta_{V(\Gamma)})}(\Bcal S_d)$ by the clopen substack (\zcref{def: boundary strata vertex decorated graph 2})
\[ \Mcalbar_{(\Gamma,\Theta_{V(\Gamma)})}^{\Theta}(\Bcal S_d) \hookrightarrow \Mcalbar_{(\Gamma,\Theta_{V(\Gamma)})}(\Bcal S_d) \]
and taking the preimage of this substack under $h$. This produces a new correspondence:
\begin{equation} \label{eqn: essential splitting second correspondence}
\begin{tikzcd}
& \Mcaltilde_{(\Gamma,\Theta_{V(\Gamma)})}^{\Theta}(\Bcal S_d) \ar[ld] \ar[rd] & \\
\prod_{v \in V(\Gamma)} \Mcalbar_{g_v,H_v(\Gamma)}^{\Theta_v}(\Bcal S_d) && \Mcalbar_{(\Gamma,\Theta_{V(\Gamma)})}^{\Theta}(\Bcal S_d) 
\end{tikzcd}
\end{equation}
which we write schematically as:
\[
\begin{tikzcd}
\prod_{v \in V(\Gamma)} \Mcalbar_{g_v,H_v(\Gamma)}^{\Theta_v}(\Bcal S_d) \ar[r,decorate,decoration={zigzag, amplitude=1mm, segment length=3.25mm},"g" yshift={0.1cm}] & \Mcalbar_{(\Gamma,\Theta_{V(\Gamma)})}^{\Theta}(\Bcal S_d).
\end{tikzcd}
\]
To calculate the degree $\deg(g)$ of this correspondence, we consider the commuting diagram
\[
\begin{tikzcd}
\prod_{v \in V(\Gamma)} \Mcalbar_{g_v,H_v(\Gamma)}^{\Theta_v}(\Bcal S_d) \ar[d,"q"] \ar[r,decorate,decoration={zigzag, amplitude=1mm, segment length=3.25mm},"g" yshift={0.1cm}] & \Mcalbar_{(\Gamma,\Theta_{V(\Gamma)})}^{\Theta}(\Bcal S_d) \ar[d,"p"] \\
\prod_{v \in V(\Gamma)} \Mcalbar^{\tw}_{g_v,r_{H_v(\Gamma)}} \ar[r,"f"] & \Mcalbar^{\tw}_{(\Gamma,r)}.
\end{tikzcd}
\]
from which we obtain:
\begin{equation} \label{eqn: degree of g in terms of other maps} \deg(g) = \dfrac{\deg(f) \deg(q)}{\deg(p)}.\end{equation}
We first compute $\deg(f)$. Consider the commuting diagram
\[
\begin{tikzcd}
\prod_{v \in V(\Gamma)} \Mcalbar^{\tw}_{g_v,r_{H_v(\Gamma)}} \ar[r,"f"] \ar[d,"1" blue] & \Mcalbar^{\tw}_{(\Gamma,r)} \ar[d,"\prod_{e \in E(\Gamma)} r_e^{-1}" blue] \\
\prod_{v \in V(\Gamma)} \Mcalbar_{g_v,H_v(\Gamma)} \ar[r,"\left|\Aut(\Gamma)\right|" blue] & \Mcalbar_\Gamma
\end{tikzcd}
\]
with the degrees of the maps indicated in blue. Combining with \eqref{eqn: group order as lcm} we obtain:
\[ \label{eqn: splitting proof deg of f} \deg(f) = |\Aut(\Gamma)| \prod_{e \in E(\Gamma)} \lcm(m_e).\]
We now express $\deg(q)$ and $\deg(p)$ via counts of monodromy representations of a particular type, as given in Definitions~\ref{def: group homs for v}~and~\ref{def: group homs for VGamma}. By \zcref{lem: degree target map global} and \zcref{lem: degree target map vertexwise stratum} we indeed have:
\[ \deg(q) = \dfrac{\prod_{v \in V(\Gamma)} \big| \Hom^{\Theta_v}(\pi_1^{\orb}(v,r_{H_v(\Gamma)}),S_d) \big|}{(d!)^{|V(\Gamma)|}}, \qquad \deg(p) = \dfrac{\big| \Hom^{\Theta}_{\Theta_{V(\Gamma)}}(\pi_1^{\orb}(\Gamma,r),S_d) \big|}{d!}.\]
Plugging back into \eqref{eqn: degree of g in terms of other maps} we obtain:
\[ \deg(g) = \dfrac{|\Aut(\Gamma)| \cdot \prod_{e \in E(\Gamma)} \lcm(m_e) \cdot \prod_{v \in V(\Gamma)} \big| \Hom^{\Theta_v}(\pi_1^{\orb}(v,r_{H_v(\Gamma)}),S_d) \big|}{(d!)^{|V(\Gamma)|-1} \cdot \big| \Hom^{\Theta}_{\Theta_{V(\Gamma)}}(\pi_1^{\orb}(\Gamma,r),S_d) \big|} \]
as required.
\end{proof}

\subsection{Special case: double covers} \label{sec: double covers} The classical case is $d=2$, where many simplifications arise due to the fact that $S_2$ is abelian. The global type $\Theta$ coincides with the global image list $\theta$ and, since all conjugacy classes are singletons, it takes the form 
\[ \Theta = \theta = (G,g_1,\ldots,g_n) \]
where $G \leqslant S_2$ and $g_i \in G$. We consider the interesting case $G=S_2$, and further assume (for notational simplicity) that every marking has non-trivial associated group element, i.e. that $g_i = (12)$ for all $i \in [n]$. It follows from Riemann--Hurwitz that $n$ must be even, otherwise the moduli space is empty.

We first compute the number of monodromy representations
\[ \big| \Hom^{\Theta}(\pi_1(S_{g,n}),S_2) \big|. \]
The images of $\alpha_i,\beta_i$ may be chosen freely. The only constraint is to ensure that the monodromy representation is surjective, so that the induced global type $\Theta_{P \to \Bcal}$ coincides with the chosen global type $\Theta$. We conclude from \zcref{lem: degree target map global} that:
\[ \deg \big( \Mcalbar_{g,n}^{\Theta}(\Bcal S_2) \to \Mcalbar_{g,n} \big) = \begin{cases}(2^{2g}-1)/2 \qquad \quad & \text{$n=0$,} \\ 2^{2g-1} \qquad  \ \ & \text{$n > 0$.} \\ \end{cases} \]

We now investigate vertexwise boundary strata (\zcref{sec: vertexwise strata}). With a view to future calculations we focus on divisors, specifically divisors whose edge has a non-trivial associated group element. There are two classes to consider
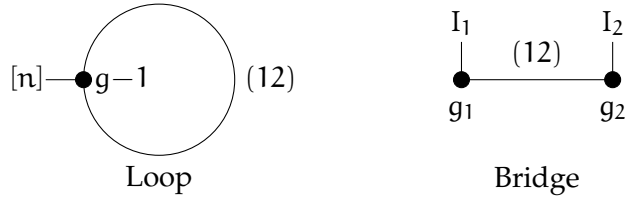
\[
\begin{tikzpicture}

% Loop locus
\draw[fill=black] (0,0) circle[radius=3pt];
\draw (0,0) -- (-0.5,0);
\draw (-0.75,0) node{$[n]$};
\draw (0,0) node[right]{$g\!-\!1$};
\draw (1,0) circle[radius=1];
\draw (2,0) node[right]{$(12)$};
\draw (1,-1) node[below]{$\text{Loop}$};

% Bridge locus
\draw[fill=black] (5,0) circle[radius=3pt];
\draw (5,-0.15) node[below]{$g_1$};
\draw (5,0) -- (5,0.5);
\draw (5,0.75) node{$I_1$};
\draw (5,0) -- (7,0);
\draw (6,0) node[above]{$(12)$};
\draw[fill=black] (7,0) circle[radius=3pt];
\draw (7,-0.15) node[below]{$g_2$};
\draw (7,0) -- (7,0.5);
\draw (7,0.75) node{$I_2$};
\draw (6,-1) node[below]{$\text{Bridge}$};

\end{tikzpicture}
\]
where in the bridge locus, $g_1+g_2=g$ and $I_1 \sqcup I_2 = [n]$ with each $I_j$ necessarily consisting of an odd number of markings. Each of these loci is indexed by a vertexwise $S_d$-graph
\[ (\Gamma,\Theta_{V(\Gamma)}) \]
with $G_v=S_2$ for all vertices $v \in V(\Gamma)$. In both cases the edge partition is $m_e = (2) \vdash 2$.

We begin with the loop locus, and first determine its degree over the associated boundary divisor in the moduli space of stable curves. For the monodromy representation we have a free choice for the images of $\alpha_1,\beta_1,\ldots,\alpha_{g-1},\beta_{g-1},\delta_e$, giving:
\[ \big| \Hom^{\Theta}_{\Theta_{V(\Gamma)}}(\pi_1^{\orb}(\Gamma,r),S_2) \big| = 2^{2g-1}. \]
Applying \zcref{lem: degree target map vertexwise stratum} and noting that $d!=\lcm(m_e)=2$, we obtain:
\[ \deg \big( \Mcalbar^{\Theta}_{(\Gamma,\Theta_{V(\Gamma)})}(\Bcal S_2) \to \Mcalbar_\Gamma \big) = \dfrac{2^{2g-1}}{2 \cdot 2} = 2^{2g-3}.\]
Turning to the splitting formalism, we compute
\begin{align*}
\big| \Hom^{\Theta_{v_0}}(\pi_1^{\orb}(v_0,r_{H_{v_0}(\Gamma)}),S_2) \big| & = 2^{2g-2}, \\[0.2cm]
\big| \Hom_{\Theta_{V(\Gamma)}}^{\Theta}(\pi_1^{\orb}(\Gamma,r),S_2) \big| & = 2^{2g-1},
\end{align*}
with the difference accounted for by the free choice for the image of $\delta_e$. Noting that $|\Aut(\Gamma)| = \lcm(m_e)=2$, it follows from \zcref{prop: splitting} that the degree of the splitting correspondence \eqref{eqn: splitting correspondence} is:
\[ \deg \big(\Mcalbar^{[S_2,(12)^{n+2}]}_{g-1,n+2}(\Bcal S_2) \rightsquigarrow \Mcalbar_{(\Gamma,\Theta_{V(\Gamma)})}^{\Theta}(\Bcal S_2) \big) = \dfrac{2 \cdot 2 \cdot 2^{2g-2}}{2^{2g-1}} = 2.\]

We now turn the bridge locus. For the monodromy representation we have a free choice of the images of the group elements
\[ \alpha^{(v_1)}_1,\ldots,\beta^{(v_1)}_{g_1},\alpha^{(v_2)}_1,\ldots,\beta^{(v_2)}_{g_2}\]
which gives:
\[ \big| \Hom^{\Theta}_{\Theta_{V(\Gamma)}}(\pi_1^{\orb}(\Gamma,r),S_2) \big| = 2^{2g_1+2g_2} = 2^{2g}. \]
Applying \zcref{lem: degree target map vertexwise stratum} and noting that $d!=\lcm(m_e)=2$ gives:
\[ \deg \big( \Mcalbar^{\Theta}_{(\Gamma,\Theta_{V(\Gamma)})}(\Bcal S_2) \to \Mcalbar_\Gamma \big) = \dfrac{2^{2g}}{2 \cdot 2} = 2^{2g-2}.\]
Turning to the splitting formalism, we compute: 
\begin{align*}
& \prod_{i=1}^2 \big| \Hom^{\Theta_{v_i}}(\pi_1^{\orb}(v_i,r_{H_{v_i}(\Gamma)}),S_2) \big| = 2^{2g_1+2g_2} = 2^{2g}, \\[0.2cm]
& \big| \Hom^{\Theta}_{\Theta_{V(\Gamma)}}(\pi_1^{\orb}(\Gamma,r),S_2) \big| = 2^{2g}.
\end{align*}
Moreover since the $|I_j|$ must be odd we have $n > 0$ and so $|\Aut(\Gamma)|=1$. It follows from \zcref{prop: splitting} that the degree of the splitting correspondence \eqref{eqn: splitting correspondence} is:
\[ \deg \bigg( \prod_{j=1}^2 \Mcalbar^{[S_2,(12)^{|I_j|+1}]}_{g_i,|I_j|+1}(\Bcal S_2) \rightsquigarrow \Mcalbar_{(\Gamma,\Theta_{V(\Gamma)})}^{\Theta}(\Bcal S_2) \bigg) = \dfrac{1 \cdot 2 \cdot 2^{2g}}{2 \cdot 2^{2g}} = 1.\]

\subsection{Summary} \label{sec: covers summary} We recap the key points of \zcref{sec: covers}. Choosing a global type $\Theta$ (\zcref{sec: setup}) determines a moduli space (\zcref{def: decomposition by type}):
\[ \Mcalbar_{g,n}^{\Theta}(\Bcal S_d). \]
This moduli space is stratified into vertexwise boundary strata (\zcref{def: boundary strata vertex decorated graph 2}):
\[ \Mcalbar_{(\Gamma,\Theta_{V(\Gamma)})}^{\Theta}(\Bcal S_d), \]
indexed by vertexwise $S_d$-graphs $(\Gamma,\Theta_{V(\Gamma)})$ (\zcref{def: vertexwise Sd graph}). The degree of the forgetful map
\[ \Mcalbar_{(\Gamma,\Theta_{V(\Gamma)})}^{\Theta}(\Bcal S_d) \to \Mcalbar_\Gamma \]
is given by group-theoretic data (Propositions~\ref{lem: degree target map vertexwise stratum}~and~\ref{prop: boundary char formula}), while the splitting formalism furnishes a recursive correspondence
\[
\begin{tikzcd}
\prod_{v \in V(\Gamma)} \Mcalbar_{g_v,H_v(\Gamma)}^{\Theta_v}(\Bcal S_d) \ar[r,decorate,decoration={zigzag, amplitude=1mm, segment length=3.25mm}] & \Mcalbar_{(\Gamma,\Theta_{V(\Gamma)})}^{\Theta}(\Bcal S_d)
\end{tikzcd}
\]
whose degree is again given by group-theoretic data (\zcref{prop: splitting}).

%%%%%
% SECTION: ABELIAN VARIETIES
%%%%%

\section{Polarised abelian varieties} \label{sec: abelian varieties}

\subsection{Moduli of polarised abelian varieties} \label{sec: moduli of polarised abelian varieties} The theory of the moduli spaces $\Acal_{g,\delta}$ closely parallels that of $\Acal_g$. A good survey is provided in \cite[Section~2]{IribarLopez_NoetherLefschetz}. We outline the basic properties of these moduli spaces and their compactifications, with a focus on their tautological intersection theory.

\subsubsection{Polarisation types} \label{sec: polarisation types} Fix an abelian variety $A$ of dimension $g$. A polarisation on $A$ is the first Chern class $\Theta = c_1(L) \in \operatorname{NS}(A)$ of an ample line bundle. There is an induced morphism
\begin{align*} \lambda_\Theta \colon A & \to A^\vee \\
x & \mapsto t_x^\star L^{-1} \otimes L
\end{align*}
which depends only on $\Theta$ and not on the choice of representative $L$. This morphism is an isogeny, and
\[ \Ker \lambda_\Theta \cong ( \Z/d_1\Z \times \cdots \times \Z/d_g\Z)^2 \]
for a unique vector $\delta = (d_1,\ldots,d_g)$ satisfying $d_i|d_{i+1}$ and referred to as the \textbf{type} of the polarisation. A polarisation is principal if and only if $\lambda_\Theta$ is an isomorphism, i.e. if and only if $\delta = (1,\ldots,1).$

\subsubsection{Moduli spaces with fixed polarisation type} \label{sec: open moduli abelian varieties}
Fix such a $\delta=(d_1,\ldots,d_g)$ with $d_i|d_{i+1}$. There is an associated \textbf{moduli space of polarised abelian varieties of type $\delta$}, denoted
\[ \Acal_{g,\delta}.\]
Analytically this is constructed as a quotient of the usual Siegel upper half-space $\Hcal_g$ by a discrete group depending on $\delta$ \cite[Section~8.1]{BL04}. It is a smooth, connected, non-proper Deligne--Mumford stack of dimension
\[ {g+1 \choose 2}. \]
The case $\delta = (1,\ldots,1)$ produces the usual moduli space $\Acal_g$ of principally polarised abelian varieties, and the generalisations $\Acal_{g,\delta}$ share all of its key properties. These properties can be deduced through the correspondence of finite morphisms
\begin{equation} \label{eqn: level correspondence}
\begin{tikzcd}
& \Acal_{g,\delta}^{\operatorname{lev}} \ar[ld,"\pi_\delta"{xshift=-11pt, yshift=11pt}] \ar[rd,"\varphi_\delta"] & \\
\Acal_{g,\delta} & & \Acal_{g},
\end{tikzcd}
\end{equation}
where $\Acal_{g,\delta}^{\operatorname{lev}}$ is a moduli space of polarised abelian varieties equipped with a level structure. This correspondence arises from a correspondence of finite-index inclusions of discrete subgroups acting on the Siegel upper half-space \cite[Lemma~9]{IribarLopez_NoetherLefschetz}:
\[ G_\delta \supseteq G_\delta[\delta] \subseteq \operatorname{Sp}(2g,\Z). \]
Each of the moduli spaces above supports a universal abelian variety, and these are compatible under pullback (in the case of $\varphi_\delta$ only up to isogeny). The degrees of these finite morphisms are calculated in \cite[Propositions~26~and~27]{IribarLopez_NoetherLefschetz} and will be used in \zcref{sec: proportionality constant} below.

\subsubsection{Compactifications} \label{sec: compact moduli abelian varieties} The theory of compactifications of $\Acal_{g,\delta}$ closely mirrors that of $\Acal_g$, see \cite[Section~2.3]{IribarLopez_NoetherLefschetz}. As discussed above, this moduli space arises as a quotient:
\[\Acal_{g,\delta} = [\Hcal_g/G_\delta]. \]
The work \cite{AMRT} applies in this level of generality, producing a bijection between toroidal compactifications of $\Acal_{g,\delta}$ and decompositions of the usual cone $\Omega_g^{\rt}$ of semidefinite forms with rational null space, admissible with respect to the action of a certain subgroup $\Gamma_\delta \leqslant \GL_g(\Q)$. We use $\Sigma$ to denote such an admissible decomposition, and write
\[ \Acalbar_{g,\delta}^\Sigma \supseteq \Acal_{g,\delta} \]
for the corresponding toroidal compactification. This space supports a universal semiabelian variety
\begin{equation} \label{eqn: universal semiabelian variety}
\begin{tikzcd}
 \Ucal_{g,\delta}^{\Sigma} \ar[r,"q" below] & \Acalbar_{g,\delta}^{\Sigma} \ar[l,"e" above, bend right]	
\end{tikzcd}
\end{equation}
although it is not itself a moduli space of semiabelian varieties, because it also parametrises certain additional data \cite{LopezEtAlNotes}. Some choices of $\Sigma$ also support a fibrewise compactification of $\Ucal_{g,\delta}^\Sigma$, but we will not use this.

\subsection{Hodge bundle and tautological ring}

Fix an admissible decomposition $\Sigma$ inducing a compactification
\[ \Acalbar_{g,\delta}^\Sigma \supseteq \Acal_{g,\delta}.\]
Recall that this supports a universal semiabelian variety \eqref{eqn: universal semiabelian variety}.

\subsubsection{Hodge bundle and lambda classes} \label{sec: Hodge bundle}

\begin{definition} The \textbf{Hodge bundle} on $\Acalbar_{g,\delta}^\Sigma$ is defined as:
\[ \EE \colonequals e^\star \Omega_q. \]
The \textbf{lambda classes} are the Chern classes of the Hodge bundle:
\[ \lambda_i \colonequals c_i(\EE) \in \CH^i(\Acalbar_{g,\delta}^\Sigma). \]
\end{definition}

\begin{remark} If $\Sigma$ is such that there exists a fibrewise compactification
\[ 
\begin{tikzcd}
\Ucal_{g,\delta}^\Sigma \ar[r,hook] \ar[rd,"q" right] & \overline{\Xcal}_{g,\delta}^\Sigma \ar[d,"p"] \\
& \Acalbar_{g,\delta}^\Sigma
\end{tikzcd}
\]
then $\EE = p_\star \Omega_p$, see \cite[Lemma~22]{CMOP_tautologicalProjection}.
\end{remark}

Note that the choice of admissible decomposition $\Sigma$ is not reflected in the notation $\EE$. This is justified by the following lemma. 

\begin{lemma} \label{lem: lambda classes preserved} Consider a refinement of admissible decompositions $\Sigma^\prime \to \Sigma$ and let 
\[ b \colon \Acalbar_{g,\delta}^{\Sigma^\prime} \to \Acalbar_{g,\delta}^\Sigma .\] 
denote the associated toroidal modification. Then we have
\[ b^\star \EE = \EE \]
and so in particular $b^\star \lambda_i = \lambda_i$. 
\end{lemma}

\begin{proof} This follows immediately from \cite[Proposition~I.2.7]{FaltingsChai}  which implies that the universal semiabelian varieties pull back:
\[ \Ucal_{g,\delta}^{\Sigma^\prime} = b^\star \Ucal_{g,\delta}^\Sigma.\qedhere\]
\end{proof}

\subsubsection{Tautological ring}

\begin{definition} \label{def: taut ring} The \textbf{tautological ring} 
\[ R^\star(\Acalbar_{g,\delta}^\Sigma) \subseteq \CH^\star(\Acalbar_{g,\delta}^\Sigma) \]
is the subring generated by the lambda classes.   
\end{definition}

\begin{lemma} Given admissible decompositions $\Sigma_1$ and $\Sigma_2$, there is a canonical isomorphism:
\[ R^\star(\Acalbar_{g,\delta}^{\Sigma_1}) = R^\star(\Acalbar_{g,\delta}^{\Sigma_2}). \]
\end{lemma}

\begin{proof} The decompositions $\Sigma_1$ and $\Sigma_2$ admit a common refinement $\Sigma_{12}$, producing a correspondence:
\[
\begin{tikzcd}
& \Acalbar_{g,\delta}^{\Sigma_{12}} \ar[ld] \ar[rd] & \\
\Acalbar_{g,\delta}^{\Sigma_1} & & \Acalbar_{g,\delta}^{\Sigma_2}.
\end{tikzcd}
\]
The claim then follows from \zcref{lem: lambda classes preserved}.
\end{proof}

Thus the tautological ring is insensitive to the choice of the compactification (of course, this does not hold for the full Chow ring). We therefore drop $\Sigma$ from the notation and simply write
\[ R^\star(\Acalbar_{g,\delta}).\]
In fact, the tautological ring admits a simple presentation:
\begin{theorem} \label{thm: tautological ring presentation} There is an isomorphism of graded rings
\[ R^\star(\Acalbar_{g,\delta}) \cong \QQ[\lambda_1,\ldots,\lambda_g]/ \left((1+\lambda_1+\cdots+\lambda_g)(1-\lambda_1+\cdots+(-1)^g\lambda_g) =1  \right) \]
where $\uplambda_i$ has degree $i$. In particular, as a $\QQ$-vector space $R^\star(\Acalbar_{g,\delta})$ is based by the squarefree monomials in the lambda classes.
\end{theorem}

\begin{proof} This essentially follows from \cite[proof of Theorem~1]{IribarLopez_NoetherLefschetz}, the only caveat being that the cited result deals with the non-compact moduli instead. As explained in \cite[Section~2.3]{IribarLopez_NoetherLefschetz}, the correspondence \eqref{eqn: level correspondence} extends to the compactifications
\begin{equation} \label{eqn: level correspondence compactified}
\begin{tikzcd}
& \Acalbar_{g,\delta}^{\operatorname{lev},\Sigma} \ar[ld,"\pi_\delta"{xshift=-11pt, yshift=11pt}] \ar[rd,"\varphi_\delta"] & \\
\Acalbar_{g,\delta}^\Sigma & & \Acalbar_{g}^\Sigma
\end{tikzcd}
\end{equation}
and these extensions preserve the Hodge bundles. We conclude that
\[ R^\star(\Acalbar_{g,\delta}) = R^\star(\Acalbar_{g}) \]
and the general case then follows from the principally polarised case \cite[Theorem~1.11]{vdG}.
\end{proof}

\subsubsection{Tautological projection} \label{sec: taut projection}

Using the above presentation, one can show that the restriction of the Poincar\'e pairing to the tautological ring remains perfect:         
\begin{align} \label{eqn: pairing on taut} R^k(\Acalbar_{g,\delta}) \times R^{{g+1 \choose 2}-k}(\Acalbar_{g,\delta}) & \to \QQ \\
\nonumber \langle \gamma_1,\gamma_2\rangle  & \colonequals \int_{\Acalbar_{g,\delta}} \gamma_1\gamma_2.
\end{align}
Here we abuse notation and write the integral as being over $\Acalbar_{g,\delta}$ instead of $\Acalbar_{g,\delta}^\Sigma$ because the value does not depend on the choice of $\Sigma$.

Given a class $\gamma \in \CH^k(\Acalbar_{g,\delta}^\Sigma)$ we then consider the operator
\[ \langle \gamma, - \rangle \colon R^{{g+1 \choose 2}-k}(\Acalbar_{g,\delta}) \to \QQ. \]
This corresponds, via the perfect pairing \eqref{eqn: pairing on taut}, to a unique element $\taut(\gamma) \in R^k(\Acalbar_{g,\delta})$ which is referred to as the \textbf{tautological projection} of $\gamma$. We thus obtain a vector space map
\begin{align*}
 \taut \colon \CH^k(\Acalbar_{g,\delta}^\Sigma) & \to R^k(\Acalbar_{g,\delta}) \\
\gamma & \mapsto \taut(\gamma)	
\end{align*}
which constitutes a canonical left inverse to the inclusion of the tautological subring.

\subsubsection{Proportionality constant} \label{sec: proportionality constant} The canonical isomorphism $R^\star(\Acalbar_{g,\delta}) = R^\star(\Acalbar_g)$ masks a subtlety. The socle (top-degree piece) of the tautological ring is $1$-dimensional, with generator
\[ \lambda_1 \cdots \lambda_g \in R^{{g+1 \choose 2}}(\Acalbar_{g,\delta}). \]
Importantly, while we have a natural isomorphism $R^\star(\Acalbar_{g,\delta}) = R^\star(\Acalbar_g)$, the integral of the above generator depends on $\delta$. Examining the correspondence \eqref{eqn: level correspondence compactified} we see that
\[ \int_{\Acalbar_{g,\delta}} \lambda_1 \cdots \lambda_g = \dfrac{\deg \varphi_\delta}{\deg \pi_\delta} \cdot \int_{\Acalbar_g} \lambda_1 \cdots \lambda_g.\]
The ratio $\deg \varphi_\delta/ \deg \pi_\delta$ is calculated in \cite[Proposition~27]{IribarLopez_NoetherLefschetz} while the integral on $\Acalbar_g$ is a classical calculation \cite{Siegel,Harder} (see also \cite[Section~4]{Faber_algorithms}). Combining, we obtain:
\begin{equation}
\int_{\Acalbar_{g,\delta}} \lambda_1 \cdots \lambda_g = \dfrac{1}{\prod_{i=1}^g d_i^{2g+2-4i}}
 \cdot \prod_{1 \leqslant i < j \leqslant g} \prod_{p|d_j/d_i} \dfrac{(1-p^{-2(j-i+1)})}{(1-p^{-2(j-i)})} \cdot \prod_{i=1}^g \dfrac{|B_{2i}|}{4i}
\end{equation}
where $(d_1,\ldots,d_g) = \delta$ is the polarisation type, the product is over primes $p$, and $B_{2i}$ is the Bernoulli number. In particular we make the following simple observation, which will be of relevance when discussing choices of polarisations on Prym varieties (\zcref{rmk: choice of polarisation}).

\begin{lemma} \label{lem: integral on multiple of principal polarisation space} If $\delta=(d,\ldots,d)$ then we have
\[ \int_{\Acalbar_{g,\delta}} \lambda_1 \cdots \lambda_g = \int_{\Acalbar_g} \lambda_1 \cdots \lambda_g.\]	
\end{lemma}

%%%%%%%%%%%%%%%%
% SECTION: PRYM VARIETIS
%%%%%%%%%%%%%%%%

\section{Prym map} \label{sec: Prym}

\noindent Sections~\ref{sec: global G-covers}~and~\ref{sec: abelian varieties} developed, respectively, the theories of the moduli spaces
\[ \Mcalbar_{g,n}^{\Theta}(\Bcal S_d) \qquad \text{and} \qquad \Acalbar_{g,\delta}^{\Sigma}.\]
We now bring these two theories together: the bridge is the Prym map. 

\subsection{Prym varieties} \label{sec: Prym varieties}
We refer to \cite{BorowkaOrtegaNotes} for basics on general Prym varieties. 

\subsubsection{Prym varieties for smooth curves} \label{sec: Prym smooth} Fix an $S_d$-cover $P \to (\Bcal,b_1,\ldots,b_n)$ with $\Bcal$ smooth, and let
\[ f \colon C \to B \]
be the associated degree $d$ cover of smooth curves (\zcref{sec: G-cover to d-cover}). There is an induced norm map
\begin{align*}
\Nm(f) \colon \Jac C & \to \Jac B \\ 
\Sigma_{p \in C} c_p [p] & \mapsto \Sigma_{p \in C} c_p [f(p)]
\end{align*}
which is well-defined because rational functions can be descended using the field norm for the extension $\kfield(C)/\kfield(B)$.

\begin{definition} The \textbf{Prym variety} $\Prym(f)$ is the abelian variety defined equivalently as:
\begin{enumerate}
    \item The identity connected component of $\Ker \Nm(f)$.
    \item The abelian subvariety of $\Jac(C)$ complementary to $f^\star \Jac(B) \subseteq \Jac(C)$.
    \item The dual of the quotient abelian variety $\Jac(C)/f^\star \Jac(B)$.
\end{enumerate}
\end{definition}
There is a natural embedding $\Prym(f) \hookrightarrow \Jac(C)$ and a natural isogeny $f^\star \Jac(B) \times \Prym(f) \to \Jac(C)$. The Prym variety has dimension $\tilde{g} = g_C - g_B$ and there is a short exact sequence of tangent spaces at the identity:
\begin{equation} \label{eqn: ses of tangent spaces Prym} 0 \to T_{\Prym(f),e_f} \to T_{\Jac(C),e_C} \to T_{\Jac(B),e_B} \to 0. \end{equation}
Each marking $b_i \in B$ carries an associated ramification profile $m_i = (m_{i1},\ldots,m_{il_i}) \vdash d$, corresponding to the conjugacy class $[g_i] \subseteq S_d$ encoded as part of the global type $\Theta$. By Riemann--Hurwitz the dimension of $\Prym(f)$ is
\begin{equation} \label{eqn: dim of Prym} \tilde{g} = (d-1)(g-1) + r/2 + (b_0(C)-1) \end{equation}
where $r$ is the degree of the ramification divisor in $C$, given by
\[ r = \sum_{i=1}^n \sum_{k=1}^{l_i} (m_{ij}-1) = \sum_{i=1}^n (d-l_i) \]
and $b_0(C)$ is the number of connected components of $C$.

\begin{remark}[\'Etale double covers] The classical setting for Prym varieties is when $d=2$ and $n=0$, i.e. \'etale double covers
\[ f \colon C \to B. \]
In this setting, the Prym variety may also be described as the image of the endomorphism
\[ (\Id - \uptau^\star) \colon \Jac(C) \to \Jac(C) \]
where $\uptau \colon C \to C$ is the covering involution. It follows that $\Prym(f) \subseteq \Jac(C)$ consists of even-length sums of differences of conjugate points:
\[ \Prym(f) = \left\{ \sum_{i=1}^{2k} (p_i - \uptau(p_i)) \colon k \in \NN, p_i \in C \right\} \subseteq \Jac(C). \]
By contrast, $\Ker \Nm(f)$ also includes odd-length sums of differences of conjugate points, and decomposes into two connected components determined by the parity of the sum. We have $\dim \Prym(f) = \tilde{g}=g-1$.
\end{remark}

\subsubsection{Semiabelian Prym varieties for nodal curves} \label{sec: Prym nodal}

We extend the definition of Prym varieties given in \zcref{sec: Prym smooth} to covers of nodal curves. The following construction appears in \cite[Sections~3~and~5]{BeauvilleAdmissible} in the $d=2$ case, and \cite[Section~2]{LangeOrtegaCompactification} in the general case.

Fix an $S_d$-cover $P \to \Bcal$ and consider the associated degree $d$ cover of nodal curves
\[ f \colon C \to B \]
where we emphasise that $B$ is the coarse space of $\Bcal$. There are semiabelian varieties $\Jac(C)$ and $\Jac(B)$ parametrising line bundles of multidegree zero. Since the cover is admissible it is in particular flat, and so by \cite[(6.5.2)]{EGAII} there is an induced norm map
\[ \Nm(f) \colon \Jac(C) \to \Jac(B).\]
We define the \textbf{Prym variety} $\Prym(f)$ as the identity connected component of $\Ker \Nm(f)$.

This definition is formally indistinguishable from the definition for smooth curves. The following result, and its proof, serves to elucidate the complexities which arise when considering covers of nodal curves.

\begin{proposition} \label{prop: Prym semiabelian} $\Prym(f)$ is a semiabelian variety. Its toric part has dimension $b_1(\Gamma_C)-b_1(\Gamma_B)$ where $\Gamma_C$ and $\Gamma_B$ are the dual graphs of $C$ and $B$.\end{proposition}

\begin{proof} Let $\Gamma_f \colon \Gamma_C \to \Gamma_B$ denoted the induced map on dual graphs, and let $\tilde{f} \colon \tilde{C} \to \tilde{B}$ denote the induced map between the (smooth and typically disconnected) normalisations. Then $\Nm(f)$ fits into the following diagram
\[
\begin{tikzcd}
0 \ar[r] & H_1(\Gamma_C,\Z) \otimes \Gm \ar[d,"\Gamma_{f\star}"] \ar[r] & \Jac(C) \ar[d,"\Nm(f)"] \ar[r] & \Jac(\tilde{C}) \ar[d,"\Nm(\tilde{f})"] \ar[r] & 0 \\
0 \ar[r] & H_1(\Gamma_B,\Z) \otimes \Gm \ar[r] & \Jac(B) \ar[r] & \Jac(\tilde{B}) \ar[r] & 0
\end{tikzcd}
\]
with exact rows. A direct calculation shows that the cokernel of
\[ \Gamma_{f\star} \colon H_1(\Gamma_C,\Z) \to H_1(\Gamma_B,\Z) \]
is $d$-torsion (see \cite[Section~4.3]{BN09} for a proximate argument). This means that the image of $\Gamma_{f\star}$ is a finite-index sublattice of $H_1(\Gamma_B,\Z)$. It follows that after tensoring by $\Gm$ the map
\[ \Gamma_{f\star} \colon H_1(\Gamma_C,\Z) \otimes \Gm \to H_1(\Gamma_B,\Z) \otimes \Gm \]
is surjective, with kernel a finite extension of an algebraic torus $T$ of dimension $b_1(\Gamma_C)-b_1(\Gamma_B)$. Applying the Snake Lemma and restricting to the identity component then gives
\[ 0 \to T \to \Prym(f) \to \Prym(\tilde{f}) \to 0 \]
as required.
\end{proof}

The moduli space of $S_d$-covers of type $\Theta$ thus carries three universal semiabelian varieties
\begin{equation} \label{eqn: three universal semiabelian varieties}
\begin{tikzcd}
\Prym(f) \ar[r] \ar[rd,"q_f" below] & \Jac(C) \ar[r] \ar[d,"q_C"] & \Jac(B) \ar[ld,"q_B"] \\
& \Mcalbar_{g,n}^{\Theta}(\Bcal S_d) \ar[ul, bend left, "e_f"] \ar[ur, bend right,"e_B" right] \ar[u,bend left,"e_C"] 
\end{tikzcd}
\end{equation}
while the short exact sequence of tangent spaces \eqref{eqn: ses of tangent spaces Prym} persists over the boundary:
\begin{equation} \label{eqn: ses of tangent spaces Prym compactified} 0 \to T_{\Prym(f),e_f} \to T_{\Jac(C),e_C} \to T_{\Jac(B),e_B} \to 0.\end{equation}

\subsection{Induced polarisation} \label{sec: Prym polarisation} Fix a cover $f \colon C \to B$ of smooth curves. The Prym variety $\Prym(f)$ embeds naturally into $\Jac(C)$ and thus inherits a polarisation
\begin{equation} \label{eqn: Prym polarisation first appearance} \lambda_{\Prym(f)}  \colon \Prym(f) \to \Prym(f)^\vee \end{equation}
by pulling back the theta divisor from $\Jac(C)$. This polarisation is rarely principal: in this section we determine its type $\delta$ in terms of the monodromy data of the cover $f$.

The key results are Proposition~\ref{prop: Prym polarisation via pullback on Jacobian}, which expresses the polarisation type in terms of an inclusion of finite groups, and Proposition~\ref{prop: polarisation kernel in terms of monodromy} which expresses that inclusion in terms of the monodromy data associated to the cover. Explicit calculations are given in \zcref{sec: calculating the polarisation type}. The consequences for our study are discussed in \zcref{sec: Prym map}.

\subsubsection{Background: duality and the Weil pairing} We begin by recalling the Weil pairing. Given a group $G$, we write
\begin{align*}
G^{D} & \colonequals \Hom(G,\Gm),\\
G^\vee & \colonequals \Ext^1(G,\Gm).
\end{align*}
These operations give the correct notions of duality for finite groups and abelian varieties respectively, and give rise to long exact sequences. In particular given an isogeny
\[ 0 \to G \to A \to A^\prime \to 0 \]
of abelian varieties, we obtain a \textbf{dual isogeny}
\[ 0 \to G^D \to (A^\prime)^{\vee} \to A^\vee \to 0 \]
using $\Hom(A,\Gm)=0$ and $\Ext^1(G,\Gm) = 0$. Given an abelian variety $A$, we consider the isogeny relating multiplication by $d$ to the $d$-torsion: 
\[ 0 \to A[d] \to A \xrightarrow[]{d} A \to 0. \]
The dual isogeny is
\[ 0 \to A[d]^D \to A^{\vee} \xrightarrow[]{d} A^{\vee} \to 0 \]
and it follows that the torsion of the dual is the dual of the torsion:
\[ A^{\vee}[d] = A[d]^D. \]
We thus obtain a canonical pairing 
\begin{align*}
A[d] \times A^{\vee}[d] = A[d] \times \Hom(A[d],\Gm) \to & \ \Gm \\
\langle x,y \rangle \colonequals & \ y(x)
\end{align*}
which factors through $\mu_d \leqslant \Gm$. Given now a polarisation
\[ \lambda \colon A \to A^{\vee} \]
we can compose with $\lambda$ to obtain a pairing on $A[d]$, called the \textbf{Weil pairing}:
\begin{align*}
A[d] \times A[d] & \to \mu_d \\
\left \langle x ,y \right \rangle_\lambda & \colonequals \left \langle x, \lambda(y) \right \rangle.
\end{align*}

\subsubsection{Polarisation kernel via pullback on Jacobians}
Consider a degree $d$ cover of smooth curves:
\[ f \colon C \to B.\]
We introduce shorthand notation for the associated abelian varieties:
\begin{align*}
    J_C & \colonequals \Jac(C), \\
    J_B & \colonequals \Jac(B), \\
    P_f & \colonequals \Prym(f).
\end{align*}
We let $\lambda_C$ and $\lambda_B$ denote the canonical principal polarisations (theta divisors) on $J_C$ and $J_B$. As explained above, the induced polarisation \eqref{eqn: Prym polarisation first appearance} on $P_f$ is defined by pullback:
\[ \lambda_{P_f} = \lambda_C|_{P_f}.\]
To describe the type of this polarisation, we must describe the finite group $ \Ker \lambda_{\Prym(f)}$.

\begin{proposition} \label{prop: Prym polarisation via pullback on Jacobian} Consider the pullback $f^\star \colon J_B \to J_C$. Then $\Ker(f^\star) \leqslant J_B[d]$ and:
\[ \Ker(\lambda_{\Prym(f)}) = \Ker(f^\star)^\perp / \Ker(f^\star) \]
where the orthogonal complement is taking with respect to the Weil pairing on $J_B[d]$.
\end{proposition}

\begin{proof} The statement can be found in \cite[Proposition~3.2.9]{LangeRodriguezBook}, but we give a self-contained proof. The isogeny $\lambda_{P_f}$ is given by the composite:
\[ P_f \to J_C \xrightarrow[]{\lambda_C} J_C^\vee \to P_f^\vee \]
and from this we can describe the kernel as an intersection of subgroups of $J_C$:
\begin{equation} \label{eqn: ker lambda v1} \Ker \lambda_{P_f} = P_f \cap \lambda_C^{-1} \left( \Ker(J_C^\vee \to P_f^\vee) \right). \end{equation}
We describe $\Ker (J_C^\vee \to P_f^\vee)$. Consider $\Ker \Nm_f$ and the finite group indexing its connected components:
\begin{align*}
K_f & \colonequals \Ker \Nm_f, \\
F_f & \colonequals K_f/P_f = \Ker \Nm_f/\Prym(f).
\end{align*}
There are associated short exact sequences:
\begin{align}
\label{eqn: Prym polarisation equation group of components}   0 \to P_f \to K_f \to F_f \to 0, \\
\label{eqn: Prym polarisation equation ker norm}  0 \to K_f \to J_C \xrightarrow{\Nm_f} J_B \to 0.
\end{align} 
Dualising \eqref{eqn: Prym polarisation equation group of components} gives $K_f^D = F_f^D$ and $K_f^{\vee} = P_f^{\vee}$ which we plug into the dual of \eqref{eqn: Prym polarisation equation ker norm} to obtain
\begin{equation} \label{eqn: exact sequence Norm f dual}
    0 \to F_f^D \to J_B^{\vee} \xrightarrow[]{\Nm_f^\vee} J_C^{\vee} \to P_f^{\vee} \to 0 
\end{equation}
from which we conclude:
\begin{equation} \label{eqn: ker as image of norm} \Ker(J_C^\vee \to P_f^\vee) = \Nm_f^\vee(J_B^\vee). \end{equation}
The maps $\Nm_f$ and $f^\star$ are dual in the sense that the following diagram commutes
\begin{equation} \label{eqn: square compatbility norm and f star}
\begin{tikzcd}
		J_B \ar[r,"\lambda_B"] \ar[d,"f^\star"] & J_B^\vee \ar[d,"\Nm_f^\vee"] \\
		J_C \ar[r,"\lambda_C"] & J_C^\vee
\end{tikzcd}
\end{equation}
and it follows that
\[ \lambda_C^{-1} \Nm_f^\vee(J_B^\vee) = f^\star J_B. \]
Combining with \eqref{eqn: ker as image of norm} and \eqref{eqn: ker lambda v1} we obtain:
\begin{equation} \label{eqn: ker lambda v2} \Ker(\lambda_{P_f}) = P_f \cap f^\star J_B. \end{equation}

We now realise this as a subgroup of the $d$-torsion. Pulling back via $f^\star$ and then pushing forward via $\Nm_f$ has the effect of multiplying by $d$: 
\begin{equation} \label{eqn: f star followed by Nm} \Nm_f(f^\star \gamma) = d \cdot \gamma. \end{equation} 
It follows that $\Nm(f^\star \gamma)=0$ if and only if $\gamma$ is $d$-torsion, i.e:
\[ K_f \cap f^\star J_B = f^\star J_B[d]. \]
Since $P_f \leqslant K_f$ it follows from \eqref{eqn: ker lambda v2} that
\begin{equation} \label{eqn: ker lambda v3} \Ker(\lambda_{P_f}) = P_f \cap f^\star J_B[d]. \end{equation}
The identity \eqref{eqn: f star followed by Nm} also gives $\Ker f^\star \leqslant J_B[d]$ and it follows that we have a short exact sequence:
\begin{equation} \label{eqn: short exact sequence JB and f star JB} 0 \to \Ker f^\star \to J_B[d] \xrightarrow[]{\varphi} f^\star J_B[d] \to 0. \end{equation}
Replacing the final term by its intersection with $P_f$ we obtain:
\begin{equation} \label{eqn: ker lambda v4} 0 \to \Ker f^\star \to \varphi^{-1}(P_f \cap f^\star J_B[d]) \to \Ker(\lambda_{P_f}) \to 0. \end{equation}
To describe the middle term, note that it is equal to the kernel of the composite
\[ J_B[d] \xrightarrow{f^\star} f^\star J_B[d] \hookrightarrow K_f \to F_f \]
Dualising produces a map $F_f^D \to J_B[d]^D$. Applying \eqref{eqn: exact sequence Norm f dual} and \eqref{eqn: square compatbility norm and f star} we obtain:
\[ F_f^D = \Ker (\Nm_f^\vee) = \lambda_B(\Ker f^\star). \]
We thus obtain a commuting square whose vertical maps are isomorphisms:
\[
\begin{tikzcd}
    F_f^D \ar[d,"\cong"',"\lambda_B^{-1}"] \ar[r] & J_B[d]^D \ar[d,"\cong"',"\lambda_B^{-1}"] \\
    \Ker f^\star \ar[r,hook] & J_B[d].
\end{tikzcd}
\]
and it follows that, as subgroups of $J_B[d]$, we have
\[ \varphi^{-1}(P_f \cap f^\star J_B[d]) = \Ker(J_B[d] \to F_f) = \lambda_B^{-1} \left( \Im(F_f^D \to J_B[d]^D)^\perp \right) = (\Ker f^\star)^\perp \]
where the orthogonal complement is taken with respect to the Weil pairing on $J_B[d]$. Plugging into \eqref{eqn: ker lambda v4} gives the desired formula.\end{proof}

\subsubsection{Pullback on Jacobians via monodromy data} \zcref{prop: Prym polarisation via pullback on Jacobian} describes the kernel of the Prym polarisation in terms of the subgroup inclusion
\[ \Ker f^\star \leqslant J_B[d] \]
and its interaction with the Weil pairing. We now describe the latter data explicitly in terms of the monodromy data associated to the cover (\zcref{prop: polarisation kernel in terms of monodromy}). This completes the description of the polarisation type of a general Prym variety.

We establish notation for the monodromy data, following \zcref{sec: covers}. Fix a degree $d$ cover of smooth curves
\[ f \colon C \to B \]
\'etale away from the marked points $b_1,\ldots,b_n \in B$. We introduce shorthand notation for the associated punctured curves:
\begin{align*}
    B^\circ \colonequals B \setminus \{ b_1,\ldots,b_n\}, \qquad C^\circ & \colonequals f^{-1}(B^\circ).
\end{align*}
The cover $f$ corresponds to a monodromy representation
\[ \rho \colon \pi_1(B^\circ) \to S_d \]
well-defined up to post-conjugation. As in \zcref{sec: setup} we let $G \leqslant S_d$ denote the image of $\rho$, and
\[ g_i \colonequals \rho(\gamma_i) \]
denote the image of a small loop $\gamma_i$ around $b_i$. As in \zcref{sec: covers}, each $g_i$ is well-defined up to conjugation by elements of $G$, and the data $(G,[g_1],\ldots,[g_n])$ is well-defined up to simultaneous conjugation by elements of $S_d$.

We now define basepoint stabiliser subgroups of $G$. Suppose that $C$ has connected components $C_1,\ldots,C_k$. Implicit in the definition of the fundamental groups are choices of basepoints:
\begin{align*}
    b_0 \in B^\circ, \qquad c_j \in C_j \cap f^{-1}(b_0).
\end{align*}
The monodromy representation encodes the action of $\pi_1(B^\circ)$ on $f^{-1}(b_0)$. 

\begin{definition} For $j \in [k]$ the associated \textbf{basepoint stabiliser subgroup}
\[ H_j \leqslant G \]
is the image under $\rho$ of the subgroup of $\pi_1(B^\circ)$ which fixes the basepoint $c_j$, equivalently the image under $\rho$ of the subgroup $\pi_1(C_j^\circ) \leqslant \pi_1(B^\circ)$.
\end{definition}

\begin{remark}
Concretely, the decomposition of $C$ into connected components corresponds to a factorisation of the monodromy representation
\[ \rho \colon \pi_1(B^\circ) \to S_{I_1} \times \cdots \times S_{I_k} \leqslant S_d \]
for a partition $I_1 \sqcup \cdots \sqcup I_k = [d]$. The basepoint stabiliser subgroup $H_j$ is then the stabiliser in $G$ of an arbitrarily-chosen element $i \in I_j$.    
\end{remark}

Abelianising the monodromy representation $\rho$ produces a surjective homomorphism
\[ \rho^{\operatorname{ab}} \colon H_1(B^\circ) \to G^{\operatorname{ab}} \]
and since the kernel of $H_1(B^\circ) \to H_1(B)$ is generated by the $\gamma_i$, this descends to a surjective homomorphism
\begin{equation} \label{eqn: homomorphism H1 to ablenisation quotient} H_1(B) \to G^{\operatorname{ab}} / \langle \overline{g}_1,\ldots,\overline{g}_n \rangle \to G^{\operatorname{ab}} / \langle \overline{g}_1,\ldots,\overline{g}_n,\overline{H}_1,\ldots,\overline{H}_k \rangle \end{equation}
where  $\overline{g}_i$ and $\overline{H}_j$ denote the images in the abelianisation, noting that the former do not depend on the choice of conjugacy class representative.

We are now positioned to describe the inclusion $\Ker f^\star \leqslant J_B[d]$ which gives rise to $\Ker(\lambda_{\Prym(f)})$ in \zcref{prop: Prym polarisation via pullback on Jacobian}.
\begin{proposition} \label{prop: polarisation kernel in terms of monodromy} With the above notation, there is a canonical identification
\begin{equation} \label{eqn: identification torsion and H1} J_B[d] = \Hom(H_1(B),\mu_d) \end{equation}
which identifies the Weil pairing with the intersection pairing. Under this identification, $\Ker f^\star \leqslant J_B[d]$ coincides with the inclusion dual to the homomorphism \eqref{eqn: homomorphism H1 to ablenisation quotient}:
\begin{equation} \label{eqn: map description of Ker fstar} \Hom( G^{\operatorname{ab}}/\langle \overline{g}_1,\ldots,\overline{g}_n,\overline{H}_1,\ldots,\overline{H}_k \rangle, \mu_d) \to \Hom(H_1(B), \mu_d). \end{equation}
\end{proposition}

\begin{proof}
The identification \eqref{eqn: identification torsion  and H1} and its compatibility with the respective pairings is a basic fact about Jacobians. The subgroup $\Ker f^\star$ consists of $d$-torsion line bundles on $B$ whose pullback to $C$ is trivial. This is given by:
\[
\Ker f^\star = \Ker \big( \Hom(\pi_1(B),\mu_d) \to \Hom(\pi_1(C),\mu_d) \big). 
\]
First assume that $C$ is connected and consider $H \colonequals H_1 \leqslant G$ the basepoint stabiliser subgroup. A loop in $\pi_1(B^\circ)$ belongs to the image of $\pi_1(C^\circ)$ if and only if it acts trivially on the basepoint. We therefore have:
\[ \Im \big( \pi_1(C^\circ) \hookrightarrow \pi_1(B^\circ) \big) = \rho^{-1}(H). \]
In particular we have an inclusion $\Ker \rho \hookrightarrow \pi_1(C^\circ)$ fitting into the following diagram:
\begin{equation} \label{eqn: Prym polarisation monodromy proof diagram 1}
\begin{tikzcd}
	0 \ar[r] & \Ker \rho \ar[r] \ar[d,hook] & \pi_1(B^\circ) \ar[r,"\rho"] \ar[d,equals] & G \ar[r] & 0 \\ 
	0 \ar[r] & \pi_1(C^\circ) \ar[r] & \pi_1(B^\circ). & &
\end{tikzcd}
\end{equation}
Consider the following diagram and its abelianisation (recalling that this functor is right-exact):
\[
\begin{tikzcd}
    0 \ar[r] & K \ar[r] \ar[d,equals] & \pi_1(C^\circ) \ar[r] \ar[d,hook] & H \ar[d,hook] \ar[r] & 0 & & K^{\operatorname{ab}} \ar[r] \ar[d,equals] & H_1(C^\circ) \ar[r] \ar[d] & H^{\operatorname{ab}} \ar[r] \ar[d] & 0 \\
    0 \ar[r] & K \ar[r] & \pi_1(B^\circ) \ar[r] & G \ar[r] & 0, & & K^{\operatorname{{ab}}} \ar[r] & H_1(B^\circ) \ar[r,"\rho^{\operatorname{ab}}"] & G^{\operatorname{ab}} \ar[r] & 0.
\end{tikzcd}
\]
Setting $\overline{H} \colonequals \Im(H^{\operatorname{ab}} \to G^{\operatorname{ab}})$ a diagram chase gives:
\[ \Im \big( H_1(C^\circ) \to H_1(B^\circ) \big) = (\rho^{\operatorname{ab}})^{-1}(\overline{H}). \]
Abelianising \eqref{eqn: Prym polarisation monodromy proof diagram 1} then gives:
\begin{equation}
    \begin{tikzcd}
       0 \ar[r] & \Ker ( \rho^{\operatorname{ab}} )\ar[r] \ar[d,"\iota"] & H_1(B^\circ) \ar[r,"\rho^{\operatorname{ab}}"] \ar[d,equals] & G^{\operatorname{ab}} \ar[d,"\varphi"] \ar[r] & 0 \\
        0 \ar[r] & (\rho^{\operatorname{ab}})^{-1}(\overline{H}) \ar[r] & H_1(B^\circ) \ar[r] & H_1(B^\circ)/\Im(H_1(C^\circ) \to H_1(B^\circ)) \ar[r] & 0.
    \end{tikzcd}
\end{equation}
The Snake Lemma then shows that $\varphi$ is surjective with:
\[ \Ker \varphi = \Coker \iota = \overline{H} \]
and so we have a short exact sequence:
\begin{equation} \label{eqn: exact sequence Hbar Gab H1} 0 \to \overline{H} \to G^{\operatorname{ab}} \to H_1(B^\circ)/\Im(H_1(C^\circ) \to H_1(B^\circ)) \to 0. \end{equation}
We now consider torsion bundles. Recall that we wish to describe:
\[ \Ker f^\star = \Ker \big( \Hom(\pi_1(B),\mu_d) \to \Hom(\pi_1(C),\mu_d) \big). \]
The surjection $\pi_1(B^\circ) \to \pi_1(B)$ gives rise to an inclusion
\[ \Hom(\pi_1(B),\mu_d) \hookrightarrow \Hom(\pi_1(B^\circ),\mu_d) \]
whose image consists of homomorphisms $\chi$ that send the small loops $\gamma_i$ to the identity. We thus have:
\begin{equation} \label{eqn: Prym polarisation proof intersection} \Ker f^\star = \Ker \big( \Hom(\pi_1(B^\circ),\mu_d) \to \Hom(\pi_1(C^\circ),\mu_d) \big) \cap \left\{ \chi \mid \chi(\gamma_1) = \cdots = \chi(\gamma_n) = 1 \right\}. \end{equation}
Focusing on the first term, the Hurewicz theorem gives:
\[ \Ker \big( \Hom(\pi_1(B^\circ),\mu_d) \to \Hom(\pi_1(C^\circ),\mu_d) \big) = \Ker \big( \Hom(H_1(B^\circ),\mu_d) \to \Hom(H_1(C^\circ),\mu_d) \big). \]
The exact sequence \eqref{eqn: exact sequence Hbar Gab H1} then identifies this with:
\[ \big\{ \chi \colon G^{\operatorname{ab}} \to \mu_d \mid \chi(\overline{H})=1 \big\}. \]
Performing the intersection \eqref{eqn: Prym polarisation proof intersection} then gives the desired formula:
\[ \Ker f^\star = \big\{ \chi \colon G^{\operatorname{ab}} \to \mu_d \mid \chi(\overline{H})=\chi(\overline{g}_1) = \cdots = \chi(\overline{g}_n) = 1 \big\}. \]
This completes the proof when $C$ is connected. For the disconnected case, we have
\[ \Ker f^\star = \cap_{j=1}^k \Ker f_j^\star \]
where the intersection takes place inside $J_B$. For each $j \in [k]$ the connected case gives
\[ \Ker f_j^\star = \big\{ \chi \colon G^{\operatorname{ab}} \to \mu_{d_j} \mid \chi(\overline{g}_1)=\cdots=\chi(\overline{g}_n)=\chi(\overline{H}_j)=1 \big\} \leqslant \Hom(\pi_1(B),\mu_{d_j}) = J_B[d_j]. \]
We then perform the intersection inside the group of torsion line bundles on $B$:
\[ \Hom(\pi_1(B),\Q/\Z). \]
Since by \zcref{prop: Prym polarisation via pullback on Jacobian} we know that $\Ker f^\star$ must be contained inside $J_B[d]$ this gives the required identity:
\[ \Ker f^\star = \big\{ \chi \colon G^{\operatorname{ab}} \to \mu_d \mid \chi(\overline{g}_1)=\cdots=\chi(\overline{g}_n)=\chi(\overline{H}_1)=\cdots=\chi(\overline{H}_k)=1 \big\}. \qedhere \]
\end{proof}

We note the following consequence, which allows us to quickly determine the polarisation type in many important cases (see e.g. \zcref{prop: cases with uniform polarisation type}).

\begin{corollary} \label{cor: polarisation type cyclic quotient group}
	Suppose that the quotient group
    \begin{equation} \label{eqn: polarisation type quotient group} G^{\operatorname{ab}}/\left \langle \overline{g}_1,\ldots,\overline{g}_n, \overline{H}_1,\ldots,\overline{H}_k \right \rangle \end{equation}
    is cyclic of order $s$. Then the polarisation type $\delta$ of the Prym variety is 
	\[
	(1,\ldots,1,d/\gcd(d,s),d,\ldots,d)
	\]
    where $d$ occurs $g-1$ times.
\end{corollary}

\begin{proof}
Writing $\ell \colonequals \gcd(d,s)$ we have:
\[
\Hom(\mu_s,\mu_d) \cong \Hom(\mu_{\ell},\mu_d) \cong \mu_\ell.
\]
The image of this group under the map \eqref{eqn: map description of Ker fstar} is the subgroup of $\Hom(H_1(B),\mu_d) \cong (\Z/d\Z)^{2g}$ consisting of multiples of a single vector $v$. This vector has order precisely $\ell$ since the map is injective, and hence it can be identified with the $d/\ell$ multiple of the first vector of a symplectic basis $e_1,f_1,\ldots,e_g,f_g$. In terms of this basis we have:
\begin{align*} 
\Span(v) & = (\Z/d\Z) (d/\ell \cdot e_1),\\
\Span(v)^\perp & = (\Z/d\Z)e_1 \oplus (\Z/d\Z)(\ell \cdot f_1) \oplus (\Z/d\Z)^{2g-2}
\end{align*}
and hence we find: 
\begin{align*}
\Ker(\lambda_{\Prym(f)}) & = \Span(v)^\perp/\Span(v) \\
& \cong \ZZ/(d/\ell)\ZZ \oplus \ZZ/(d/\ell) \ZZ \oplus (\ZZ/d\ZZ)^{2g-2} \\
& \cong (\Z/(d/\ell)\Z \oplus (\Z/d\Z)^{g-1})^2
\end{align*}
which means that the polarization has type 
\[
\delta = (1,\ldots,1, d/\ell,d,\ldots,d)
\]
with $d$ appearing $g-1$ times, as claimed.
\end{proof}

\subsubsection{Calculating the polarisation type} \label{sec: calculating the polarisation type} 
We now explain an effective method for calculating $\Ker f^\star$ using \zcref{prop: polarisation kernel in terms of monodromy}. Choose a standard symplectic basis
\[ \alpha_1,\beta_1,\ldots,\alpha_g,\beta_g \]
of $H_1(B)$. This corresponds to a choice of isomorphism
\begin{equation} \label{eqn: choice of basis for H1B} H_1(B) \xrightarrow{\cong} \Z^{2g} \end{equation}
identifying the intersection form with the standard symplectic form $J$. We consider the image of the basis elements under the abelianisation of the monodromy representation:
\[ \overline{x}_i \colonequals \rho^{\operatorname{ab}}(\alpha_i), \overline{y}_i \colonequals \rho^{\operatorname{ab}}(\beta_i) \in G^{\operatorname{ab}}. \]
Apply \zcref{prop: polarisation kernel in terms of monodromy} and choose a basis of characters generating $\Ker f^\star$:
\[ \chi_1,\ldots,\chi_r \colon G^{\operatorname{ab}} \to \mu_d. \]
Taking $\zeta$ to be a primitive $d$th root of unity, we can then write
\[
\chi_j(\overline{x}_i) = \zeta^{A_{ij}} \qquad \chi_j(\overline{y}_i) = \zeta^{B_{ij}}
\]  
for integers $A_{ij},B_{ij} \in \{0,\ldots,d-1\}$. Via the identification \eqref{eqn: choice of basis for H1B}, the subgroup $\Ker f^\star \leqslant (\Z/d\Z)^{2g}$ is identified with the column span of the following $2g \times r$ matrix:
\[
M_\rho \colonequals \begin{pmatrix} A_{11} & \cdots & A_{1r} \\ 
	\vdots & \cdots & \vdots \\ 
    A_{g1} & \cdots & A_{gr} \\
    B_{11} & \cdots & B_{1r} \\
    \vdots & \cdots & \vdots \\ 
	B_{g1} & \cdots & B_{gr} \end{pmatrix}.
\]
Let $J$ denote the standard $2g \times 2g$ symplectic matrix 
\[ J = \begin{pmatrix} 0 & I_g \\ -I_g & 0 \end{pmatrix} \]
so that the Weil pairing is given by $\left \langle u , v \right \rangle = u^{T}Jv$. We then have 
\[
(\Ker f^\star)^{\perp} = \Ker (M_\rho^{T}J)
\]
where $M_\rho^T J$ is considered as a map $(\ZZ/d\ZZ)^{2g} \to \ZZ^{r}$. We thus arrive at the key identity: 
\begin{equation}
\Ker (\lambda_{\Prym(f)}) = \Ker (M_\rho^T J)/Im(M_\rho)
\end{equation}
This is illustrated in the following example.

\begin{example} \label{example: same global type different polarisation type}
	Consider the subgroup $G \leqslant S_4$ generated by: 
    \[ u = (12)(34), \qquad v = (13)(24).\]
    We have $uv = vu = (14)(23)$, so $G \cong \mu_2 \times \mu_2$. We will study two quadruple covers with monodromy group $G \leqslant S_4$, base genus $g=2$ and $n=0$ (no branch points). We will see that the associated Prym varieties have different polarisation types.
    
    For the first cover $f_1 \colon C_1 \to B$, we take the monodromy representation 
	\begin{align*}
	\rho_1 \colon \pi_1(B) & \to G \\
    \alpha_1 & \mapsto u,\\
    \alpha_2 & \mapsto 1, \\
    \beta_1 & \mapsto v, \\
    \beta_2 & \mapsto 1.
	\end{align*}
    The basepoint stabiliser subgroup $H \leqslant G$ which fixes the point $1 \in \{1,2,3,4\}$ is trivial in this case, and there are no loops $\gamma_i$. Therefore
    \[ \Ker f^\star = \Hom(G^{\operatorname{ab}},\mu_4) = \Hom(G,\mu_4) \]
    since $G$ is already abelian. A basis for the characters $G \to \mu_4$ consists of $\chi_1,\chi_2$ given by  
	\[
    \begin{alignedat}{2}
	\chi_1(u) & = \zeta^2,\qquad \qquad && \chi_2(u) = 1, \\
    \chi_1(v)& = 1, && \chi_2(v) = \zeta^2. 
    \end{alignedat}
	\]
	Therefore, the matrix $M_{\rho_1}$ is given by:
	\[
	M_{\rho_1} = \begin{pmatrix}  2 & 0 \\ 0 & 0 \\ 0 & 2 \\ 0 & 0\end{pmatrix}
	\]
	Letting $e_1,e_2,f_1,f_2$ denote the standard basis of $(\ZZ/4\ZZ)^4$ we have:
    \begin{align*}
    \Im(M_{\rho_1}) & = \Span(2e_1,2f_1), \\   
    \Ker(M_{\rho_1}^T J) & =  \Span(2e_1,2f_1,e_2,f_2).
    \end{align*}
    We thus obtain the kernel of the Prym polarisation for the first cover
	\[
	\Ker (\lambda_{\Prym(f_1)}) = \Span(e_2,f_2) \cong (\ZZ/4\ZZ)^2
	\]
    and it follows that the polarisation type is $\delta_1=(1,1,4)$.

    For the second cover $f_2 \colon C_2 \to B$ we take the monodromy representation:
	\begin{align*}
	\rho_2 \colon \pi_1(B) & \to G \\
    \alpha_1 & \mapsto u \\
    \alpha_2 & \mapsto v \\
    \beta_1 & \mapsto 1 \\
    \beta_2 & \mapsto 1.
    \end{align*}
    This produces the following matrix
	\[
	M_{\rho_2} = \begin{pmatrix} 2 & 0 \\ 0 & 2 \\ 0 & 0 \\ 0 & 0 \end{pmatrix}
	\]
	from which we compute:
	\begin{align*}
	    \Im(M_{\rho_2}) & = \Span(2e_1,2e_2),\\
        \Ker(M_{\rho_2}^T J) & = \Span(e_1,e_2,2f_1,2f_2)
	\end{align*}
	Therefore the quotient is
	\[
	\Ker (\lambda_{\Prym(f_2)}) \cong (\ZZ/2\ZZ)^4
	\]
    and it follows that the polarisation type is $\delta_2=(1,2,2)$.
\end{example}

\begin{remark} \label{rmk: choice of polarisation}
The Prym polarisation is almost never principal. When the cover is connected, it is a multiple of a principal polarisation if and only if $\delta = (d,\ldots,d)$. This occurs in a handful of known cases (enumerated e.g. in \cite[Theorem~3.2.6]{LangeRodriguezBook}), including the classical case of \'etale double covers. These are essentially the cases when $\dim \Prym(f) \leqslant \dim \Jac(B)$, with the extra condition that when $d=3$ and $g=2$ the cover must be non-cyclic.

In this paper, we always equip our Prym varieties with the induced non-principal polarisation from $\Jac(C)$, even in cases where an alternative principal polarisation exists. This will not affect the eventual formulae for the tautological projections, due to the identity (\zcref{lem: integral on multiple of principal polarisation space}):
\[ \int_{\Acalbar_{g,(d,\ldots,d)}} \lambda_1 \cdots \lambda_g = \int_{\Acalbar_g} \lambda_1 \cdots \lambda_g.\]
\end{remark}

\subsection{Prym map} \label{sec: Prym map} We will now define the Prym map
\begin{equation} \label{eqn: Prym map} \Mcalbar_{g,n}^{\Theta}(\Bcal S_d)^\Sigma \to \Acalbar_{\tilde{g},\delta}^\Sigma \end{equation}
sending a cover to the associated Prym variety. For this to be well-defined, we must restrict to a special class of global types: those having uniform polarisation type.

\subsubsection{Uniform Prym polarisation types} \label{sec: uniform polarisation type}

\begin{definition} \label{uniform polarisation type} A global type $\Theta$ has \textbf{uniform polarisation type} if the polarisation type $\delta$ of $\Prym(f)$ is uniform for all smooth covers $f \colon C \to B$ in the associated moduli space $\Mcal_{g,n}^{\Theta}(\Bcal S_d)$.
\end{definition}

This condition is necessary in order to produce a Prym map \eqref{eqn: Prym map}. Happily, it is satisfied in all the cases of interest to us:

\begin{proposition} \label{prop: cases with uniform polarisation type}
    Fix a global type $\Theta$ with equivalence class representative $(G \leqslant S_d,[g_1],\ldots,[g_n])$. Then $\Theta$ has uniform polarisation type in all the following cases:
    \begin{itemize}
        \item $G=S_d$.
        \item $G \cong \mu_d$.
        \item $G=A_d$.
        \item $d$ is prime and $G$ acts transitively on $[d]$.
    \end{itemize}
    In the case $G=S_d$, if we assume $d \geqslant 3$ then the polarisation has type $\delta=(1,\ldots,1,d,\ldots,d)$ where $d$ occurs $g$ times.
\end{proposition}

\begin{proof} 
We apply \zcref{cor: polarisation type cyclic quotient group}. For the first three cases the abelianisations $G^{\operatorname{ab}}$ are already cyclic, and so the quotient group \eqref{eqn: polarisation type quotient group} is certainly cyclic. Since this group is determined by the global type (only the embedding in $J_B[d]$ depends on more subtle monodromy data), the claim follows.

For the final case, we consider the basepoint stabiliser subgroup $H \leqslant G$. The orbit-stabiliser theorem gives $[G \colon H] = d$. Letting $N$ denote the normal closure of $H$ in $G$ it is easy to prove that
\[ G^{\operatorname{ab}}/\overline{H} \cong (G/N)^{\operatorname{ab}}.\]
Since $H$ has prime index $d$ it follows that $N$ has index $1$ or $d$. Therefore the quotient $G/N$ is cyclic, so the same is true for $G^{\operatorname{ab}}/\overline{H}$ and \zcref{cor: polarisation type cyclic quotient group} applies again.

We now determine the polarisation type when $G=S_d$ and $d \geqslant 3$. We have $[G,G]=A_d$ whilst the basepoint stabiliser subgroup $H \leqslant S_d$ contains a $2$-cycle. Therefore the quotient $G^{\operatorname{ab}}/\overline{H}$ is trivial. The desired polarisation type follows from \zcref{cor: polarisation type cyclic quotient group}.
\end{proof}

When dealing with Prym maps and Prym classes, we always assume that our global type has uniform polarisation type. On the other hand, our comparison formula (\zcref{sec: comparison}) and algorithm for computing Prym Hodge integrals (\zcref{sec: algorithm}) do not require this assumption.

\begin{remark}
    \zcref{example: same global type different polarisation type} exhibits two covers with the same global type but different polarisation types. This shows that the above definition is indeed necessary.
\end{remark}

\begin{remark} \label{rmk: connected components}
An alternative to our approach would be to dispense with the global types entirely and work with the finer decomposition given by the connected components of the moduli space. The polarisation type is certainly uniform on the connected components, and hence the Prym map can be defined. Upcoming work \cite{GlynnHigherGenus,GlynnThesis} finally provides an explicit combinatorial description of the connected components of $\Mcalbar_{g,n}(\Bcal S_d)$, making this a viable approach.

The drawback from our point of view is that the splitting formalism (\zcref{prop: splitting}) and associated theory is not yet developed for these connected components (we believe this to be a worthwhile problem to pursue). Since the global types are sufficient for the cases of primary interest (\zcref{prop: cases with uniform polarisation type}), we content ourselves to work with this formalism throughout.
\end{remark}

\subsubsection{Construction of the Prym map} \label{sec: Prym map construction}

Fix setup data as in \zcref{sec: setup} and consider the associated moduli space (\zcref{def: decomposition by type}):
\[ \Mcalbar_{g,n}^{\Theta}(\Bcal S_d).\]
Assume as in \zcref{sec: uniform polarisation type} that $\Theta$ has uniform polarisation type. We let $\tilde{g}$ denote the dimension \eqref{eqn: dim of Prym} of the associated Prym variety, and $\delta$ its induced polarisation type. The construction of Prym varieties for smooth curves (\zcref{sec: Prym smooth}) generalises to families and produces the \textbf{Prym map}:
\begin{align*}
\Prym \colon \Mcal_{g,n}^{\Theta}(\Bcal S_d) & \to \Acal_{\tilde{g},\delta} \\
(f \colon C \to B) & \mapsto (\Prym(f), \Theta_{\Prym(f)})
\end{align*}

We now wish to extend this to the compactifications. Following \cite{AMRT} this has a polyhedral criterion: given any admissible decomposition $\Sigma$ corresponding to a toroidal compactification of the space of abelian varieties, there is an associated toroidal modification of the space of $S_d$-covers, on which the Prym map extends:
\begin{equation} \label{eqn: extended Prym}
\begin{tikzcd}
\Mcalbar_{g,n}^{\Theta}(\Bcal S_d)^{\Sigma} \ar[r,"\Prym"] \ar[d,"b"] & \Acalbar_{\tilde{g},\delta}^\Sigma \\
\Mcalbar_{g,n}^{\Theta}(\Bcal S_d).
\end{tikzcd}
\end{equation}
Moreover, the universal semiabelian variety \eqref{eqn: universal semiabelian variety} and the universal Prym variety \eqref{eqn: three universal semiabelian varieties} agree after pullback:
\begin{equation} \label{eqn: universal Prym is universal semiabelian after pullback} \Prym^\star \Ucal_{\tilde{g},\delta}^\Sigma \cong b^\star \Prym(f). \end{equation}

\begin{remark} A version of the Prym map can be defined even for global types whose polarisation type is not uniform. Consider the diagram
\[
\begin{tikzcd}
& & \bigsqcup_{\delta} \Acalbar_{\tilde{g},\delta}^{\operatorname{lev},\Sigma} \ar[ld,"\pi"{xshift=-11pt, yshift=11pt}] \ar[rd,"\varphi"] & \\
\Mcalbar_{g,n}^{\Theta}(\Bcal S_d)^\Sigma \ar[r,"\Prym"] & \bigsqcup_{\delta}\Acalbar_{\tilde{g},\delta}^\Sigma & & \Acalbar_{\tilde{g}}^\Sigma
\end{tikzcd}
\]
where the correspondence comes from \eqref{eqn: level correspondence compactified} and the union is over all polarisation types $\delta$ arising from the fixed global type $\Theta$. We may then define the Prym class as:
\[ \varphi_\star \pi^\star \Prym_\star [\Mcalbar_{g,n}^{\Theta}(\Bcal S_d)^\Sigma]. \]
This produces a class on $\Acalbar_{\tilde{g}}^\Sigma$ whose tautological projection we can calculate using the algorithm of \zcref{sec: algorithm}. However, this class is somewhat unnatural because it ignores the proportionality constants arising in the transition $\varphi_\star \pi^\star$ (see \zcref{sec: proportionality constant}).    
\end{remark}

%%%%%%%%%%%%%%%
% SECTION: TAUTOLOGICAL PROJECTION AND LAMBDA INTEGRALS
%%%%%%%%%%%%%%%

\section{Tautological projection and Hodge integrals} \label{sec: tautological}

\subsection{Tautological projection of the Prym class} \label{sec: taut projection of the Prym class} As in the previous section, fix a global type $\Theta$ with uniform polarisation type, so that the (compactified) Prym map \eqref{eqn: extended Prym} is well-defined:
\[ \Prym \colon \Mcalbar_{g,n}^{\Theta}(\Bcal S_d)^\Sigma \to \Acalbar_{\tilde{g},\delta}^\Sigma.\]
This morphism is proper and the pushforward of the fundamental class is referred to as the \textbf{Prym class}:
\[ \Prym_\star \left[\Mcalbar_{g,n}^{\Theta}(\Bcal S_d)^\Sigma \right] \in \CH^{{\tilde{g}+1 \choose 2}-3g+3-n}(\Acalbar_{\tilde{g},\delta}^\Sigma).\]

\begin{remark} When the Prym map is generically injective, the Prym class is precisely the fundamental class of the closure of the Prym locus: that is, the closure of the locus in $\Acal_{\tilde{g},\delta}$ parametrising Prym varieties arising from covers of smooth curves. 

In the classical setting $(d=2,n=0)$ the Prym map is generically injective if and only if $g \geqslant 7$ \cite{FriedmanSmith}. In the ramified setting $(d=2,n>0)$ there is also a complete answer, albeit more intricate \cite{NagarajRamanan,BardelliCilibertoVerra,MarcucciPirola,MarcucciNaranjo,NaranjoOrtega}. Partial results are also known for cyclic coverings of prime degree \cite{LangeOrtegaSeven,LangeOrtegaCyclic}.\end{remark}

We seek to calculate the tautological projection of the Prym class:
\[ \taut \left( \Prym_\star \left[ \Mcalbar_{g,n}^{\Theta}(\Bcal S_d)^\Sigma \right] \right) \in R^{{\tilde{g}+1 \choose 2}-3g+3-n}(\Acalbar_{\tilde{g},\delta}) . \]
By \zcref{thm: tautological ring presentation}, this is equivalent to calculating the integrals
\begin{equation} \label{eqn: integrals we want 1} \int_{\Mcalbar_{g,n}^{\Theta}(\Bcal S_d)^\Sigma} F(\Prym^\star \lambda_1,\ldots,\Prym^\star \lambda_{\tilde{g}}) \end{equation}
where $F(z_1,\ldots,z_{\tilde{g}})$ runs over all the squarefree monomials. We begin by investigating the classes $\Prym^\star \lambda_i$.

\subsection{Three Hodge bundles on $S_d$-covers} \label{sec: Hodge bundles on moduli space} Recall from \eqref{eqn: three universal semiabelian varieties} that the moduli space of $S_d$-covers carries three universal semiabelian varieties:
\[
\begin{tikzcd}
\Prym(f) \ar[r] \ar[rd,"q_f" below] & \Jac(C) \ar[r] \ar[d,"q_C"] & \Jac(B) \ar[ld,"q_B"] \\
& \Mcalbar_{g,n}^{\Theta}(\Bcal S_d) \ar[ul, bend left, "e_f"] \ar[ur, bend right,"e_B" right] \ar[u,bend left,"e_C"] 
\end{tikzcd}
\]
Each universal semiabelian variety produces an associated Hodge bundle and lambda classes, denoted and defined:
\begin{alignat*}{3}
\Etilde & \colonequals e_f^\star \Omega_{q_f}, \qquad && \lambdatilde_i \colonequals c_i(\Etilde), \\
\EE_C & \colonequals e_C^\star \Omega_{q_C}, \qquad && \lambda^C_i \colonequals c_i(\EE_C),\\
\EE_B & \colonequals e_B^\star \Omega_{q_B}, \qquad && \lambda^B_i \colonequals c_i(\EE_B).
\end{alignat*}
The short exact sequence \eqref{eqn: ses of tangent spaces Prym compactified} produces a short exact sequence of vector bundles:
\begin{equation} \label{eqn: ses Hodge bundles} 0 \to \EE_B \to \EE_C \to \Etilde \to 0. \end{equation}
We now turn to the Prym map. The identification \eqref{eqn: universal Prym is universal semiabelian after pullback} allows us to identify:
\[ \Prym^\star \EE = b^\star \Etilde.\]
It follows that the desired integrals \eqref{eqn: integrals we want 1} can be expressed as integrals on the space of $S_d$-covers itself, rather than a blowup thereof:
\[ \int_{\Mcalbar_{g,n}^{\Theta}(\Bcal S_d)^\Sigma} F(\Prym^\star \lambda_1,\ldots,\Prym^\star \lambda_{\tilde{g}}) = \int_{\Mcalbar_{g,n}^{\Theta}(\Bcal S_d)} F(\lambdatilde_1,\ldots,\lambdatilde_{\tilde{g}}).\]
Calculating the tautological projection of the Prym class therefore reduces to the following:

\begin{problem} \label{problem: first formulation} Integrate monomials in Prym lambda classes $\lambdatilde_i$ over the moduli space of $S_d$-covers of global type $\Theta$.\end{problem}

The rest of the paper is devoted to solving this problem. We produce an algorithm for performing such integrals (\zcref{sec: algorithm}). We computer-implement the algorithm for $d=2$ (\zcref{appendix}) and hand-implement it in a few other cases (Sections~\ref{sec: degree 3 covers non-cyclic}~and~\ref{sec: degree 3 covers cyclic}).

For organisational reasons, it is better to phrase the problem in terms of the Chern \emph{characters} of the Hodge bundle. We denote these
\begin{align*}
\chitilde_i & \colonequals ch_i(\Etilde), \\
\chi^C_i & \colonequals ch_i(\EE_C), \\
\chi^B_i & \colonequals ch_i(\EE_B),
\end{align*}
and refer to them as \textbf{chi classes}. \zcref{problem: first formulation} then becomes:
\begin{problem} \label{problem: second formulation} Integrate monomials in Prym chi classes $\chitilde_i$ over the moduli space of $S_d$-covers of global type $\Theta$:
\[ \int_{\Mcalbar_{g,n}^{\Theta}(\Bcal S_d)} \chitilde_{i_1} \cdots \chitilde_{i_k}.\]
\end{problem}
	
\section{Comparison formula} \label{sec: comparison}

\noindent Following \zcref{problem: second formulation}, our goal is to integrate monomials in Prym chi classes over the moduli space of $S_d$-covers. The basic idea is to use the projection formula for the finite target morphism
\[ t \colon \Mcalbar_{g,n}^{\Theta}(\Bcal S_d) \to \Mcalbar_{g,n} \]
whose degree is given by a count of monodromy representations (\zcref{lem: degree target map global}). The resulting chi integrals on $\Mcalbar_{g,n}$ can then be computed using Faber's algorithm \cite{Faber_algorithms}.

It remains to recursively express the $\chitilde_i$ in terms of classes pulled back from $\Mcalbar_{g,n}$. The short exact sequence \eqref{eqn: ses Hodge bundles} gives the K-theory identity
\begin{equation} \label{eqn: Etilde in terms of C and B} \Etilde = \EE_C - \EE_B = \EE_C - t^\star \EE. \end{equation}
The goal therefore is to express $\EE_C$ in terms of $\EE_B = t^\star \EE$. This will be achieved using Grothendieck--Riemann--Roch for the universal curves $C$ and $B$.

\subsection{Statement of the comparison} \label{sec: statement of the comparison}
The comparison involves a correction term, which incorporates psi classes at the branch points and boundary divisors in the moduli space of $S_d$-covers. We introduce the necessary notation.

Recall (\zcref{sec: vertexwise strata}) that boundary strata in the moduli space are indexed by vertexwise $S_d$-graphs $(\Gamma,\Theta_{V(\Gamma)})$. We write
\[ j_{(\Gamma,\Theta_{V(\Gamma)})} \colon \Mcalbar_{(\Gamma,\Theta_{V(\Gamma)})}^{\Theta}(\Bcal S_d) \hookrightarrow \Mcalbar_{g,n}^{\Theta}(\Bcal S_d) \]
for the inclusion of the associated stratum. A stratum is a divisor if and only if $|E(\Gamma)|=1$. There are two types of such graphs, loops and bridges:
\[
\begin{tikzpicture}

% Loop locus
\draw[fill=black] (0,0) circle[radius=3pt];
\draw (0,0) -- (-0.5,0);
\draw (-0.75,0) node{$[n]$};
\draw (0,0) node[right]{$g\!-\!1$};
\draw (-0.3,0.35) node{$G_{v_0}$};
\draw (1,0) circle[radius=1];
\draw[-Stealth] (0.9,1) -- (1.1,1) [->];
\draw (1,1) node[above]{$[g_{\vec{e}}]$};
\draw[-Stealth] (0.9,-1) -- (1.1,-1) [->];
\draw (1,-1) node[below]{$[g_{\cev{e}}]$};
\draw (1,-2) node[below]{$\text{Loop}$};

% Bridge locus
\draw[fill=black] (5,0) circle[radius=3pt];
\draw (5,-0.15) node[below]{$g_1$};
\draw (5,0) node[left]{$G_{v_1}$};
\draw (5,0) -- (5,0.5);
\draw (5,0.75) node{$I_1$};
\draw (5,0) -- (7,0);
\draw[-Stealth] (5.4,0) -- (5.6,0) [->];
\draw (5.5,0) node[above]{$[g_{\vec{e}}]$};
\draw[-Stealth] (6.6,0) -- (6.4,0) [->];
\draw (6.5,0) node[above]{$[g_{\cev{e}}]$};
\draw[fill=black] (7,0) circle[radius=3pt];
\draw (7,-0.15) node[below]{$g_2$};
\draw (7,0) node[right]{$G_{v_2}$};
\draw (7,0) -- (7,0.5);
\draw (7,0.75) node{$I_2$};
\draw (6,-2) node[below]{$\text{Bridge}$};

\end{tikzpicture}
\]
The vertexwise type $\Theta_{V(\Gamma)}$ (\zcref{def: vertexwise type}) constitutes the data of subgroups $G_v \leqslant S_d$ for each vertex and conjugacy classes $[g_h] \subseteq G_v$ for each half-edge $h \in H_v(\Gamma)$.

Fix such a divisorial vertexwise $S_d$-graph $(\Gamma,\Theta_{V(\Gamma)})$. As in \zcref{sec: ramification profile boundary stratum}, the conjugacy classes $[g_{\vec{e}}],[g_{\cev{e}}]$ determine a single conjugacy class in $S_d$ which corresponds to a positive partition
\[ m_e = (m_{e1},\ldots,m_{el_e}) \vdash d.\]
Over the boundary divisor we consider the forgetful map to the moduli space of curves
\[ t \colon \Mcalbar_{(\Gamma,\Theta_{V(\Gamma)})}^{\Theta}(\Bcal S_d) \to \Mcalbar_\Gamma. \]
On the base, there are half-edge psi classes $\psi_{\vec{e}}, \psi_{\cev{e}}$ whose sum gives the first Chern class of the conormal bundle of $\Mcalbar_\Gamma$. We pull these back along $t$ to obtain
\[ t^\star \psi_{\vec{e}}, t^\star \psi_{\cev{e}}.\]
These are the psi classes associated to the node $q_e \in B$. Their sum is now the first Chern class of the conormal bundle of $j_{(\Gamma,\Theta_{V(\Gamma)})}$. Under the correspondence \eqref{eqn: splitting correspondence} in the splitting formalism, they pull back to the psi classes associated to the two markings obtained after splitting the node.

Finally we require notation for the markings. Recall that we have marked branch points $b_1,\ldots,b_n \in B$ on the base curve. For each $i \in [n]$ we let
\[ m_i = (m_{i1},\ldots,m_{il_i}) \vdash d \]
denote the associated ramification profile (see again \zcref{sec: ramification profile boundary stratum}). Here $l_i$ is the length of the partition. Recall also (\zcref{sec: Prym varieties}) that $b_0(C)$ denotes the number of connected components of $C$. We are now ready to state the key comparison theorem.

\begin{theorem} \label{thm: comparison}	We have the following identity on $\Mcalbar^{\Theta}_{g,n}(\Bcal S_d)$:
\begin{align*}
	\ch^\vee(\EE_C) = \ & d \cdot t^\star \ch^\vee(\EE) + b_0(C) - d + \dfrac{1}{2} \sum_{i=1}^n (d-l_i) \ + \\
	& -\sum_{i=1}^n \bigg( \sum_{k \geqslant 1} \dfrac{B_{2k}}{(2k)!} \sum_{j=1}^{l_i} \dfrac{m_{ij}^{2k}-1}{m_{ij}^{2k-1}} \bigg) t^\star \psi_i^{2k-1} + \\	
	& \sum_{(\Gamma,\Theta_{V(\Gamma)})} (j_{(\Gamma,\Theta_{V(\Gamma)})})_\star \left( \sum_{k \geqslant 1} \bigg( \dfrac{B_{2k}}{(2k)!} \cdot \lcm(m_e) \cdot \sum_{i=1}^{l_e}  \dfrac{m_{ei}^{2k}-1}{m_{ei}^{2k-1}} \bigg) \cdot (t^\star \psi_{\vec{e}}^{2k-2} - t^\star (\psi_{\vec{e}}^{2k-3} \psi_{\cev{e}}) + \cdots + t^\star \psi_{\cev{e}}^{2k-2} ) \right)
\end{align*}
where the third line is a sum over divisorial vertexwise $S_d$-graphs $(\Gamma,\Theta_{V(\Gamma)})$.
\end{theorem}

This result allows us to compute source Hodge integrals recursively. For Prym Hodge integrals, we use instead the following immediate corollary:
\begin{theorem}[Theorem~\ref{thm: comparison introduction}] \label{thm: Etilde in terms of EB} We have the following identity on $\Mcalbar^{\Theta}_{g,n}(\Bcal S_d)$:
\begin{align*}
	\ch^\vee(\Etilde) = \ & (d-1) \cdot t^\star \ch^\vee(\EE) + b_0(C) - d + \dfrac{1}{2} \sum_{i=1}^n (d-l_i) \ + \\
	& - \sum_{i=1}^n \bigg( \sum_{k \geqslant 1} \dfrac{B_{2k}}{(2k)!} \sum_{j=1}^{l_i} \dfrac{m_{ij}^{2k}-1}{m_{ij}^{2k-1}} \bigg) t^\star \psi_i^{2k-1} + \\	
	& \sum_{(\Gamma,\Theta_{V(\Gamma)})} (j_{(\Gamma,\Theta_{V(\Gamma)})})_\star \left( \sum_{k \geqslant 1} \bigg( \dfrac{B_{2k}}{(2k)!} \cdot \lcm(m_e) \cdot \sum_{i=1}^{l_e}  \dfrac{m_{ei}^{2k}-1}{m_{ei}^{2k-1}} \bigg) \cdot (t^\star \psi_{\vec{e}}^{2k-2} - t^\star ( \psi_{\vec{e}}^{2k-3} \psi_{\cev{e}} ) + \cdots + t^\star \psi_{\cev{e}}^{2k-2} ) \right).
\end{align*}
\end{theorem}
\begin{proof} Combine \zcref{thm: comparison} with \eqref{eqn: Etilde in terms of C and B}.	
\end{proof}

\begin{remark}
Notice in particular that boundary divisors with $m_e=(1,\ldots,1)$, i.e. divisors for which the map is unramified over the node, do not contribute. 
\end{remark}

The remainder of \zcref{sec: comparison} is dedicated to the proof of \zcref{thm: comparison}. In \zcref{sec: algorithm} we will then show how this result leads to a recursive algorithm for computing Prym chi integrals.

\subsection{Grothendieck--Riemann--Roch setup}
Consider the universal degree-$d$ cover over the moduli space (\zcref{sec: G-cover to d-cover}):
\begin{equation} \label{eqn: universal degree d cover}
\begin{tikzcd}[column sep=tiny]
C \ar[rr,"f"] \ar[rd,"p_C"{xshift={-0.4cm}, yshift={-0.4cm}}] & & B \ar[ld,"p_B"] \\
& \Mcalbar_{g,n}^{\Theta}(\Bcal S_d). &
\end{tikzcd}
\end{equation}
The Hodge bundles $\EE_C$ and $\EE_B$ defined in \zcref{sec: Hodge bundles on moduli space} can alternatively be described as pushforwards of the relative dualising sheaves of the universal curves:
\[ \EE_C = p_{C\star} \omega_{p_C}, \qquad \EE_B = p_{B\star} \omega_{p_B}.\]
By Serre duality we then have:
\[ \EE_C^\dual = R^1 p_{C\star} \OO_C, \qquad \EE_B^\dual = R^1 p_{B\star} \OO_B, \]
which gives the following identities in K-theory:
\begin{equation} \label{eqn: derived pushforward classes} \Rder p_{C\star} \OO_C = \OO^{\oplus b_0(C)} - \EE_C^\dual, \qquad \Rder p_{B\star} \OO_B = \OO - \EE_B^\dual.\end{equation}
We will compare $\ch^\vee(\EE_C)$ and $\ch^\vee(\EE_B)$ by applying Grothendieck--Riemann--Roch to each of the classes \eqref{eqn: derived pushforward classes} and comparing the resulting formulae. The difficult part is comparing the Todd classes of the tangent sheaves $T_{p_C}$ and $T_{p_B}$.

\subsection{Tropical description of the Todd class} \label{sec: tropical Todd} We assume familiarity with the language of piecewise polynomials on cone complexes, and their relationship to Chow classes on logarithmic schemes. A recap is provided in \zcref{sec: piecewise} at the end of this section.

For this section we fix a family of logarithmically smooth curves
\begin{equation} \label{eqn: family log smooth curves} p \colon C \to S.\end{equation}
We will later apply our results with both $p=p_C$ and $p=p_B$. Given a family \eqref{eqn: family log smooth curves} there is a short exact sequence on $C$,
\begin{equation} \label{eqn: exact sequence dualising and residue} 0 \to \Omega_p \to \omega_p \to R_p \to 0 \end{equation}
where $R_p$ is a torsion sheaf supported on the nodal locus of $C$. We have:
\begin{equation} \label{eqn: Todd of Tp in terms of omegas} \Td(T_p) = \Td^\vee(\omega_p) \Td^\vee(R_p)^{-1}. \end{equation}
We now give a tropical description of the class $\Td^\vee(R_p)^{-1}$ above. We will see (\zcref{lem: residue series formula}) that this description unravels to recover the universal formula for Todd classes of codimension-$2$ loci \cite[Lemma~5.1]{MumfordTowards}. However one of the advantage of the tropical language is that it allows us to recognise invariance properties under subdivisions, which allows us to state \zcref{prop: Todd is res} in the toroidal setting, generalising Mumford's result and sidestepping the need to work with a crepant resolution of the universal curve.

\begin{construction}[Residue series] \label{construction: residue series}
The family $p$ tropicalises to a morphism of cone complexes
\[ \Sigma_p \colon \Sigma_C \to \Sigma_S. \]
We will describe a piecewise power series $\res$ on $\Sigma_C$. Given a cone $\upsigma \in \Sigma_S$ the general fibre of $\Sigma_p$ is a graph which we denote $\Gamma_\upsigma$. Cones in $\Sigma_C$ are indexed by pairs $(\upsigma, w)$, where $\upsigma \in \Sigma_S$ and $w \in V(\Gamma_\upsigma) \sqcup E(\Gamma_\upsigma) \sqcup L(\Gamma_\sigma)$. We abuse notation by writing $(\sigma,w) \in \Sigma_C$.

We will focus on cones of the form $(\upsigma,e)$ with $e \in E(\Gamma_\upsigma)$. Letting $v_1,v_2$ denote the endpoints of $e$, a linear function on such a cone is equivalent to an assignment
\[ \{ v_1,v_2 \} \xrightarrow{\varphi} S_\upsigma \colonequals \Hom(\upsigma, \R_{\geqslant 0})_\NN \]
with integer slope, meaning that $\varphi(v_2)-\varphi(v_1) \in \ZZ \cdot \ell_e$ where $\ell_e \in S_\upsigma$ is the edge length function. Given such a cone $(\upsigma,e)$ we metrise the edge $e$ as an interval $[0,\ell_e]$ with coordinate $x$. We then consider the following linear functions on the cone $(\upsigma,e) \in \Sigma_C$,
\[
\begin{tikzpicture}

\draw[fill=black] (0,0) circle[radius=2pt];
\draw (0,0) -- (2,0);
\draw[fill=black] (2,0) circle[radius=2pt];
\draw (1,0) node[above]{\footnotesize$x$};
\draw[blue] (0,0) node[above]{\scriptsize$0$};
\draw[blue] (2,0) node[above]{\scriptsize$\ell_e$};
\draw (1,-0.75) node{$1_{\vec{e}}$};

\draw[fill=black] (4,0) circle[radius=2pt];
\draw (4,0) -- (6,0);
\draw[fill=black] (6,0) circle[radius=2pt];
\draw (5,0) node[above]{\footnotesize$\ell_e\!-\!x$};
\draw[blue] (4,0) node[above]{\scriptsize$\ell_e$};
\draw[blue] (6,0) node[above]{\scriptsize$0$};
\draw (5,-0.75) node{$1_{\cev{e}}$};

\draw[fill=black] (8,0) circle[radius=2pt];
\draw (8,0) -- (10,0);
\draw[fill=black] (10,0) circle[radius=2pt];
\draw (9,0) node[above]{\footnotesize$\ell_e$};
\draw[blue] (8,0) node[above]{\scriptsize$\ell_e$};
\draw[blue] (10,0) node[above]{\scriptsize$\ell_e$};
\draw (9,-0.75) node{$1_{\vec{e}}+1_{\cev{e}}$};
\end{tikzpicture}
\]
where we note that the final linear function is simply $\ell_e = \Sigma_p^\star \ell_e$. To each of these we apply the Todd operator (see \zcref{sec: piecewise}) to produce the following power series on $(\upsigma,e)$:
\[
\begin{tikzpicture}

\draw[fill=black] (0,0) circle[radius=2pt];
\draw (0,0) -- (2,0);
\draw[fill=black] (2,0) circle[radius=2pt];
\draw (1,0) node[above]{\footnotesize$\Td(x)$};
\draw[blue] (0,0) node[above]{\scriptsize$1$};
\draw[blue] (2,0) node[above]{\scriptsize$\Td(\ell_e)$};
\draw (1,-0.75) node{$\Td(1_{\vec{e}})$};

\draw[fill=black] (4,0) circle[radius=2pt];
\draw (4,0) -- (6,0);
\draw[fill=black] (6,0) circle[radius=2pt];
\draw (5,0) node[above]{\footnotesize$\Td(\ell_e\!-\!x)$};
\draw[blue] (3.9,0) node[above]{\scriptsize$\Td(\ell_e)$};
\draw[blue] (6,0) node[above]{\scriptsize$1$};
\draw (5,-0.75) node{$\Td(1_{\cev{e}})$};

\draw[fill=black] (8,0) circle[radius=2pt];
\draw (8,0) -- (10,0);
\draw[fill=black] (10,0) circle[radius=2pt];
\draw (9,0) node[above]{\footnotesize$\Td(\ell_e)$};
\draw[blue] (8,0) node[above]{\scriptsize$\Td(\ell_e)$};
\draw[blue] (10,0) node[above]{\scriptsize$\Td(\ell_e)$};
\draw (9,-0.75) node{$\Td(1_{\vec{e}}+1_{\cev{e}})$};
\end{tikzpicture}
\]
We then define $\res(1_{\vec{e}},1_{\cev{e}})$ as the following power series on $(\upsigma,e)$:
\[ \res(1_{\vec{e}},1_{\cev{e}}) \colonequals \dfrac{\Td(1_{\vec{e}}) \Td(1_{\cev{e}})}{\Td(1_{\vec{e}}+1_{\cev{e}})}.  \]

We now globalise the above description. We begin by fixing a cone $\upsigma \in \Sigma_S$ and studying the subcomplex $\Sigma_C|_\upsigma \hookrightarrow \Sigma_C$. For each $e \in E(\Gamma_\upsigma)$ the piecewise power series $\res(1_{\vec{e}},1_{\cev{e}})$ has value $1$ at the endpoints $v_1,v_2$. Hence they glue into a piecewise power series on the cone complex $\Sigma_C|_\upsigma$, which we can view as a piecewise power series on $\Gamma_\upsigma$ taking values in $S_\upsigma$.

Varying over cones $\upsigma \in \Sigma_S$ we see that the piecewise power series defined on each $\Sigma_C|_\upsigma$ are compatible with face inclusions $\upsigma^\prime \subseteq \upsigma$, since these correspond to edge contractions $\Gamma_\upsigma \to \Gamma_\upsigma^\prime$. We thus obtain a globally-defined piecewise power series on $\Sigma_C$, which we denote
\[ \res \in \PPhat^\star(\Sigma_C) \]
and refer to as the \textbf{residue series}.
\end{construction}

We require the following algebraic lemma.
\begin{lemma} \label{lem: residue series formula} The power series
\[ \res(x,y) \colonequals \dfrac{\Td(x) \Td(y)}{\Td(x+y)} \in \QQ \llbracket x,y \rrbracket\]
can be expressed as
\[ \res(x,y) = 1 + xy \sum_{k \geqslant 1} \dfrac{B_{2k}}{(2k)!} \left( x^{2k-2} - x^{2k-3}y + \cdots - xy^{2k-3} + y^{2k-2} \right). \]

\end{lemma}

For later use, we write the trailing power series as
\begin{equation} \label{eqn: P series} P(x,y) \colonequals \sum_{k \geqslant 1} \dfrac{B_{2k}}{(2k)!} \left( x^{2k-2} - x^{2k-3}y + \cdots - xy^{2k-3} + y^{2k-2} \right),\end{equation}
so the statement of the lemma becomes
\[
\res(x,y) = 1 + xyP(x,y).
\]

\begin{proof} It is straightforward to verify the following algebraic identities:
\[ \dfrac{1-e^{-x-y}}{(1-e^{-x})(1-e^{-y})} = \dfrac{1}{1-e^{-x}} + \dfrac{1}{e^y-1}, \qquad \Td(-y) = \Td(y)-y.\]
Given these, we compute:
\begin{align*}
\res(x,y) & = \dfrac{x}{1-e^{-x}} \cdot \dfrac{y}{1-e^{-y}} \cdot \dfrac{1-e^{-x-y}}{x+y} \\[0.2cm]
& = \dfrac{xy}{x+y} \cdot \dfrac{1-e^{-x-y}}{(1-e^{-x})(1-e^{-y})} \\[0.2cm]
& = \dfrac{xy}{x+y} \left( \dfrac{1}{1-e^{-x}} + \dfrac{1}{e^y-1} \right) \\[0.2cm]
& = \dfrac{y}{x+y} \Td(x) + \dfrac{x}{x+y} \cdot \dfrac{-y}{1-e^y} \\[0.2cm]
& = \dfrac{1}{x+y} \left( y \Td(x) + x \Td(-y) \right) \\[0.2cm]
& = \dfrac{1}{x+y} \left( y \Td(x) + x \Td(y) - xy \right).
\end{align*}
We then expand this out in terms of Bernoulli numbers:
\begin{align*}
\res(x,y) & = \dfrac{1}{x+y} \left( y + x + \dfrac{yx}{2} + \dfrac{xy}{2} - xy + \sum_{k \geqslant 1} \dfrac{B_{2k}}{(2k)!} (yx^{2k} + xy^{2k}) \right) \\[0.2cm]
& = 1 + xy \sum_{k \geqslant 1}	\dfrac{B_{2k}}{(2k)!} \left( \dfrac{x^{2k-1}+y^{2k-1}}{x+y} \right) \\[0.2cm]
& = 1 + xy \sum_{k \geqslant 1} \dfrac{B_{2k}}{(2k)!} \left( x^{2k-2} - x^{2k-3}y + \cdots - xy^{2k-3} + y^{2k-2} \right). \qedhere
\end{align*}
\end{proof}

We now come to the key formula.
\begin{proposition} \label{prop: Todd is res} The Chow class associated to the series $\res$ (\zcref{construction: residue series}) is the inverse dual Todd class of $R_p$:
\[ \Phi(\res) = \Td^\vee(R_p)^{-1}.\]	
\end{proposition}

\begin{proof}
When $C$ and $S$ are both smooth this is a translation of \cite[(5.11)~and~(5.12)]{ACGII} in the piecewise polynomial language. We claim however that the formula remains correct for an arbitrary family of logarithmic curves.

For simplicity we give the proof only in the case of interest to us, which is when $\Sigma_C$ is simplicial. First, we note that the residue sheaf is concentrated on the nodal locus of $C$, and so can be studied local to a node. Since piecewise polynomials satisfy \'etale descent, we can check the formula locally on $\Sigma_C$, and thus we can assume $\Sigma_S=\sigma$ is a ray and $\Sigma_C = (\sigma,e)$ is a single cone corresponding to an edge $e \in E(\Gamma_\sigma)$. We let $\ell_e \in \PL(\Sigma_S)$ denote the length of the edge $e$.

We must show that the formula is invariant under subdivisions of $\Sigma_C = (\sigma,e)$. By induction, it is enough to work with a subdivision that introduces a single additional vertex to $e$. We write
\[ b \colon \Sigma_C^\prime \to \Sigma_C \]
for this subdivision, and introducing accompanying notation:
\[
\begin{tikzpicture}

\draw (-0.5,2) node[left]{$\Sigma_C^\prime$};
\draw[fill=black] (0,2) circle[radius=2pt];
\draw (0,2) -- (2,2);
\draw[fill=black] (2,2) circle[radius=2pt];
\draw[fill=black] (1,2) circle[radius=2pt];
\draw (1,2) node[above]{\scriptsize$v_0$};
\draw (0,2) node[above]{\scriptsize$v_1$};
\draw (2,2) node[above]{\scriptsize$v_2$};
\draw (0.5,2) node[below]{\scriptsize$e_1$};
\draw (1.5,2) node[below]{\scriptsize$e_2$};
\draw [->] (1,1.5) -- (1,0.5);
\draw (1,1) node[right]{\scriptsize$b$};

\draw (-0.5,0) node[left]{$\Sigma_C$};
\draw[fill=black] (0,0) circle[radius=2pt];
\draw (0,0) -- (2,0);
\draw[fill=black] (2,0) circle[radius=2pt];
\draw (1,0) node[below]{\footnotesize$e$};
\draw (0,0) node[above]{\scriptsize$v_1$};
\draw (2,0) node[above]{\scriptsize$v_2$};
	
\end{tikzpicture}
\]
where $\ell_e = \ell_{e_1} + \ell_{e_2}$. From \zcref{lem: residue series formula} we see that on $\Sigma_C^\prime$ the $k$th term of $\res$ is:
\begin{align}
\nonumber & \dfrac{B_{2k}}{(2k)!} \left( 1_{\vec{e}_1} 1_{\cev{e}_1} \big( 1_{\vec{e}_1}^{2k-2} - \cdots + 1_{\cev{e}_1}^{2k-2} \big) + 1_{\vec{e}_2} 1_{\cev{e}_2} \big( 1_{\vec{e}_2}^{2k-2} - \cdots + 1_{\cev{e}_2}^{2k-2} \big) \right) \\[0.2cm]
\nonumber = \ & \dfrac{B_{2k}}{(2k)!} \left( 1_{\vec{e}_1} 1_{\cev{e}_1} \dfrac{ 1_{\vec{e}_1}^{2k-1} + 1_{\cev{e}_1}^{2k-1} }{ 1_{\vec{e}_1} + 1_{\cev{e}_1} } + 1_{\vec{e}_2} 1_{\cev{e}_2} \dfrac{ 1_{\vec{e}_2}^{2k-1} + 1_{\cev{e}_2}^{2k-1}}{ 1_{\vec{e}_2} + 1_{\cev{e}_2} } \right) \\[0.2cm]
\label{eqn: proof residue equals Todd formula for two edges} = \ & \dfrac{B_{2k}}{(2k)!} \left( 1_{\vec{e}_1} 1_{\cev{e}_1} \dfrac{ 1_{\vec{e}_1}^{2k-1} + 1_{\cev{e}_1}^{2k-1} }{ \ell_{e_1} } + 1_{\vec{e}_2} 1_{\cev{e}_2} \dfrac{ 1_{\vec{e}_2}^{2k-1} + 1_{\cev{e}_2}^{2k-1}}{ \ell_{e_2} } \right).
\end{align}
We now introduce the following piecewise-linear functions on $\Sigma_C$,
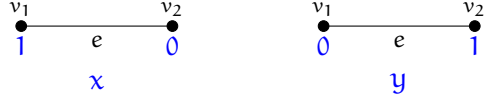
\[
\begin{tikzpicture}

\draw[fill=black] (0,0) circle[radius=2pt];
\draw (0,0) -- (2,0);
\draw[fill=black] (2,0) circle[radius=2pt];
\draw (1,0) node[below]{\footnotesize$e$};
\draw (0,0) node[above]{\scriptsize$v_1$};
\draw (2,0) node[above]{\scriptsize$v_2$};
\draw[blue] (0,0) node[below]{\small$1$};
\draw[blue] (2,0) node[below]{\small$0$};
\draw[blue] (1,-0.75) node{$x$};

\draw[fill=black] (4,0) circle[radius=2pt];
\draw (4,0) -- (6,0);
\draw[fill=black] (6,0) circle[radius=2pt];
\draw (5,0) node[below]{\footnotesize$e$};
\draw (4,0) node[above]{\scriptsize$v_1$};
\draw (6,0) node[above]{\scriptsize$v_2$};
\draw[blue] (4,0) node[below]{\small$0$};
\draw[blue] (6,0) node[below]{\small$1$};
\draw[blue] (5,-0.75) node{$y$};

\end{tikzpicture}
\]
and the following piecewise-linear functions on $\Sigma_C^\prime$:
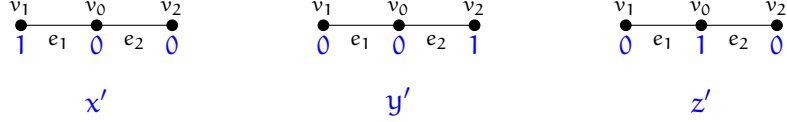
\[
\begin{tikzpicture}

\draw[fill=black] (0,2) circle[radius=2pt];
\draw (0,2) -- (2,2);
\draw[fill=black] (2,2) circle[radius=2pt];
\draw[fill=black] (1,2) circle[radius=2pt];
\draw (1,2) node[above]{\scriptsize$v_0$};
\draw (0,2) node[above]{\scriptsize$v_1$};
\draw (2,2) node[above]{\scriptsize$v_2$};
\draw (0.5,2) node[below]{\scriptsize$e_1$};
\draw (1.5,2) node[below]{\scriptsize$e_2$};
\draw[blue] (0,2) node[below]{\small$1$};
\draw[blue] (1,2) node[below]{\small$0$};
\draw[blue] (2,2) node[below]{\small$0$};
\draw[blue] (1,1) node{$x^\prime$};

\draw[fill=black] (4,2) circle[radius=2pt];
\draw (4,2) -- (6,2);
\draw[fill=black] (6,2) circle[radius=2pt];
\draw[fill=black] (5,2) circle[radius=2pt];
\draw (5,2) node[above]{\scriptsize$v_0$};
\draw (4,2) node[above]{\scriptsize$v_1$};
\draw (6,2) node[above]{\scriptsize$v_2$};
\draw (4.5,2) node[below]{\scriptsize$e_1$};
\draw (5.5,2) node[below]{\scriptsize$e_2$};
\draw[blue] (4,2) node[below]{\small$0$};
\draw[blue] (5,2) node[below]{\small$0$};
\draw[blue] (6,2) node[below]{\small$1$};
\draw[blue] (5,1) node{$y^\prime$};

\draw[fill=black] (8,2) circle[radius=2pt];
\draw (8,2) -- (10,2);
\draw[fill=black] (10,2) circle[radius=2pt];
\draw[fill=black] (9,2) circle[radius=2pt];
\draw (9,2) node[above]{\scriptsize$v_0$};
\draw (8,2) node[above]{\scriptsize$v_1$};
\draw (10,2) node[above]{\scriptsize$v_2$};
\draw (8.5,2) node[below]{\scriptsize$e_1$};
\draw (9.5,2) node[below]{\scriptsize$e_2$};
\draw[blue] (8,2) node[below]{\small$0$};
\draw[blue] (9,2) node[below]{\small$1$};
\draw[blue] (10,2) node[below]{\small$0$};
\draw[blue] (9,1) node{$z^\prime$};
\end{tikzpicture}
\]
We then have the following identities, relating the above piecewise-linear functions to those appearing in \eqref{eqn: proof residue equals Todd formula for two edges}:
\begin{alignat*}{5}
\text{On $e_1$}: & \qquad && \ell_1 x^\prime = 1_{\cev{e}_1},\quad && \ell_1 z^\prime = 1_{\vec{e}_1}. \\
\text{On $e_2$}: & \qquad && \ell_2 y^\prime =1_{\vec{e}_2},\quad && \ell_2 z^\prime = 1_{\cev{e}_2}. \\
\text{On $e$}: & \qquad && \ell_e x = 1_{\cev{e}}, \quad && \ell_e y = 1_{\vec{e}}.
\end{alignat*}
On the other hand, we also have identities relating these piecewise-linear functions to each other:
\begin{alignat*}{5}
\text{On $e_1$}: & \qquad && x^\prime = x - \dfrac{\ell_2}{\ell_1} y, \quad && z^\prime = \dfrac{\ell_e}{\ell_1} y.\\[0.2cm]
\text{On $e_2$}: & \qquad && y^\prime = y - \dfrac{\ell_1}{\ell_2} x, \quad && z^\prime = \dfrac{\ell_e}{\ell_2} x.
\end{alignat*}
We now compute the pushforward of \eqref{eqn: proof residue equals Todd formula for two edges} along $b$. We use the formula \cite[Section~3.5]{BrionVergne} for pushforwards of piecewise polynomials along subdivisions of simplicial cone complexes. In our particular setting this gives
\[ b_\star F = \dfrac{1_{\vec{e}} 1_{\cev{e}}}{\ell_e} \left( \dfrac{\ell_1 F|_{e_1}}{1_{\vec{e}_1} 1_{\cev{e}_1}} + \dfrac{\ell_2 F|_{e_2}}{1_{\vec{e}_2} 1_{\cev{e}_2}} \right) .\]
Applying this to \eqref{eqn: proof residue equals Todd formula for two edges} and using the above relations, we compute:
\begin{align*}
   & b_\star \left( 1_{\vec{e}_1}1_{\cev{e}_1} \dfrac{1_{\vec{e}_1}^{2k-1}+1_{\cev{e}_1}^{2k-1}}{\ell_1} + 1_{\vec{e}_2}1_{\cev{e}_2} \dfrac{1_{\vec{e}_2}^{2k-1}+1_{\cev{e}_2}^{2k-1}}{\ell_2} \right) \\[0.2cm]
= & \ \dfrac{1_{\vec{e}} 1_{\cev{e}}}{\ell_e} \left( 1_{\vec{e}_1}^{2k-1} + 1_{\cev{e}_1}^{2k-1} + 1_{\vec{e}_2}^{2k-1} + 1_{\cev{e}_2}^{2k-1} \right) \\[0.2cm]
= & \ \dfrac{1_{\vec{e}} 1_{\cev{e}}}{\ell_e} \left( \ell_1^{2k-1} \bigg(\dfrac{\ell_e}{\ell_1}y \bigg)^{2k-1} + \ell_1^{2k-1} \bigg(x - \dfrac{\ell_2}{\ell_1}y \bigg)^{2k-1} + \ell_2^{2k-1} \bigg( y - \dfrac{\ell_1}{\ell_2}x \bigg)^{2k-1} + \ell_2^{2k-1} \bigg( \dfrac{\ell_e}{\ell_2} x \bigg)^{2k-1} \right)  \\[0.2cm]
= & \ \dfrac{1_{\vec{e}} 1_{\cev{e}}}{\ell_e} \bigg( \ell_e^{2k-1}y^{2k-1} + \ell_e^{2k-1}x^{2k-1} \bigg) \\[0.2cm]
= & \ 1_{\vec{e}} 1_{\cev{e}} \left( \dfrac{1_{\vec{e}}^{2k-1}+1_{\cev{e}}^{2k-1}}{\ell_e} \right)
\end{align*}
where between the third and fourth lines the two middle terms cancel because $2k-1$ is odd. This shows $b_\star \res = \res$ as required.\end{proof}

\subsection{Structure of the Todd class} We investigate the consequences of the above description. Consider the nodal locus in $C$. Since nodes do not intersect, this is a disjoint union
\[ \bigsqcup_{(\upsigma,e)} q_{(\upsigma,e)} \hookrightarrow C \]
where the nodes are indexed by pairs $(\upsigma,e)$ with $\upsigma \in \Sigma_S(1)$ and $e \in E(\Gamma_\upsigma)$. Each node maps isomorphically onto the boundary divisor $S(\sigma) \subseteq S$,
\[ q_{(\upsigma,e)} \xrightarrow{\cong} S(\upsigma) \]
and there is a codimension-$2$ inclusion:
\[ j_{(\upsigma,e)} \colon q_{(\upsigma,e)} \hookrightarrow C.\]

\begin{proposition} \label{prop: boundary divisor formula for residue class} We have the following identity on $C$:
\begin{equation} \label{eqn: divisorial formula for res} \Td^\vee(R_p)^{-1} = 1 + \sum_{(\upsigma,e)} j_{(\upsigma,e)\star} \left( P(x_1,x_2) \right) \end{equation}
where $P$ is the power series defined in \eqref{eqn: P series}. The sum is over cones $\upsigma \in \Sigma_S(1)$ and edges $e \in E(\Gamma_\upsigma)$, and for each such pair, $x_1$ and $x_2$ are the Chern roots of the normal bundle of the inclusion $j_{(\upsigma,e)}$.
\end{proposition}

\begin{proof}
\zcref{lem: residue series formula} gives the following description of the residue series on the cone $(\sigma,e) \in \Sigma_C$:
\[ \res|_{(\upsigma,e)} = 1 + 1_{\vec{e}} 1_{\cev{e}} \cdot P(1_{\vec{e}},1_{\cev{e}}).\]
Here the product $1_{\vec{e}} 1_{\cev{e}}$ corresponds to the fundamental class $j_{(\upsigma,e)\star} (1) \in CH^2(C)$ while the term $P(1_{\vec{e}},1_{\cev{e}})$ corresponds to a power series in the Chern roots of the normal bundle of $j_{(\upsigma,e)}$.	
\end{proof}

We obtain the following straightforward but important consequence:

\begin{corollary} \label{cor: product of canonical with residue} We have
\[ \Td(T_p) = \Td^\vee(\omega_p) + \Td^\vee(R_p)^{-1}_{\geqslant 1}.\]	
\end{corollary}

\begin{proof}
From \eqref{eqn: Todd of Tp in terms of omegas} we have:
\begin{align*} \Td(T_p) & = \Td^\vee(\omega_p) \Td^\vee(R_p)^{-1} \\
& = 1 + \Td^\vee(\omega_p)_{\geqslant 1} + \Td^\vee(R_p)^{-1}_{\geqslant 1} + \Td^\vee(\omega_p)_{\geqslant 1} \Td^\vee(R_p)^{-1}_{\geqslant 1}.	
\end{align*}
We claim that the final term vanishes. Indeed, by \zcref{prop: boundary divisor formula for residue class}, it is equal to:
\begin{equation} \label{eqn: residue sheaf comparison formula vanishing term}  \Td^\vee(\omega_p)_{\geqslant 1} \sum_{(\upsigma,e)} j_{(\upsigma,e)\star} \left( P(x_1,x_2) \right).\end{equation}
But we claim that $\omega_p$ is trivial when restricted to the nodal locus. Indeed letting $Z \subseteq C$ denote the nodal locus, a local analysis of the residue sequence
\[ 0 \to \Omega_p \to \omega_p \to \OO_Z \to 0 \]
shows that $\Omega_p \cong \omega_p \otimes I_{Z|C}$ via the above map (see e.g. \cite[(2.20)]{ACGII} and \cite[Section~4]{FarkasBirationalAspects}). The above exact sequence thus becomes
\[ 0 \to \omega_p \otimes I_{Z|C} \to \omega_p \to \OO_Z \to 0 \]
while on the other hand the ideal sheaf sequence for $Z$ gives
\[ 0 \to \omega_p \otimes I_{Z|C} \to \omega_p \to \omega_p|_Z \to 0 \]
and so we conclude that $\omega_p|_Z \cong \OO_Z$. The term \eqref{eqn: residue sheaf comparison formula vanishing term} then vanishes by the projection formula. We conclude that:
\[ \Td(T_p) = 1 + \Td^\vee(\omega_p)_{\geqslant 1} + \Td^\vee(R_p)^{-1}_{\geqslant 1} = \Td^\vee(\omega_p) + \Td^\vee(R_p)^{-1}_{\geqslant 1}. \qedhere \]
\end{proof}

\subsection{Comparing residue sheaves} We now return to the universal family \eqref{eqn: universal degree d cover} over the moduli space of $S_d$-covers. Following on from \zcref{cor: product of canonical with residue}, we wish to compare the classes:
\[ p_{C\star} \Td^\vee(R_{p_C})^{-1}_{\geqslant 1} \qquad \text{and} \qquad p_{B\star} \Td^\vee(R_{p_B})^{-1}_{\geqslant 1}.\]

\begin{proposition} \label{prop: residue sheaf comparison} With notation as in \zcref{sec: statement of the comparison}, we have:
\begin{align*} p_{C\star} \Td^\vee(R_{p_C})^{-1}_{\geqslant 1} = &\ d \cdot p_{B\star} \Td^\vee(R_{p_B})^{-1}_{\geqslant 1} \ + \\
& \sum_{(\Gamma,\Theta_{V(\Gamma)})} (j_{(\Gamma,\Theta_{V(\Gamma)})})_\star \left( \sum_{k \geqslant 1} \bigg( \dfrac{B_{2k}}{(2k)!} \cdot \lcm(m_e) \cdot \sum_{i=1}^{l_e}  \dfrac{1-m_{ei}^{2k}}{m_{ei}^{2k-1}} \bigg) \cdot (t^\star \psi_{\vec{e}}^{2k-2} - \cdots + t^\star \psi_{\cev{e}}^{2k-2} ) \right) \end{align*}
where $\lcm(m_e) \colonequals \lcm(m_{e1},\ldots,m_{el_e})$. The sum is over vertexwise $S_d$-graphs $(\Gamma,\Theta_{V(\Gamma)})$ with $|E(\Gamma)|=1$, and $j_{(\Gamma,\Theta_{V(\Gamma)})}$ is the inclusion of the associated boundary divisor.
\end{proposition}

\begin{proof}
The universal family \eqref{eqn: universal degree d cover} tropicalises to a family of tropical maps
\[
\begin{tikzcd}
    \Sigma_C \ar[rr,"\Sigma_f"] \ar[rd,"\Sigma_{p_c}" left] & & \Sigma_B \ar[ld,"\Sigma_{p_B}"] \\
    & \Sigma_{\Mcalbar_{g,n}^{\Theta}(\Bcal S_d)}.
\end{tikzcd}
\]
There are piecewise power series $\res$ on both $\Sigma_C$ and $\Sigma_B$. We wish to study the difference
\[ \res_{\geqslant 1} - \Sigma_f^\star \res_{\geqslant 1}.\]
By \zcref{prop: boundary divisor formula for residue class} it is sufficient to work over boundary divisors in the moduli space. Fix therefore a divisorial vertexwise $S_d$-graph $(\Gamma,\Theta_{V(\Gamma)})$. Over the associated boundary divisor in the moduli space there is a universal node inside the universal base curve $B$:
\[
\begin{tikzcd}
 q_e \ar[r,hook] \ar[rd] & B \ar[d,"p_B"] \\
 & \Mcalbar^{\Theta}_{(\Gamma,\Theta_{V(\Gamma)})}(\Bcal S_d).
\end{tikzcd}
\]
Recall from \zcref{sec: statement of the comparison} that we write $m_e=(m_{e1},\ldots,m_{el_e})$ for the ramification profile over $q_e$. After choosing a labelling of the components of the preimage, there are corresponding universal nodes inside the universal curve $C$
\[
\begin{tikzcd}
 q_i \ar[r,hook] \ar[rd] & C \ar[d,"p_C"] \\
 & \Mcalbar^{\Theta}_{(\Gamma,\Theta_{V(\Gamma)})}(\Bcal S_d)
\end{tikzcd}
\]
for $i \in [l_e]$. Moreover on each branch of $C$ at $q_i$ the map $f \colon C \to B$ is ramified to order $m_{ei}$. Now let
\[ \sigma \in \Sigma_{\Mcalbar_{g,n}^{\Theta}(\Bcal S_d)} \]
denote the ray associated to the chosen boundary divisor. Let $\Gamma_B,\Gamma_C$ denote the graphs obtained respectively as the general fibres of $\Sigma_{p_B},\Sigma_{p_C}$ over $\sigma$, and let $e \in E(\Gamma_B)$, respectively $e_i \in E(\Gamma_C)$ denote the edges corresponding to the nodes $q_e \in B$, respectively $q_i \in C$ introduced above.

For $i \in [l_e]$ we then consider the associated cone $(\sigma,e_i) \in \Sigma_C$. On this cone we have
\[ \Sigma_f^\star 1_{\vec{e}} = m_{ei} 1_{\vec{e}_i}, \qquad \Sigma_f^\star 1_{\cev{e}} = m_{ei} 1_{\cev{e}_i}, \]
from which we obtain:
\[ (\Sigma_f^\star \res_{\geqslant 1})|_{(\sigma,e_i)} = m_{ei}^2 \cdot 1_{\vec{e}_i} 1_{\cev{e}_i} \sum_{k \geqslant 1} \dfrac{B_{2k}}{(2k)!} m_{ei}^{2k-2} (1_{\vec{e}_i}^{2k-2} - \cdots + 1_{\cev{e}_i}^{2k-2}) \]
This then gives:
\begin{equation} \label{eqn: GRR calculation comparison of res} (\res_{\geqslant 1} - \Sigma_f^\star \res_{\geqslant 1})|_{(\sigma,e_i)} = 1_{\vec{e}_i} 1_{\cev{e}_i} \sum_{k \geqslant 1} \dfrac{B_{2k}}{(2k)!}  (1 - m_{ei}^{2k}) (1_{\vec{e}_i}^{2k-2} - \cdots + 1_{\cev{e}_i}^{2k-2}).\end{equation}
We will now compute the associated Chow class. Consider the diagram
\begin{equation} \label{eqn: diagram residue sheaf comparison proof}
\begin{tikzcd}
q_i \ar[r,hook,"j_i"] \ar[d,"\cong"] & C \ar[d,"p_C"] \\
\Mcalbar^{\Theta}_{(\Gamma,\Theta_{V(\Gamma)})}(\Bcal S_d) \ar[r,hook,"j_{(\Gamma,\Theta_{V(\Gamma)})}"] & \Mcalbar^{\Theta}_{g,n}(\Bcal S_d)
\end{tikzcd}
\end{equation}
where $j_i$ is a codimension-$2$ inclusion. We first claim that
\begin{equation} \label{eqn: pushforward multiplicity from node of C} \Phi(1_{\vec{e}_i} 1_{\cev{e}_i}) = \dfrac{\lcm(m_e)}{m_{ei}} \cdot j_{i\star}(1).\end{equation}
Examine the tropical map $\Gamma_C \to \Gamma_B$ over the edge $e \in E(\Gamma_B)$. We have the continuity relations
\[ m_{e1} \ell_1 = m_{e2} \ell_2 = \cdots = m_{el_e} \ell_{l_e} = \ell_e \]
where $\ell_i$ is the length of $e_i$ and $\ell_e$ is the length of $e$. The local tropical parameter on the moduli space is then
\[ \ell_e/\lcm(m_e) \]
and it follows that the cone $(\sigma,e_i)$ is an $A_{k_i}$-singularity where $k_i \colonequals \lcm(m_e)/m_{ei}$. This proves \eqref{eqn: pushforward multiplicity from node of C}. Applying $\Phi$ to the right-hand side of \eqref{eqn: GRR calculation comparison of res} then gives:
\[ j_{i\star}\left( \sum_{k \geqslant 1} \bigg( \dfrac{B_{2k}}{(2k)!} \lcm(m_e) \dfrac{1-m_{ei}^{2k}}{m_{ei}} \bigg) (x_1^{2k-2}-\cdots+x_2^{2k-2}) \right) \]
where $x_1,x_2$ are the Chern roots of the normal bundle of $j_i$. Pushing forward along $p_C$ and refactoring via the diagram \eqref{eqn: diagram residue sheaf comparison proof} gives:
\begin{align*} & (j_{(\Gamma,\Theta_{V(\Gamma)})})_\star \left( \sum_{k \geqslant 1} \bigg( \dfrac{B_{2k}}{(2k)!} \lcm(m_e) \dfrac{1-m_{ei}^{2k}}{m_{ei}} \bigg) \big( (-\psi_{\vec{e}_i})^{2k-2} - \cdots + (-\psi_{\cev{e}_i})^{2k-2} \big) \right) \\
= \ & (j_{(\Gamma,\Theta_{V(\Gamma)})})_\star \left( \sum_{k \geqslant 1} \bigg( \dfrac{B_{2k}}{(2k)!} \lcm(m_e) \dfrac{1-m_{ei}^{2k}}{m_{ei}} \bigg) (\psi_{\vec{e}_i}^{2k-2} - \cdots + \psi_{\cev{e}_i}^{2k-2}) \right)
\end{align*}
where the equality holds because the power series has only even-degree terms. Finally we have (e.g. via a comparison of piecewise-linear functions):
\[ \psi_{\vec{e}_i} = t^\star \psi_{\vec{e}}/m_{ei}, \qquad \psi_{\cev{e}_i} = t^\star \psi_{\cev{e}}/m_{ei}. \]
Substituting this in, we see that the contribution of $i \in [l_e]$ is precisely:
\[ (j_{(\Gamma,\Theta_{V(\Gamma)})})_\star \left( \sum_{k \geqslant 1} \bigg( \dfrac{B_{2k}}{(2k)!} \lcm(m_e) \dfrac{1-m_{ei}^{2k}}{m_{ei}^{2k-1}}\bigg) \cdot  (t^\star \psi_{\vec{e}}^{2k-2} - \cdots + t^\star \psi_{\cev{e}}^{2k-2} ) \right).\]
Summing over all $i \in [l_e]$ and then over all divisorial vertexwise $S_d$-graphs $(\Gamma,\Theta_{V(\Gamma)})$ gives the result.
\end{proof}

\subsection{Comparing dualising sheaves} We now compare the dualising sheaves. While for the Todd classes the difference was supported at the nodes, for the dualising sheaves the difference will be supported at the markings. 

We introduce the necessary notation. Recall that we have markings $b_1,\ldots,b_n \in B$ on the base curve. For each $i \in [n]$ we let
\[ m_i = (m_{i1},\ldots,m_{il_i}) \vdash d \]
denote the ramification profile, with associated ramification points
\[ c_{i1},\ldots,c_{il_i} \in C.\]

\begin{lemma} We have
\[  f^\star \omega_{p_B}^\vee = \omega_{p_C}^\vee \big( \Sigma_{i=1}^n \Sigma_{j=1}^{l_i} (m_{ij}-1)c_{ij} \big).\]
\end{lemma}

\begin{proof} Since the morphism $f$ is logarithmically \'etale, we have
\[ f^\star \omega^{\log}_{p_B} = \omega^{\log}_{p_C} \]
where:
\[ \omega^{\log}_{p_B} = \omega_{p_B}(\Sigma_{i=1}^n b_i), \qquad \omega^{\log}_{p_C} = \omega_{p_C}(\Sigma_{i=1}^n \Sigma_{j=1}^{l_i} c_{ij}). \]
On the other hand, for $i \in [n]$ we have
\[ f^\star \OO_B(b_i) = \OO_C(\Sigma_{j=1}^{l_i} m_{ij} c_{ij})\]
and the claim follows.	
\end{proof}

\begin{proposition} \label{prop: dualising sheaf comparison} We have:
\[ p_{C\star} \Td^\vee(\omega_{p_C}) = d \cdot p_{B\star} \Td^\vee (\omega_{p_B}) - \dfrac{1}{2} \sum_{i=1}^n (d-l_i) + \sum_{i=1}^n \left( \sum_{k \geqslant 1} \dfrac{B_{2k}}{(2k)!} \sum_{j=1}^{l_i} \dfrac{m_{ij}^{2k}-1}{m_{ij}^{2k-1}} \right) t^\star \psi_i^{2k-1}.\]	
\end{proposition}

\begin{proof}
We first observe the algebraic identity:
\begin{align}
\nonumber	\Td(x+y) & = \sum_{k \geqslant 0} \dfrac{B_k}{k!} \sum_{r=0}^k {k \choose r} x^r y^{k-r} \\
\label{eqn: Todd of sum}	& = \Td(x) + \sum_{k \geqslant 0} \dfrac{B_k}{k!} \sum_{r=0}^{k-1} {k \choose r} x^r y^{k-r}.
\end{align}
We apply this identity to the divisors:
\[ x = c_1(\omega_{p_C}^\vee), \qquad y = \sum_{i=1}^n \sum_{j=1}^{l_i} (m_{ij}-1) c_{ij}.\]
Observe that since the $c_{ij}$ are pairwise disjoint, we have
\[ y^k = \sum_{i=1}^n \sum_{j=1}^{l_i} (m_{ij}-1)^k c_{ij}^k \]
while on the other hand we have
\[ c_1(\omega_{p_C}^\vee) \cdot c_{ij} = c_{ij}^2 \]
since $\omega_{p_C}^\vee = T_{p_C}$ in a neighbourhood of $c_{ij}$ and thus serves as the normal bundle to the marking section. We then compute using \eqref{eqn: Todd of sum}:
\begin{align*}
f^\star \Td^\vee(\omega_{p_B}) & = \Td^\vee(\omega_{p_C}) + \sum_{k \geqslant 0} \dfrac{B_k}{k!} \sum_{r=0}^{k-1} {k \choose r} c_1(\omega_{p_C}^\vee)^r \sum_{i=1}^n \sum_{j=1}^{l_i} (m_{ij}-1)^{k-r} (c_{ij})^{k-r} \\
& = \Td^\vee(\omega_{p_C}) + \sum_{k \geqslant 0} \dfrac{B_k}{k!} \sum_{r=0}^{k-1} {k \choose r} \sum_{i=1}^n \sum_{j=1}^{l_i} (m_{ij}-1)^{k-r} (c_{ij})^k \\
& = \Td^\vee(\omega_{p_C}) + \sum_{k \geqslant 0} \dfrac{B_k}{k!} \sum_{i=1}^n \sum_{j=1}^{l_i} (c_{ij})^k \sum_{r=0}^{k-1} {k \choose r} (m_{ij}-1)^{k-r} \\
& = \Td^\vee(\omega_{p_C}) + \sum_{k \geqslant 0} \dfrac{B_k}{k!} \sum_{i=1}^n \sum_{j=1}^{l_i} (m_{ij}^k-1) (c_{ij})^k.
\end{align*}
We now apply $p_{C\star}$. On the left-hand side we have by the projection formula
\[ p_{C\star} f^\star \Td^\vee(\omega_{p_B}) = d \cdot p_{B\star} \Td^\vee (\omega_{p_B}).	\]
In the second term on the right-hand side, we have:
\[ p_{C\star} (c_{ij})^k = (-1)^{k-1} \psi_{ij}^{k-1} \]
where $\psi_{ij}$ is the psi class corresponding to the ramification point $c_{ij} \in C$, and we interpret $\psi_{ij}^{-1}$ as zero. Moreover we have
\[ \psi_{ij} = t^\star \psi_i/m_{ij} \]
where $t^\star \psi_i$ is the psi class corresponding to the branch point $b_i \in B$. Combining, we obtain:
\begin{align*} d \cdot p_{B\star} \Td^\vee (\omega_{p_B}) & = p_{C\star} \Td^\vee(\omega_{p_C}) + \sum_{k \geqslant 1} \dfrac{(-1)^{k-1} B_k}{k!} \sum_{i=1}^n \bigg( \sum_{j=1}^{l_i} \dfrac{m_{ij}^k-1}{m_{ij}^{k-1}} \bigg) t^\star \psi_i^{k-1} \\
& = p_{C\star} \Td^\vee(\omega_{p_C}) + \dfrac{1}{2} \sum_{i=1}^n (d-l_i) - \sum_{k \geqslant 1} \dfrac{B_{2k}}{(2k)!} \sum_{i=1}^n \bigg( \sum_{j=1}^{l_i} \dfrac{m_{ij}^{2k}-1}{m_{ij}^{2k-1}} \bigg) t^\star \psi_i^{2k-1} 
\end{align*}
which rearranges to give the desired formula.
\end{proof}

\subsection{Grothendieck--Riemann--Roch} Having established the comparison of Todd classes (\zcref{prop: residue sheaf comparison}) and dualising sheaves (\zcref{prop: dualising sheaf comparison}), we are now in a position to prove the main result.

\begin{proof}[Proof of \zcref{thm: comparison}] We begin with $\EE_C$. We have:
\begin{alignat*}{3}
 b_0(C) - \ch^\vee(\EE_C) & = \ch(\Rder p_{C\star} \OO_C) \qquad \qquad \qquad  && \text{by \eqref{eqn: derived pushforward classes}} \\
 & = p_{C\star} \Td(T_{p_C}) && \text{by GRR} \\
 & = p_{C\star} \Td^\vee (\omega_{p_C}) + p_{C\star} \Td^\vee(R_{p_C})_{\geqslant 1}^{-1} \qquad \qquad && \text{by \zcref{cor: product of canonical with residue}}
\end{alignat*}
Plugging in \zcref{prop: dualising sheaf comparison} and \zcref{prop: residue sheaf comparison}, the main terms collect to give
\[ d \left( p_{B\star} \Td^\vee(\omega_{p_B}) + p_{B\star} \Td^\vee(R_{p_B})^{-1}_{\geqslant 1} \right) = d - d \cdot \ch^\vee(\EE_B)\]
again by GRR and \eqref{eqn: derived pushforward classes}. Rearranging and incorporating the correction terms gives the result.
\end{proof}

\begin{remark} \label{remark:comparison to existing formulae} Cousins of \zcref{thm: comparison} have appeared in various guises and levels of generality, see e.g. \cite[Proposition~3.2]{CEFS}, \cite[Theorem~1.1]{vdGK}, \cite[Chapitre 10]{BertinRomagny}, \cite{ChiodoGCovers}.

The most direct comparison is for cyclic covers. When $d=2$, \zcref{thm: comparison} can be deduced from Chiodo's formula \cite[Theorem~1.1.1]{ChiodoTowards}. Indeed we have the splitting (see e.g. \cite[Remark~4.1.7]{Lazarsfeld})
\[ f_\star \OO_C = \OO_{B} \oplus L \]
where $L$ is the pushforward of the universal square root $\Lcal$ along the coarse moduli map $\Bcal \to B$. The difference between $\EE_C$ and $\EE_B$ is thus controlled by the Chern character of the pushforward of $\Lcal$ to the moduli space, which is precisely what Chiodo's formula calculates. However, there is a difference in formalisms:
\begin{itemize}
    \item \underline{\smash{Our formula:}} $\ch(\EE_C) = 2 \ch(\EE_B) + \text{(corrections)}.$
    \item \underline{\smash{Chiodo's formula:}} $\ch(\EE_C) = \ch(\EE_B) + \text{(corrections)}.$
\end{itemize}
To translate between Chiodo's formula and ours thus requires identifying another copy of $\ch(\EE_B)$ sitting amongst his corrections. This can be achieved using Mumford's formula \cite{MumfordTowards}.
 
When $d \geqslant 3$ and the group $G \leqslant S_d$ appearing in the global type $\Theta$ is cyclic of order $d$, the above splitting generalises to:
\[ f_\star \OO_C = \OO_B \oplus L \oplus \cdots \oplus L^{d-1}.\]
It should be possible to generalise Chiodo's formula to $L^k$ and then assemble the $d-1$ individual formulae to obtain \zcref{thm: comparison}. However the algebra appears to be rather complicated.

Moving to non-cyclic $G$, there are Grothendieck--Riemann--Roch calculations involving the universal vector bundles on $\Bcal G$ associated to irreducible $G$-representations \cite{ZhouHurwitz,Galeotti,JPT}, but it is unclear how these relate to the Prym Hodge bundle.

We conclude with a curious observation. Our proof seems to bypass a complication arising in previous approaches: the need to pass to a crepant resolution of the universal coarse curve. This simplification is made possible by leveraging an invariance property of the residue series (see the proof of \zcref{prop: Todd is res}), discovered using the language of piecewise polynomials.
\end{remark}

\subsection{Special case: double covers} \label{sec: comparison degree 2}

Consider $d=2$ with global type $\Theta$ given by taking $G=S_2$ and $n$ branch points, i.e. $g_i = (12)$ for all $i \in [n]$ (necessarily $n$ is even). Here having global type $\Theta$ is equivalent to the cover being connected, and we adopt the more conventional notation:
\[ \Rcalbar_{g,n} \colonequals \Mcalbar_{g,n}^{[S_2,(12)^n]}(\Bcal S_2). \]
\zcref{thm: Etilde in terms of EB} then specialises to:
\begin{theorem} The following identity holds in $\Rcalbar_{g,n}$:
\begin{align*}
	\ch^\vee(\Etilde) = \ & t^\star \ch^\vee(\EE) - 1 + n/2 \ + \\
	& - \sum_{i=1}^n \bigg( \sum_{k \geqslant 1} \dfrac{B_{2k}}{(2k)!} \cdot \dfrac{2^{2k}-1}{2^{2k-1}} \bigg) t^\star \psi_i^{2k-1} + \\	
	& \sum_{(\Gamma,\Theta_{V(\Gamma)})} (j_{(\Gamma,\Theta_{V(\Gamma)})})_\star \left( \sum_{k \geqslant 1} \bigg( \dfrac{B_{2k}}{(2k)!} \cdot  \dfrac{2^{2k}-1}{2^{2k-2}} \bigg) \cdot (t^\star \psi_{\vec{e}}^{2k-2} - \cdots + t^\star \psi_{\cev{e}}^{2k-2} ) \right).
\end{align*}
The final sum is over the loop and bridge loci studied in \zcref{sec: double covers}.
\end{theorem}

\subsection{Recap: piecewise polynomials, power series, and the Todd operator} \label{sec: piecewise} We recap the language of piecewise polynomial functions \cite{MPS,CaseStudy} used above in the proof of \zcref{thm: comparison}.

\subsubsection{Piecewise polynomials}
A \textbf{polynomial} on a cone $\sigma$ is an element of the symmetric algebra
\[ \Sym(S_\sigma), \]
where $S_\sigma$ is the dual monoid to $\sigma$. A \textbf{piecewise polynomial} on a cone complex $\Sigma$ is a polynomial on each cone, compatible along the face maps. If $X$ is a logarithmic stack with tropicalisation $\Sigma_X$, there is a well-defined map
\[ \Phi \colon \PP^\star(\Sigma_X) \to CH^\star(X) \]
sending each piecewise polynomial to its associated Chow class.

\subsubsection{Piecewise power series}
The algebra $\Sym(S_\upsigma)$ has a maximal ideal generated by non-unit monomials:
\[ \mathfrak{m}_\upsigma \colonequals ( t^m \colon m \in S_\upsigma \setminus \{ 0 \} ).\]
We denote the completion of the symmetric algebra with respect to this maximal ideal by
\[ \hat{\Sym}(S_\upsigma) \]
and refer to its elements as \textbf{power series} on $\upsigma$. A \textbf{piecewise power series} on a cone complex $\Sigma$ is a power series on each cone, compatible along the face maps. If $X$ is a logarithmic stack with nilpotent Chow (e.g. if $X$ is a variety, or a finite type Artin stack) then there is a well-defined map
\[ \Phi \colon \widehat{\PP}^\star(\Sigma_X) \to CH^\star(X) \]
sending each piecewise power series to the associated Chow class.

\subsubsection{Todd operator}
Given a cone $\upsigma$ the \textbf{Todd operator} sends a linear function to a power series, and is defined as follows
\begin{align*} \Td \colon S_\upsigma & \to \hat{\Sym}(S_\upsigma) \\
x & \mapsto 1 + \dfrac{x}{2} + \sum_{k \geqslant 1} \dfrac{B_{2k}}{(2k)!} x^{2k}\end{align*}
where $B_{2k}$ is the Bernoulli number. The right-hand side is the power series expansion of $x/(1-e^{-x})$ about $x=0$. The Todd operator extends to cone complexes, producing:
\[ \Td \colon \PL(\Sigma) \to \widehat{\PP}^\star(\Sigma). \]
The Todd operator is so-named due to the commutativity of the following diagram:
\[
\begin{tikzcd}	
\PL(\Sigma_X) \ar[r,"\Phi"] \ar[d,"\Td"] & CH^1(X) \ar[d,"\Td"] \\
\widehat{\PP}^\star(\Sigma_X) \ar[r,"\Phi"] & CH^\star(X).	
\end{tikzcd}
\]

\section{Algorithm} \label{sec: algorithm}

\noindent Returning to \zcref{problem: second formulation}, we wish to compute integrals of the form:
\[ \int_{\Mcalbar_{g,n}^{\Theta}(\Bcal S_d)} \chitilde_{i_1} \cdots \chitilde_{i_k}.\]
A naive approach would be to apply \zcref{thm: Etilde in terms of EB} separately to each $\tilde{\chi}_i$ and then expand out the resulting expression. While possible in principle, the correction terms produce a mire of recursive boundary combinatorics that is extremely difficult to carry out, even for computers, and even in the classical case $d=2,n=0$.

To overcome this difficulty, a final ingredient is needed: a recursive property of the Prym Hodge bundle. The analogous property plays a central role in Faber's original algorithm.

\subsection{Recursive property of the Prym Hodge bundle} We begin by recalling the recursive property of the Hodge bundle on the moduli space of curves. Fix a stable graph $\Gamma$ and consider the associated gluing map
\[ \gl_\Gamma \colon \prod_{v \in V(\Gamma)} \Mcalbar_{g_v,H_v(\Gamma)} \to \Mcalbar_{g,n} \]
with image $\Mcalbar_{\Gamma}$. Moreover for $v \in V(\Gamma)$ let $\pi_v$ denote the projection:
\[ \pi_v \colon \prod_{v^\prime \in V(\Gamma)} \Mcalbar_{g_{v^\prime},H_{v^\prime}(\Gamma)} \to \Mcalbar_{g_v,H_v(\Gamma)}.\]

\begin{lemma} \label{lem: recursive lambda class on Mbar} There is a short exact sequence
\[ 0 \to \bigoplus_{v \in V(\Gamma)} \pi_v^\star \EE \to \gl_{\Gamma}^\star \EE \to \OO^{\oplus b_1(\Gamma)} \to 0.\]
In particular, we have:
\[ \gl_{\Gamma}^\star \ch(\EE)_{\geqslant 1}  = \sum_{v \in V(\Gamma)}\pi_v^\star \ch(\EE)_{\geqslant 1}.\]
\end{lemma}

\begin{proof} Consider the normalisation sequence and pass to cohomology.
\end{proof}

The Prym Hodge bundle satisfies a similar recursive property, which we state at the level of Chern characters. Following the boundary formalism for $S_d$-covers (\zcref{sec: vertexwise strata}) fix a global type $\Theta$ and a vertexwise $S_d$-graph $(\Gamma,\Theta_{V(\Gamma)})$. Recall (\zcref{prop: splitting}) that there is an associated splitting correspondence \eqref{eqn: splitting correspondence}:
\begin{equation} \label{eqn: splitting correspondence revisited}
\begin{tikzcd}[column sep=tiny]
& \Mcaltilde_{(\Gamma,\Theta_{V(\Gamma)})}^{\Theta}(\Bcal S_d) \ar[rd,"f"] \ar[ld,"g" above] & \\
\prod_{v \in V(\Gamma)} \Mcalbar_{g_v,H_v(\Gamma)}^{\Theta_v}(\Bcal S_d) && \Mcalbar_{(\Gamma,\Theta_{V(\Gamma)})}^{\Theta}(\Bcal S_d) 
\end{tikzcd}
\end{equation}

\begin{proposition} \label{prop: recursive structure of Prym Hodge} The Prym Chern characters pull back compatibly along the correspondence \eqref{eqn: splitting correspondence revisited}:
\[ f^\star \ch(\Etilde)_{\geqslant 1} = \sum_{v \in V(\Gamma)} g^\star \pi_v^\star \ch(\Etilde).\]
In particular, given a monomial $\tilde{\chi}_{i_1} \cdots \tilde{\chi}_{i_k}$ in the Prym chi classes, we have
\[ \int_{\Mcalbar_{(\Gamma,\Theta_{V(\Gamma)})}^{\Theta}(\Bcal S_d)} \tilde{\chi}_{i_1} \cdots \tilde{\chi}_{i_k} = \bigg( \dfrac{1}{D^{\Theta}_{(\Gamma,\Theta_{V(\Gamma)})}} \bigg) \cdot \int_{\prod_{v \in V(\Gamma)} \Mcalbar_{g_v,H_v(\Gamma)}^{\Theta_v}(\Bcal S_d)} \prod_{j=1}^k \bigg( \sum_{v \in V(\Gamma)} \pi_v^\star \chitilde_{i_j} \bigg) \]
where $D^{\Theta}_{(\Gamma,\Theta_{V(\Gamma)})}$ is the degree given in \zcref{prop: splitting}.
\end{proposition}

\begin{proof} There are three proofs. The first is to use the description of $\Prym(f)$ as a semiabelian variety (\zcref{prop: Prym semiabelian}), and note that the toric part is constant over the strata. The second is to write $\Etilde = \EE_C - \EE_B$ and appeal to the decompositions of $\EE_C$ and $\EE_B$ (\zcref{lem: recursive lambda class on Mbar}). The third (and hardest) is to use the formula in \zcref{thm: Etilde in terms of EB} directly, and prove that it pulls back accordingly (we have carried out this last strategy only for $d=2$).
\end{proof}

\subsection{The algorithm} We now use the recursive property of the Prym Hodge bundle to formulate our algorithm. Rather than only integrating monomials of $\tilde{\chi}_i$, we broaden our scope to integrals of the following form:
\begin{equation} \label{eqn: algorithm integral} \int_{\Mcalbar_{g,n}^{\Theta}(\Bcal S_d)} (\chitilde_{i_1} \cdots \chitilde_{i_k}) \cdot t^\star(\chi_{j_1} \cdots \chi_{j_l} \cdot \psi_1^{c_1} \cdots \psi_n^{c_n} ).\end{equation}
where as always $t$ denotes the forgetful morphism to the moduli space of curves
\[ t \colon \Mcalbar_{g,n}^{\Theta}(\Bcal S_d) \to \Mcalbar_{g,n}. \]
The following recursive algorithm computes these integrals.

\begin{algorithm}[Algorithm~\ref{algorithm introduction}] \label{algorithm main} Consider an integral of the form \eqref{eqn: algorithm integral}:
\begin{itemize}
\item \textbf{Base case.} This occurs when there are no Prym chi classes, i.e. $k=0$. In this case the entire integrand is pulled back along $t$, and we can apply the projection formula. The degree of $t$ is given by \zcref{lem: degree target map global} and \zcref{prop: charformula}. The resulting integral on $\Mcalbar_{g,n}$ is then computed using Faber's algorithm.
\item \textbf{Induction step.} We reduce the number of Prym chi classes by applying \zcref{thm: Etilde in terms of EB} to $\tilde{\chi}_{i_1}$. The correction terms consist of $t^\star \chi_i$, $t^\star \psi_i$ and boundary classes. The former two are already included in the form \eqref{eqn: algorithm integral}. For the boundary classes, we use \zcref{prop: recursive structure of Prym Hodge} to reduce them to other integrals of the form \eqref{eqn: algorithm integral}. In all cases, the number of Prym chi classes reduces by one. \qed
\end{itemize}
\end{algorithm}

\subsection{Implementation} The algorithm is computer-implemented for $d=2$ and $n$ arbitrary in accompanying \texttt{Sage} code. The base case of the induction relies on Faber's algorithm as implemented in \texttt{admcycles}. Tables of the tautological projections thus computed are collected in \zcref{appendix}. These include verifications of several existing results: see \zcref{sec: degree of Prym}.

We have not computer-implemented the algorithm for higher $d$. However, example hand-calculations for $d=3$ are given in \zcref{sec: degree 3 covers non-cyclic} (with $G=S_3 \leqslant S_3$) and \zcref{sec: degree 3 covers cyclic} (with $G=A_3 \leqslant S_3$).

\section{Applications} \label{sec: applications}

\subsection{Degree of the Prym map} \label{sec: degree of Prym} A case of particular interest occurs when the Prym class has codimension zero, i.e:
\begin{equation} \label{eqn: dim admissible covers equals dim moduli abelian vars} \dim \Mcalbar_{g,n}^{\Theta}(\Bcal S_d) = \dim \Acalbar^{\Sigma}_{\tilde{g},\delta}. \end{equation}
Since we clearly have
\[ R^0(\Acalbar_{\tilde{g},\delta}) = \CH^0(\Acalbar_{\tilde{g},\delta}^\Sigma) \cong \QQ \]
it follows that:
\[ \taut \Prym_\star [\Mcalbar_{g,n}^{\Theta}(\Bcal S_d)^\Sigma] = \Prym_\star [\Mcalbar_{g,n}^{\Theta}(\Bcal S_d)^\Sigma] = \deg(\Prym) \cdot [\Acalbar_{\tilde{g},\delta}^\Sigma]. \]
Our algorithm thus computes the degree of the Prym map in this setting. By implementing our algorithm, we obtain uniform proofs of a suite of results previously proved via ad hoc methods, variously by Donagi--Smith \cite{DonagiSmith_StructurePrym},  Nagaraj--Ramanan and Bardelli--Ciliberto--Verra \cite{NagarajRamanan,BardelliCilibertoVerra}, Marcucci--Naranjo \cite{MarcucciNaranjo}, and Lange--Ortega \cite{LangeOrtegaTriple}.

\subsubsection{Enumerating the cases} We begin with a classification of the relevant cases, restricting as always to the stable regime $2g-2+n > 0$ and assuming the cover is connected. Given a cover $f \colon C \to B$ of degree $d$, let 
\[ R = \sum_{p \in C}(m_p-1)p \]
denote the ramification divisor on $C$. The quantity $\deg R$ depends only on the global type $\Theta$.

\begin{proposition} \label{prop: codim zero cases} The Prym cycle has codimension zero in precisely the following cases:
\begin{itemize}
    \item $d=2$ cases:
    \begin{enumerate}
        \item[(A)] $g=6,n=0$.
        \item[(B)] $g=3,n=4$, $\deg R = 4$.
    \end{enumerate}
    \item $d=3$ case:
    \begin{enumerate}
        \item[(C)] $g=2,n=0$.
    \end{enumerate}
    \item $g=1$ cases:
    \begin{enumerate}
        \item[(D)] $d \geqslant 3$, $n=1$, $\deg R =2$.
        \item[(E)] $d \geqslant 3$, $n=3$, $\deg R = 4$.
        \item[(F)] $d \geqslant 2$, $n=6$, $\deg R = 6$.
    \end{enumerate}
    \item $g=0$ cases:
    \begin{enumerate}
        \item[(G)] $d \geqslant 3$, $n=3$, $\deg R = 2d-4$.
        \item[(H)] $d \geqslant 3$, $n=3$, $\deg R = 2d-2$.
        \item[(I)] $d \geqslant 2$, $n=4$, $\deg R = 2d$.
        \item[(J)] $d \geqslant 2$, $n=6$, $\deg R = 2d+2$.
        \item[(K)] $d \geqslant 3$, $n=9$, $\deg R = 2d+4$.    
    \end{enumerate}
\end{itemize}
\end{proposition}
In Sections~\ref{sec: codim zero degree 2}-\ref{sec: degree 3 covers cyclic} below we calculate the degree of the Prym map in all cases where either $d=2$ or $g \geqslant 2$.

\begin{proof} Riemann--Hurwitz gives the dimension of the Prym variety as
\begin{equation}
\tilde{g} = (d-1)(g-1) + \deg R/2
\end{equation}
and the identity \eqref{eqn: dim admissible covers equals dim moduli abelian vars} reads:
\begin{equation} \label{eqn: Diophantine identity codim zero}
    6g - 6 + 2n = \tilde{g}(\tilde{g}+1).
\end{equation}

We begin with the $d=2$ case. We assume without loss of generality that every marking is ramified (otherwise the Prym class vanishes), so that $\deg R = n$ and the equation \eqref{eqn: Diophantine identity codim zero} becomes:
\begin{equation} 4g^2 + (4n-28)g + (n^2-10n+24) = 0. \label{eqn: codim zero cases Diophantine} \end{equation}
We view this as a polynomial in $g$ and seek solutions for $g \in \Z_{\geqslant 0}$. This means that the discriminant $\Delta$ must be a perfect square. We have:
\begin{align*} \Delta & = (4n-28)^2 - 16(n^2-10n+24) \\
& = 16 \left ( (n-7)^2 - n^2 + 10n - 24 \right) \\
& = 16 (-4n+25).
\end{align*}
We therefore need $-4n+25=m^2$ for an integer $m$, giving:
\[ 4n = 25 - m^2. \]
Since we also need $n \in \Z_{\geqslant 0}$ there are only finitely many options, and we find that the valid choices are $m = 5,3,1$ with corresponding values $n = 0,4,6$. We now investigate these in turn.

When $n=0$ the equation \eqref{eqn: codim zero cases Diophantine} becomes:
\begin{align*} & 4g^2 - 28g + 24 = 0 \\
\Leftrightarrow \ & (g-1)(g-6) = 0. \end{align*}
The $g=1$ solution is invalid since $(g,n)=(1,0)$ is unstable, however the $g=6$ solution is valid, giving $(g,n)=(6,0)$. This is case (A) in the statement of the proposition. When $n=4$ the equation \eqref{eqn: codim zero cases Diophantine} becomes:
\begin{align*} & 4g^2 - 12g = 0 \\
\Leftrightarrow \ & g(g-3) = 0. \end{align*}
This gives the solutions $(g,n)=(3,4),(0,4)$. These are respectively case (B) and the $d=2$ part of case (I) in the statement of the proposition. Finally when $n=6$ the equation \eqref{eqn: codim zero cases Diophantine} becomes:
\begin{align*}
    & 4g^2-4g = 0 \\
    \Leftrightarrow \ & 4g(g-1) = 0
\end{align*}
This gives the solutions $(g,n)=(1,6),(0,6)$. These are the $d=2$ parts of cases (F) and (J) respectively in the statement of the proposition. This completes the $d=2$ classification. 

Turning to $d \geqslant 3$, note that $\deg R \geqslant n$ with equality if and only if the cover is simply ramified everywhere. Since $d \geqslant 3$ we have $\tilde{g} \geqslant 2g-2+n/2$, and so for \eqref{eqn: Diophantine identity codim zero} to hold we must have:
\begin{align} \nonumber & 3g-3+n \geqslant \dim \Acal_{2g-2+n/2,\delta} = (2g-1+n/2)(2g-2+n/2)/2 \\
\Leftrightarrow \ & n^2 + (8g-14)n + (16g^2-48g+32) \leqslant 0. \label{eqn: Diophantine inequality degree greater equal 3}
\end{align}
Fixing $g$ and viewing this as a quadratic inequality in $n$, we see that for the inequality to have solutions the quadratic must have real roots, and therefore the discriminant $\Delta$ must be non-negative. We have
\[ \Delta = (8g-14)^2 - 4(16g^2-48g+32) = 4(17-8g) \]
which is non-negative if and only if $g = 0,1,2$. Starting with the $g=2$ case, the inequality \eqref{eqn: Diophantine inequality degree greater equal 3} gives
\[ n(n+2)  \leqslant 0 \]
so that the only valid solution is $n=0$. We then have $\tilde{g} = d-1$ and the equation \eqref{eqn: Diophantine identity codim zero} becomes
\[ d(d-1) = 6 \]
the only solution of which is $d=3$. This is precisely case (C) in the statement of the proposition.

We turn to the $g=1$ case. The inequality \eqref{eqn: Diophantine inequality degree greater equal 3} gives
\[ n(n - 6) \leqslant 0 \]
which gives the choices $n=0,1,\ldots,6$. A case analysis then shows that the choices satisfying \eqref{eqn: Diophantine identity codim zero} are precisely $n=1$ with $\deg R = 2$, $n=3$ with $\deg R = 4$, and $n=6$ with $\deg R = 6$. These are precisely the cases (D), (E) and (F) in the statement of the proposition.

Finally we consider the $g=0$ case. The inequality \eqref{eqn: Diophantine inequality degree greater equal 3} gives
\[ n^2-14n+32 \leqslant 0 \]
which holds if and only if $n=3,4,\ldots,10,11$. Since $g=0$ we have $\tilde{g} = 1-d+\deg R/2$, and the identity \eqref{eqn: Diophantine identity codim zero} reads:
\[ 8n - 24 = (\deg R - (2d-2))(\deg R - (2d-4)).\]
It follows that $8n-24$ must be a product of two consecutive even integers. Examining the cases $n=3,4,\ldots,10,11$ we obtain precisely the $d=3$ parts of cases (G)--(K) in the statement of the proposition.
\end{proof}

\begin{remark} The degree of the Prym map has been investigated also in situations where
\[ \dim \Mcalbar^{\Theta}_{g,n}(\Bcal S_d) < \dim \Acalbar^{\Sigma}_{\tilde{g},\delta}. \]
In these situations the Prym map is only finite \emph{onto its image}, typically a subvariety of $\Acal_{\tilde{g},\delta}$ parametrising abelian varieties with additional automorphisms (see \zcref{sec: degree 3 covers cyclic} below). 

To compute these degrees using our methods would require a better understanding of the image, specifically its toroidal compactifications and tautological ring. An alternative approach would be to compute the tautological projection of the image and compare coefficients, see \zcref{sec: future directions additional auts introduction}.
\end{remark}

\subsubsection{Calculations I: degree~$2$} \label{sec: codim zero degree 2} For an overview of the $d=2$ case see \cite{NaranjoOrtega}. We again adopt the more conventional notation
\[ \Rcalbar_{g,n} \colonequals \Mcalbar_{g,n}^{[S_2,(12)^n]}(\Bcal S_2).\]
Propositions~\ref{thm: 27}-\ref{thm: genus 0 with 6 branch points} below are obtained using the accompanying \texttt{Sage} code (see \zcref{appendix} for technical details and further calculations).

\begin{proposition}[Case A] \label{thm: 27} We have
\[ \Prym_\star [\Rcalbar_{6,0}^\Sigma] = 27 \cdot [\Acalbar_5^\Sigma] \]
so that the degree of the Prym map $\Rcal_{6,0} \to \Acal_5$ is $27$. This gives a new proof of the Donagi--Smith theorem \cite[Theorem~2.1]{DonagiSmith_StructurePrym}.
\end{proposition}

\begin{proposition}[Case B] \label{thm: genus 3 with 4 branch points} We have
\[ \Prym_\star [\Rcalbar_{3,4}^\Sigma] = 72 \cdot [\Acalbar_{4,\delta}^\Sigma] \]
so that the degree of the Prym map $\Rcal_{3,4} \to \Acal_{4,\delta}$ is $72$. This gives a new proof of \cite[Theorem~5.11]{BardelliCilibertoVerra} and \cite[Theorem~9.14]{NagarajRamanan}. The discrepancy between our degree ($72$) and theirs ($3$) is $72/3=24=4!$ which accounts for the fact that we label the branch points whereas they do not.
\end{proposition}

\begin{proposition}[Case F, $d=2$] \label{thm: genus 1 with 6 branch points} We have
\[ \Prym_\star [\Rcalbar_{1,6}^\Sigma] = 720 \cdot [\Acalbar_{3,\delta}^\Sigma] \]
so that the degree of the Prym map $\Rcal_{1,6} \to \Acal_{3,\delta}$ is $720$. This gives a new proof of \cite[Theorem~1.1]{MarcucciNaranjo}. The discrepancy between our degree ($720$) and theirs ($1$) is $720/1 = 720 = 6!$ which accounts for the fact that we label the branch points whereas they do not.
\end{proposition}

The remaining $d=2$ cases (I) and (J) both have $g=0$. In this setting we have 
\[ \Prym(f) = \Jac(C) \]
so that the Prym map is equal (up to a combinatorial multiplicity) to the map sending a hyperelliptic curve to its Jacobian. Our algorithm thus computes the tautological projection of the hyperelliptic Torelli class, see Appendix~\ref{sec: tables hyperelliptic taut}.

Focusing specifically on cases (I) and (J) we have $\tilde{g}=1$ and $\tilde{g}=2$ respectively, and so in fact the hyperelliptic Torelli class is equal to the Torelli class. The following results therefore provide an independent verification of the algorithm:

\begin{proposition}[Case I, $d=2$] \label{thm: genus 0 with 4 branch points} We have
\[ \Prym_\star [\Rcalbar_{0,4}^\Sigma] = 6 \cdot [\Acalbar_1^\Sigma].\]
This also follows from the fact that the Torelli map $\Mcal_{1,1} \to \Acal_1$ has degree~$1$: choosing a marking, the map $\Rcal_{0,4} \to \Mcal_{1,1}$ has degree $3!=6$, corresponding to the labelling of the remaining markings.
\end{proposition}

\begin{proposition}[Case I, $d=2$] \label{thm: genus 0 with 6 branch points} We have
\[ \Prym_\star [\Rcalbar_{0,6}^\Sigma] = 720 \cdot [\Acalbar_2^\Sigma].\]
This also follows from the fact that the Torelli map $\Mcal_{2} \to \Acal_2$ has degree~$1$: the map $\Rcal_{0,6} \to \Mcal_{2}$ has degree $6!=720$, corresponding to the labelling of the markings.
\end{proposition}

\subsubsection{Calculations II: degree~$3$ non-cyclic covers} \label{sec: degree 3 covers non-cyclic} Having dealt with the $d=2$ cases, we now investigate the only other case which has $g \geqslant 2$, namely case (C) concerning \'etale triple covers of a genus two curve. This is not covered by our code, but has small enough dimension that it can be calculated by hand.

\begin{proposition}[Case C] \label{prop: deg 3 non-cyclic} Set $d=3,g=2,n=0$ and consider the global type $\Theta=[S_3]$, meaning that the monodromy representation must be surjective. We have
\[ \Prym_\star [\Mcalbar^{[S_3]}_{2,0}(\Bcal S_3)^\Sigma] = 10 \cdot [\Acalbar_2^\Sigma].\]
This gives a new proof of \cite[Theorem~5.1]{LangeOrtegaTriple} (in that reference, $S_3$-covers of type $(S_3)$ are referred to as ``non-cyclic'' triple covers).
\end{proposition}

There is an important reason for working with $S_3$ instead of the cyclic subgroup $A_3$, see \zcref{sec: degree 3 covers cyclic} below.

\begin{proof} The moduli spaces have dimension $3$ and the socle of the tautological ring is generated by $\lambda_1 \lambda_2$. It is thus equivalent to show:
\[ \int_{\Mcalbar^{[S_3]}_{2,0}(\Bcal S_3)} \lambdatilde_1 \lambdatilde_2 = 10 \int_{\Acalbar_2} \lambda_1 \lambda_2 = \dfrac{1}{576}.\]
Since our algorithm works better with chi classes, we note the following identity in $R^3(\Acalbar_2)$:
\[ \chi_3 = -\tfrac{1}{6} \lambda_1 \lambda_2 .\]
It is therefore equivalent to show:	
\begin{equation} \label{eqn: non-cyclic calculation goal} \int_{\Mcalbar^{[S_3]}_{2,0}(\Bcal S_3)} \chitilde_3 = -\dfrac{1}{6} \cdot \dfrac{1}{576} = \dfrac{-1}{3456}.\end{equation}

To prove this, we apply \zcref{thm: Etilde in terms of EB} to $\chitilde_3$. Notice that for any K-theory class $\EE$ we have
\[ \ch_{\mathrm{odd}}^\vee(\EE) = -\ch_{\mathrm{odd}}(\EE) \]
and so we must reverse the signs of all the correction terms. The relevant\footnote{Recall that boundary divisors with ramification profile $(1,\ldots,1)$ over the edge do not contribute. Moreover the fact that $[S_3,S_3]=A_3$ precludes the possibility of attaching a $2$-cycle to the bridge edge.} vertexwise boundary divisors are indexed by the following vertexwise $S_d$-graphs, where each vertex is labelled with both its genus $g_v$ and its associated subgroup $G_v$:
\[
\begin{tikzpicture}
% D1
\draw[fill=black] (0,0) circle[radius=3pt];
\draw (0,0) node[right]{$1$};
\draw (0,0) node[left]{$S_3$};
\draw (1,0) circle[radius=1];
\draw (2,0) node[right]{$(3)$};
\draw (1,-1) node[below]{$(\Gamma_1,\Theta_{V(\Gamma_1)})$};

% D2
\draw[fill=black] (4,0) circle[radius=3pt];
\draw (4,0) node[right]{$1$};
\draw (4,0) node[left]{$S_3$};
\draw (5,0) circle[radius=1];
\draw (6,0) node[right]{$(2,1)$};
\draw (5,-1) node[below]{$(\Gamma_2,\Theta_{V(\Gamma_2)})$};

% D3
\draw[fill=black] (8,0) circle[radius=3pt];
\draw (8,0.1) node[above]{$S_3$};
\draw (8,-0.15) node[below]{$1$};
\draw (8,0) -- (10,0);
\draw (9,0) node[above]{$(3)$};
\draw[fill=black] (10,0) circle[radius=3pt];
\draw (10,0.1) node[above]{$S_3$};
\draw (10,-0.15) node[below]{$1$};
\draw (9,-1) node[below]{$(\Gamma_3,\Theta_{V(\Gamma_3)})$};

\end{tikzpicture}
\]
See \zcref{example: char formula boundary} for a detailed discussion. \zcref{thm: Etilde in terms of EB} now gives:
\begin{equation} \label{eqn: non-cyclic calculation integral as sum of four terms} \int_{\Mcalbar^{[S_3]}_{2,0}(\Bcal S_3)} \chitilde_3 = 2 \int_{\Mcalbar^{[S_3]}_{2,0}(\Bcal S_3)} t^\star \chi_3 + \mathrm{Cont}_{(\Gamma_1,\Theta_{V(\Gamma_1)})} + \mathrm{Cont}_{(\Gamma_2,\Theta_{V(\Gamma_2)})} + \mathrm{Cont}_{(\Gamma_3,\Theta_{V(\Gamma_3)})}.\end{equation}
We compute the four terms on the right-hand side.\medskip

\noindent \textbf{\underline{\smash{Main term.}}} Beginning with the $t^\star \chi_3$ term, we wish to apply the projection formula to:
\[ t \colon \Mcalbar_{2,0}^{[S_3]}(\Bcal S_3) \to \Mcalbar_{2,0}. \]
By \zcref{example: char formula global} we have (applying the character formula in \zcref{prop: charformula}):
\[ \deg \big( t \colon \Mcalbar_{2,0}^{[S_3]}(\Bcal S_3) \to \Mcalbar_{2,0} \big) = 60. \]
This allows us to compute the first term in \eqref{eqn: non-cyclic calculation integral as sum of four terms}
\begin{equation} \label{eqn: non-cyclic main contribution} 2 \int_{\Mcalbar^{[S_3]}_{2,0}(\Bcal S_3)} t^\star \chi_3 = 2 \cdot 60 \cdot \int_{\Mcalbar_2} \chi_3 = \dfrac{-1}{288}\end{equation}
where the final integral is computed using the \texttt{admcycles} implementation of Faber's algorithm.\medskip

\noindent \textbf{\underline{\smash{$(\Gamma_1,\Theta_{V(\Gamma_1)})$ term.}}} This is given by \zcref{thm: Etilde in terms of EB} as:
\[ \dfrac{1}{81} \int_{\Mcalbar^{[S_3]}_{(\Gamma_1,\Theta_{V(\Gamma_1)})}(\Bcal S_3)} t^\star (\psi_{\vec{e}}^2 - \psi_{\vec{e}} \psi_{\cev{e}} + \psi_{\cev{e}}^2).\]
We will apply the projection formula to
\[ t \colon \Mcalbar^{[S_3]}_{(\Gamma_1,\Theta_{V(\Gamma_1)})}(\Bcal S_3) \to \Mcalbar_{\Gamma_1} \]
where we note that $\Mcalbar_{\Gamma_1}$ is simply the loop locus in $\Mcalbar_{2,0}$. In Case 2 of \zcref{example: char formula boundary} we computed $\deg t = 6$. It follows that:
\begin{align}
\nonumber \dfrac{1}{81} \int_{\Mcalbar^{[S_3]}_{(\Gamma_1,\Theta_{V(\Gamma_1)})}(\Bcal S_3)} t^\star (\psi_{\vec{e}}^2 - \psi_{\vec{e}} \psi_{\cev{e}} + \psi_{\cev{e}}^2) & = \dfrac{6}{81} \int_{\Mcalbar_{\Gamma_1}} (\psi_{\vec{e}}^2 - \psi_{\vec{e}} \psi_{\cev{e}} + \psi_{\cev{e}}^2) \\[0.2cm]
\nonumber & = \dfrac{6}{81 \cdot 2} \int_{\Mcalbar_{1,2}} (\psi_1^2 - \psi_1 \psi_2 + \psi_2^2) \\[0.2cm]
\nonumber & = \dfrac{6}{81 \cdot 2 \cdot 24} \\[0.2cm]
\label{eqn: non-cyclic Gamma1 contribution} & = \dfrac{1}{648}.	
\end{align}

\noindent \textbf{\underline{\smash{$(\Gamma_2,\Theta_{V(\Gamma_2)})$ term.}}} This is given by \zcref{thm: Etilde in terms of EB} as follows (the multiplicity differs from that of the previous term because of the different ramification profile):
\[ \dfrac{1}{192} \int_{\Mcalbar^{[S_3]}_{(\Gamma_2,\Theta_{V(\Gamma_2)})}(\Bcal S_3)} t^\star (\psi_{\vec{e}}^2 - \psi_{\vec{e}} \psi_{\cev{e}} + \psi_{\cev{e}}^2 ).\]
From Case 1 of \zcref{example: char formula boundary} we have
\[ \deg \big( t \colon \Mcalbar_{(\Gamma_2,\Theta_{V(\Gamma_2)})} \to \Mcalbar_{\Gamma_2} \big) = 16 \]
from which we obtain:
\begin{align}
\nonumber \dfrac{1}{192} \int_{\Mcalbar^{[S_3]}_{(\Gamma_2,\Theta_{V(\Gamma_2)})}(\Bcal S_3)} t^\star (\psi_{\vec{e}}^2 - \psi_{\vec{e}} \psi_{\cev{e}} + \psi_{\cev{e}}^2 )	& = \dfrac{16}{192} \int_{\Mcalbar_{\Gamma_2}} (\psi_{\vec{e}}^2 - \psi_{\vec{e}} \psi_{\cev{e}} + \psi_{\cev{e}}^2 ) \\[0.2cm]
\nonumber & = \dfrac{16}{192 \cdot 2} \int_{\Mcalbar_{1,2}} (\psi_1^2 - \psi_1\psi_2 + \psi_2^2) \\[0.2cm]
\nonumber & = \dfrac{16}{192 \cdot 2 \cdot 24} \\[0.2cm]
\label{eqn: non-cyclic Gamma2 contribution} & = \dfrac{1}{576}.
\end{align}

\noindent \textbf{\underline{\smash{$(\Gamma_3,\Theta_{V(\Gamma_3)})$ term.}}} This is given by \zcref{thm: Etilde in terms of EB} as:
\[ \dfrac{1}{81} \int_{\Mcalbar^{[S_3]}_{(\Gamma_3,\Theta_{V(\Gamma_3)})}} t^\star (\psi_{\vec{e}}^2 - \psi_{\vec{e}} \psi_{\cev{e}} + \psi_{\cev{e}}^2).\]
In this case we calculate the degree of the forgetful map by hand. We must enumerate monodromy representations of the following form:
\begin{itemize}
	\item \underline{\smash{Generators:}} $\alpha_1,\beta_1,\gamma_1,\alpha_2,\beta_2,\gamma_2 \in S_3$.
	\item \underline{\smash{Shapes:}} $\gamma_1,\gamma_2$ are $3$-cycles.
	\item \underline{\smash{Relations:}}
	\begin{enumerate}[(i)]
		\item $[\alpha_1,\beta_1] \gamma_1 = 1$.
		\item $[\alpha_2,\beta_2] \gamma_2 = 1$.
		\item $\gamma_1 = \gamma_2^{-1}$.
	\end{enumerate}
\end{itemize}
The Dijkgraaf--Witten formula (see e.g. \cite[Theorem~3]{ZagierFiniteGroups}) shows that there are $18$ choices of $\alpha_1,\beta_1,\gamma_1$ of the correct shape and satisfying (i), and the same for $\alpha_2,\beta_2,\gamma_2$ of the correct shape and satisfying (ii).

This shows that there are $18 \cdot 18$ choices overall, but we now need to impose the relation (iii). Precisely half the choices will satisfy this, and so we obtain $(18 \cdot 18)/2 = 162$ monodromy representations. Moreover we note that these are all surjective: we cannot have $\alpha_1,\beta_1 \in \langle \gamma_1 \rangle$ since then we would have $[\alpha_1,\beta_1]=1$ which violates (i). Therefore already $\alpha_1,\beta_1,\gamma_1$ generate all of $S_3$. We thus obtain:
\[ \big| \Hom^{\Theta}_{\Theta_{V(\Gamma_3)}}(\pi_1^{\orb}(\Gamma_3,r_3),S_3) \big| = 162. \]
We then have by \zcref{lem: degree target map vertexwise stratum}
\[ \deg t = \dfrac{1}{6 \cdot 3} \big| \Hom^{\Theta}_{\Theta_{V(\Gamma_3)}}(\pi_1^{\orb}(\Gamma_3,r_3),S_3) \big| = \dfrac{162}{6 \cdot 3} = 9 \]
from which we compute the contribution of $(\Gamma_3,\Theta_{V(\Gamma_3)})$:
\begin{align}
\nonumber \dfrac{1}{81} \int_{\Mcalbar^{[S_3]}_{(\Gamma_3,\Theta_{V(\Gamma_3)})}(\Bcal S_3)} t^\star (\psi_{\vec{e}}^2 - \psi_{\vec{e}} \psi_{\cev{e}} + \psi_{\cev{e}}^2) & = \dfrac{9}{81} \int_{\Mcalbar_{\Gamma_3}} (\psi_{\vec{e}}^2 - \psi_{\vec{e}} \psi_{\cev{e}} + \psi_{\cev{e}}^2) \\[0.2cm]
\nonumber & = \dfrac{9}{81\cdot 2} \int_{\Mcalbar_{1,1} \times \Mcalbar_{1,1}} (\psi_1^2 - \psi_1 \psi_2 + \psi_2^2) \\[0.2cm]
\nonumber & = \dfrac{-9}{81 \cdot 2} \left( \int_{\Mcalbar_{1,1}} \psi_1 \right)^2 \\[0.2cm]
\nonumber & = \dfrac{-9}{81 \cdot 2 \cdot 24 \cdot 24} \\[0.2cm]
\label{eqn: non-cyclic Gamma3 contribution} & = \dfrac{-1}{10368}.
\end{align}
Assembling the contributions \eqref{eqn: non-cyclic main contribution}, \eqref{eqn: non-cyclic Gamma1 contribution}, \eqref{eqn: non-cyclic Gamma2 contribution}, \eqref{eqn: non-cyclic Gamma3 contribution} and plugging back into \eqref{eqn: non-cyclic calculation integral as sum of four terms} gives
\[ \int_{\Mcalbar^{[S_3]}_{2,0}(\Bcal S_3)} \chitilde_3 = \dfrac{-1}{288} + \dfrac{1}{648} + \dfrac{1}{576} + \dfrac{-1}{10368} = \dfrac{-1}{3456} \]
which produces the desired formula \eqref{eqn: non-cyclic calculation goal}.
\end{proof}
	
\subsubsection{Calculations III: degree $3$ cyclic covers} \label{sec: degree 3 covers cyclic} In the previous example we had $d=3,g=2,n=0$ and $G=S_3$ for the global type. Fix now instead the monodromy group
\[ G = A_3 \leqslant S_3\]
and consider the global type $\Theta = [A_3]$. Just as in the previous example we have $\tilde{g}=2$ and
\[ \dim \Mcalbar^{[A_3]}_{2,0}(\Bcal S_3) = \dim \Acalbar^{\Sigma}_{2,\delta} \]
because the dimensions depend only on the degree and the ramification profiles, and not on the choice of monodromy group. However, in this case we claim that:
\begin{equation} \label{eqn: mu3 Prym pushforward is zero} \Prym_\star [\Mcalbar^{[A_3]}_{2,0}(\Bcal S_3)^\Sigma] = 0. \end{equation}
The reason is simple: given a cover with monodromy group $G \leqslant S_d$, the deck group of the cover is given (see e.g. \cite[Theorem~3]{Zoladek}) as the centraliser:
\[ C_{S_d}(G) \leqslant S_d. \]
In the previous section we had $G=S_3$ and so the deck group was $C_{S_3}(S_3) = 1$ (the covers were not Galois). However if $G=A_3$ then the deck group is $C_{S_3}(A_3) = A_3$. The cover is endowed with non-trivial automorphisms, and these induce non-trivial automorphisms of the Prym variety. The Prym map thus factors through a strict substack of the moduli space of abelian varieties, parametrising abelian varieties with larger automorphism group than the generic $\mu_2$. Given this, \eqref{eqn: mu3 Prym pushforward is zero} follows by dimensional considerations.

We now give an independent proof of the vanishing \eqref{eqn: mu3 Prym pushforward is zero}, by showing that the tautological projection is zero. This serves as one (of several) consistency checks for our algorithm.

\begin{proposition} Consider the above moduli space. Then:
	\[ \Prym_\star [\Mcalbar^{[A_3]}_{2,0}(\Bcal S_3)^\Sigma] = 0.\]
\end{proposition}

\begin{proof}
It suffices to show that
\[ \int_{\Mcalbar^{[A_3]}_{2,0}(\Bcal S_3)} \chitilde_3 = 0.\]
To prove this, we apply \zcref{thm: Etilde in terms of EB} to $\chitilde_3$. There is a single relevant vertexwise boundary divisor:
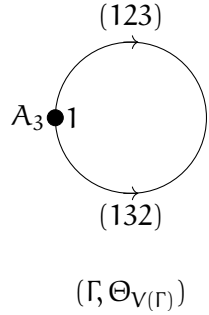
\[
\begin{tikzpicture}

\draw[fill=black] (0,0) circle[radius=3pt];
\draw (0,0) node[right]{$1$};
\draw (0,0) node[left]{$A_3$};
\draw (1,0) circle[radius=1];
\draw[-Stealth] (0.9,1) -- (1.1,1) [->];
\draw (1,1) node[above]{$(123)$};
\draw[-Stealth] (0.9,-1) -- (1.1,-1) [->];
\draw (1,-1) node[below]{$(132)$};
\draw (1,-2) node[below]{$(\Gamma,\Theta_{V(\Gamma)})$};
\end{tikzpicture}
\]
Here we represent the vertexwise type $\Theta_{V(\Gamma)}$ by a representative $\theta_{V(\Gamma)}$. Note that its conjugates exhaust all valid vertexwise image lists; the image list with $(g_{\vec{e}},g_{\cev{e}}) = ((123),(123))$ does not contribute because the associated stratum is empty, as can be seen by examining the monodromy representations. \zcref{thm: Etilde in terms of EB} then gives:
\begin{equation} \label{eqn: cyclic triple cover decomposition into three terms} \int_{\Mcalbar^{[A_3]}_{2,0}(\Bcal S_3)} \chitilde_3 = 2 \int_{\Mcalbar^{[A_3]}_{2,0}(\Bcal S_3)} t^\star \chi_3 + \mathrm{Cont}_{(\Gamma,\Theta_{V(\Gamma)})}.\end{equation}
We calculate the terms on the right-hand side.\medskip

\noindent \textbf{\underline{\smash{Main term.}}} Consider the map
\[ t \colon \Mcalbar^{[A_3]}_{2,0}(\Bcal S_3) \to \Mcalbar_{2,0}. \]
A monodromy representation is given by a free choice of $\alpha_1,\beta_1,\alpha_2,\beta_2 \in A_3$ (the product of commutators is always $1$ because $A_3$ is abelian). There are $3^4 = 81$ such representations. The trivial representation is the only one which is non-surjective, so we have:
\[ \big| \Hom^{[A_3]}(\pi_1^{\orb}(v),S_3) \big| = 80 \]
where $v$ is the stable graph consisting of a single vertex of genus~$2$ supporting no legs. It follows from \zcref{lem: degree target map global} that $\deg t = 80/6$ and so the main term becomes
\begin{equation} \label{eqn: cyclic triple cover main term} 2 \int_{\Mcalbar^{[A_3]}_{2,0}(\Bcal S_3)} t^\star \chi_3 = \dfrac{2 \cdot 80}{6} \int_{\Mcalbar_{2,0}} \chi_3 = \dfrac{-1}{1296} \end{equation}
where the final integral is calculated using the \texttt{admcycles} implementation of Faber's algorithm.\medskip

\noindent \textbf{\underline{\smash{$(\Gamma,\Theta_{V(\Gamma)})$ term.}}} The ramification profile $m_e$ associated to the edge is $(3)$, and so \zcref{thm: Etilde in terms of EB} gives this contribution as:
\[ \dfrac{1}{81} \int_{\Mcalbar^{[A_3]}_{(\Gamma,\Theta_{V(\Gamma)})}(\Bcal S_3)} t^\star (\psi_{\vec{e}}^2 - \psi_{\vec{e}} \psi_{\cev{e}} + \psi_{\cev{e}}^2 ).\]
For the monodromy representations, we refer to \zcref{def: group homs for VGamma}. We have two options for $\gamma_{\vec{e}}$ and once this choice is made, we have a single option for $\gamma_{\cev{e}}$. We then have a free choice of $\alpha_1,\beta_1,\delta_e \in A_3$. The representation is automatically surjective because $\gamma_{\vec{e}}$ and $\gamma_{\vec{e}}$ are sent to generators of $A_3$. We therefore have
\[  \big| \Hom^{[A_3]}_{(\Gamma,\Theta_{V(\Gamma)})}(\pi_1^{\orb}(\Gamma_1,r_1),S_3) \big| = 2 \cdot 3^3 = 54. \]
It follows from \zcref{lem: degree target map vertexwise stratum} that $\deg t = 54/(6\cdot 3) = 3$. We thus calculate the contribution:
\begin{align}
\nonumber	\dfrac{1}{81} \int_{\Mcalbar^{[A_3]}_{(\Gamma,\Theta_{V(\Gamma)})}(\Bcal S_3)} t^\star (\psi_{\vec{e}}^2 - \psi_{\vec{e}} \psi_{\cev{e}} + \psi_{\cev{e}}^2 ) & = \dfrac{3}{81} \int_{\Mcalbar_{\Gamma}} (\psi_{\vec{e}}^2 - \psi_{\vec{e}} \psi_{\cev{e}} + \psi_{\cev{e}}^2 ) \\[0.2cm]
\nonumber	& = \dfrac{3}{81 \cdot 2} \int_{\Mcalbar_{1,2}} (\psi_1^2 - \psi_1 \psi_2 + \psi_2^2) \\[0.2cm]
\nonumber	& = \dfrac{3}{81 \cdot 2 \cdot 24} \\[0.2cm]
\label{eqn: cyclic triple cover Gamma1 term} & = \dfrac{1}{1296}.
\end{align}
Plugging \eqref{eqn: cyclic triple cover main term} and \eqref{eqn: cyclic triple cover Gamma1 term} into \eqref{eqn: cyclic triple cover decomposition into three terms} gives
\[ \int_{\Mcalbar^{[A_3]}_{2,0}(\Bcal S_3)} \chitilde_3 = \dfrac{-1}{1296} + \dfrac{1}{1296} = 0 \]
as required.
\end{proof}

\subsection{Non-divisibility of the Jacobian class by the Prym class} \label{sec: divisibility} We now restrict to the classical case $(d=2,n=0)$ and equip our Prym varieties with the canonical principal polarisation, instead of the restriction of the theta divisor. By \zcref{rmk: choice of polarisation} this different choice of polarisation does not affect the formula for the tautological projection.

In \cite[Definition~4]{CMOP_tautologicalProjection} the authors construct a tautological projection for the non-compact moduli space $\Acal_g$, and show that the following diagram commutes:
\begin{equation} \label{eqn: commuting square taut projection interior}
\begin{tikzcd}
CH^\star(\Acalbar_g^\Sigma) \ar[r,"\taut"] \ar[d] & R^\star(\Acalbar_g) \ar[d,"\lambda_g=0"] \\
CH^\star(\Acal_g) \ar[r,"\taut"] & R^\star(\Acal_g).
\end{tikzcd}
\end{equation}
It follows that the tautological projection of $\gamma \in CH^\star(\Acal_g)$ can be obtained by first computing the tautological projection of an extension
\[ \overline{\gamma} \in CH^\star(\Acalbar_g^\Sigma) \]
and then setting $\lambda_g=0$ (see also \zcref{rmk: taut projection interior introduction}).
%Recall the homomorphism conjecture, stated in its strongest form as follows:
%\begin{conjecture}[{\cite[Conjecture~8]{HomomorphismConjecture}}] \label{conj: homomorphism}
%    The tautological projection
%    \[ \taut \colon \CH^\star(\Acal_{g}) \to R^\star(\Acal_{g}) \]
%    respects the ring structure.
%\end{conjecture}
Inside the moduli space $\Acal_{g-1}$, consider the respective \emph{closures} of the loci of Jacobian and Prym varieties. Denote these closed substacks
\[ \Jac_{g-1}, \Prym_{g} \subseteq \Acal_{g-1} \]
respectively. A result of Wirtinger (see \cite[Section~2]{BeauvillePrym}) gives
\[ \Jac_{g-1} \subseteq \Prym_{g} \]
with $g=7$ being the first interesting case. It is natural to study the relationship between their Chow cycles. We prove the following result.

\begin{theorem}[Theorem~\ref{thm: non divisible introduction}] \label{prop: Jacobian Prym not divisible} For $g \in \{7,8,9,10\}$, we have
\[ [\Prym_{g}] \nmid [\Jac_{g-1}] \]
inside $\CH^\star(\Acal_{g-1})$. That is, there is no $\gamma \in \CH^\star(\Acal_{g-1})$ such that
\[ [\Prym_{g}] \cdot \gamma = [\Jac_{g-1}]. \]
In particular, there is no substack $V \subseteq \Acal_{g-1}$ which intersects $\Prym_{g}$ transversely and is such that
\[ \Prym_{g} \cap V = \Jac_{g-1}. \]
\end{theorem}

\begin{proof}
We first argue for $g \geqslant 8$. Divisibility in $\mathsf{CH}^\star(\Acal_{g-1})$ implies a fortiori divisibility in $H^\star(\Acal_{g-1};\Q)$. It thus suffices to show that $[\Prym_{g}] \nmid [\Jac_{g-1}]$ in $H^\star(\Acal_{g-1};\Q)$. Suppose for a contradiction that
\[ [\Prym_{g}] \cdot \gamma = [\Jac_{g-1}] \]
for some $\gamma \in H^6(\Acal_{g-1};\Q)$. Borel's homological stability theorem \cite{Borel1,Borel2} implies that for $g -1 \geqslant k$, every element of
\[ H^k(\Acal_{g-1};\Q) \]
is tautological, see e.g. \cite[Theorem~1.2]{Tshishiku}. In our case we see that $\gamma$ must be tautological for $g \geqslant 8$. However the tautological projection is multiplicative for products where at least one of the classes is tautological (see \cite[Equation (3)]{HomomorphismConjecture}). We thus obtain
\[ \taut[\Prym_{g}] \cdot \gamma  = \taut[\Jac_{g-1}], \]
but this contradicts \zcref{prop: tauts not divisible below} below. 

For $g=7$ homological stability does not apply directly. However, the image of $\CH^3(\Acal_6)$ under the cycle class map lands in the lowest weight part of $H^6(\Acal_6;\Q)$:
\begin{equation} \label{eqn: graded piece cohomology of A6} W_6 H^6(\Acal_6;\Q) = Gr_6^W H^6(\Acal_6;\Q).\end{equation}
We will show that this is tautological: the previous argument then applies to deduce the result.

Using Ta\"ibi's improved bound \cite[Theorem~33]{HulekTomassi} with $g=k=6$ and $\lambda=0$, the intersection cohomology 
\[
IH^6(\Acalbar_6^{\operatorname{Sat}};\QQ)
\]
of the Satake compactification is tautological, and therefore so are its graded pieces. By \cite[Lemma~2]{Durfee} the graded pieces of the intersection cohomology of the Satake compactification surject onto the corresponding graded pieces of the intersection cohomology of $\Acal_6$. But since $\Acal_6$ is smooth, these are simply the graded pieces of the cohomology. It follows that \eqref{eqn: graded piece cohomology of A6} is tautological as required.
\end{proof}

\begin{remark}
    The idea that stability should be relevant to the proof of \zcref{prop: Jacobian Prym not divisible} was suggested to us by Gabriele~Mondello. Sam~Canning provided help with the $g=7$ case.
\end{remark}

\begin{proposition} \label{prop: tauts not divisible below} For $g \in \{7,8,9,10\}$, we have
\[ \taut[\Prym_{g}] \nmid \taut[\Jac_{g-1}] \]
in $R^\star(\Acal_{g-1})$.
\end{proposition}

\begin{proof} The proof is computational, using the accompanying \texttt{Sage} code, specifically the commands:
\begin{verbatim}
sage: load("_5_Jacobians.sage");
sage: compare_interior_Jacobian_Prym_projections(g-1);
\end{verbatim}
We compute the tautological projections by computing the tautological projections of extensions $\overline{\gamma}$ and then setting $\lambda_{g-1}=0$. Specifically, we take the extensions
\[ \Prym_\star \left[\Mcalbar^{[S_2]}_{g,0}(\Bcal S_2)^\Sigma \right], \qquad \operatorname{Tor}_\star \left[\Mcalbar_{g-1}^\Sigma \right], \]
where $\operatorname{Tor}$ is the compactified Torelli map. Generic injectivity of the Torelli and Prym maps ensure that these extend $\Prym_{g}$ and $\Jac_{g-1}$. See Appendices~\ref{sec: tables Prym taut}~and~\ref{sec: tables Jac taut} for tables of the relevant tautological projections.
\end{proof}

%\begin{landscape}
\appendix
\section{Code and tables} \label{appendix}

\noindent For $d=2$ we implement our algorithm (\zcref{sec: algorithm}) in accompanying \texttt{Sage} code. The base cases of the algorithm are integrals of chi and psi classes on $\Mcalbar_{g,n}$. These are computed using the \texttt{admcycles} implementation of Faber's algorithm, authored primarily by Zheming~Sun (at the time of writing, this is only included as part of the development version of \texttt{admcycles}, available via GitLab). 

Navigating to the package directory and then opening a \texttt{Sage} terminal, the code is run by entering:
\begin{verbatim}
sage: load("_4_Execution.sage");
sage: taut_projection(g,n);
\end{verbatim}
Here $g$ is the genus of the base curve and $n$ is the number of branch points. 

\subsection{Fixed number of branch points} \label{sec: tables Prym taut} \label{appendix Prym} Using the above code, we present tables of tautological projections for double covers with $0,2,4,6$ branch points. The first rows of Tables~\ref{table: no branch}, \ref{table: 4 branch}, \ref{table: 6 branch} cover the codimension-zero cases discussed in \zcref{sec: codim zero degree 2}.

\begin{remark}
    In \zcref{table: no branch}, the cases $g=2,3,4,5$ are $0$ for dimension reasons. However we have also calculated these using our code, and this calculation is in fact non-trivial, involving summing together a large number of non-zero terms. The fact that the code returns $0$ as the final answer therefore provides additional consistency checks for the algorithm and its implementation.
\end{remark}

%%%%%%%%%%%
% TABLE FOR ETALE COVERS
%%%%%%%%%%%
\renewcommand{\arraystretch}{1.5} % increase row height locally
\begin{longtable}{| >{\centering\arraybackslash}p{0.5cm} | >{\centering\arraybackslash}p{0.5cm} | p{15cm} |}
% Heading for first page
%\caption{My long table}
%#\label{tab:my-long-table}
%\\
\caption{Tautological projections of Prym classes associated to \'etale double covers:} \label{table: no branch}
\\
\hline
\cellcolor{gray!20} $g$ & \cellcolor{gray!20} $\tilde{g}$ & \cellcolor{gray!20} $\taut \Prym_\star [\Rcalbar_{g,0}]$ \\
\thickhline
\endfirsthead

% Heading for subsequent pages
\multicolumn{3}{l}{Continued from previous page} \\
\hline
\cellcolor{gray!20} $g$ & \cellcolor{gray!20} $\tilde{g}$ & \cellcolor{gray!20} $\taut \Prym_\star [\Rcalbar_{g,0}]$ \\
\thickhline
\endhead

% Footer for non-final pages
\hline
\multicolumn{3}{r}{Continued on next page} \\
\endfoot

% Footer for final page
\hline
\endlastfoot

$6$ & $5$ & $\begin{array}{l} 27 \cdot \mathbbm{1} \end{array}$ \\ \hline
$7$ & $6$ & $\begin{array}{l} 1350 \cdot \lambda_1\lambda_2 + 135 \cdot \lambda_3 \end{array}$ \\ \hline
$8$ & $7$ & $\begin{array}{l} 50490 \cdot \lambda_1\lambda_2\lambda_4 + 68175 \cdot \lambda_1\lambda_6 -128925 \cdot \lambda_2\lambda_5 - 29295 \cdot \lambda_3\lambda_4 - 3645 \cdot \lambda_7 \end{array}$ \\ \hline
$9$ & $8$ & $\begin{array}{l} -263790 \cdot \lambda_1\lambda_2\lambda_3\lambda_6
    + 300510 \cdot \lambda_1\lambda_2\lambda_4\lambda_5
    - \frac{8482095920790}{2499347} \cdot \lambda_1 \lambda_3 \lambda_8 \\
    - \frac{949408594695}{2499347} \cdot \lambda_1\lambda_4\lambda_7
    + \frac{263527425}{691} \cdot \lambda_1\lambda_5\lambda_6
    + \frac{5371522346565}{2499347} \cdot \lambda_2\lambda_3\lambda_7 \\
    - \frac{753553530}{691} \cdot \lambda_2 \lambda_4 \lambda_6
    - \frac{160792695}{691} \cdot \lambda_3\lambda_4\lambda_5    
    + \frac{8638118044920}{2499347} \cdot \lambda_4 \lambda_8
    - \frac{31774940550}{2499347} \cdot \lambda_5\lambda_7 \end{array}$ \\ \hline
$10$ & $9$ & $\begin{array}{l} \frac{20971589110578}{30312097} \cdot \lambda_1\lambda_2\lambda_3\lambda_4\lambda_8
- \frac{675652428}{691} \cdot \lambda_1\lambda_2\lambda_3\lambda_5\lambda_7 + \\
 \frac{186771582}{691} \cdot \lambda_1\lambda_2\lambda_4\lambda_5\lambda_6 
- \frac{465196290339779118}{109638854849} \cdot \lambda_1\lambda_2\lambda_6\lambda_9
+ \frac{389051197318493847}{109638854849} \cdot \lambda_1\lambda_2\lambda_7\lambda_8 + \\
 \frac{1498929178289742792}{109638854849} \cdot \lambda_1\lambda_3\lambda_5\lambda_9
- \frac{1424195655580021527}{109638854849} \cdot \lambda_1\lambda_3\lambda_6\lambda_8
- \frac{30893490658761246}{109638854849} \cdot \lambda_1 \lambda_4\lambda_5\lambda_8 \\
- \frac{6170817271599}{2499347} \cdot \lambda_1\lambda_4\lambda_6\lambda_7
- \frac{484213554135353340}{109638854849} \cdot \lambda_1\lambda_8\lambda_9 \\
- \frac{679566423852872739}{109638854849} \cdot \lambda_2\lambda_3\lambda_4\lambda_9
+ \frac{271123494575427027}{109638854849} \cdot \lambda_2\lambda_3\lambda_5\lambda_8
+ \frac{10068429339828}{2499347} \cdot \lambda_2\lambda_3\lambda_6\lambda_7 \\
- \frac{262802475}{691} \cdot \lambda_2\lambda_4\lambda_5\lambda_7
+ \frac{346150683823154301}{109638854849} \cdot \lambda_2\lambda_7\lambda_9
- \frac{160279965}{691} \cdot \lambda_3\lambda_4\lambda_5\lambda_6 \\
- \frac{49434478228593}{109638854849} \cdot \lambda_3\lambda_6\lambda_9
- \frac{48408736832864622}{109638854849} \cdot \lambda_3\lambda_7\lambda_8 \\
- \frac{1128469485770899110}{109638854849} \cdot \lambda_4\lambda_5\lambda_9
+ \frac{1619854602283211724}{109638854849} \cdot \lambda_4\lambda_6\lambda_8
+ \frac{2195291197989}{2499347} \cdot \lambda_5\lambda_6\lambda_7\end{array}$ \\ \hline
\end{longtable}

%%%%%%%%%%%
% TABLE FOR 2 BRANCH POINTS
%%%%%%%%%%%
\renewcommand{\arraystretch}{1.5} % increase row height locally
\begin{longtable}{| >{\centering\arraybackslash}p{0.5cm} | >{\centering\arraybackslash}p{0.5cm} | p{15cm} |}
% Heading for first page
%\caption{My long table}
%#\label{tab:my-long-table}
%\\
\caption{Tautological projections of Prym classes associated to double covers with $2$ branch points:} \label{table: 2 branch}
\\
\hline
\cellcolor{gray!20} $g$ & \cellcolor{gray!20} $\tilde{g}$ & \cellcolor{gray!20} $\taut \Prym_\star [\Rcalbar_{g,2}]$ \\
\thickhline
\endfirsthead

% Heading for subsequent pages
\multicolumn{3}{l}{Continued from previous page} \\
\hline
\cellcolor{gray!20} $g$ & \cellcolor{gray!20} $\tilde{g}$ & \cellcolor{gray!20} $\taut \Prym_\star [\Rcalbar_{g,2}]$ \\
\thickhline
\endhead

% Footer for non-final pages
\hline
\multicolumn{3}{r}{Continued on next page} \\
\endfoot

% Footer for final page
\hline
\endlastfoot

$5$ & $5$ & $\begin{array}{l} 3672 \cdot \lambda_1 \end{array}$ \\ \hline
$6$ & $6$ & $\begin{array}{l} 307800 \cdot \lambda_1 \lambda_3 - 343440 \cdot \lambda_4 \end{array}$ \\ \hline
$7$ & $7$ & $\begin{array}{l} -4967784 \cdot \lambda_1 \lambda_2 \lambda_5 + 6040872 \cdot \lambda_1 \lambda_3 \lambda_4 - 15005304 \cdot \lambda_1 \lambda_7 + 20448288 \cdot \lambda_2 \lambda_6 \\ - 14361840 \cdot \lambda_3 \lambda_5 \end{array}$ \\ \hline
$8$ & $8$ & $\begin{array}{l} \frac{146028642314760}{2499347} \cdot \lambda_1 \lambda_2 \lambda_3 \lambda_7
- \frac{59032210320}{691} \cdot \lambda_1 \lambda_2 \lambda_4 \lambda_6
+ \frac{17111946600}{691} \cdot \lambda_1 \lambda_3 \lambda_4 \lambda_5 + \\
\frac{529075279198080}{2499347} \cdot \lambda_1 \lambda_4 \lambda_8
- \frac{1067143623392880}{2499347} \cdot \lambda_1 \lambda_5 \lambda_7
- \frac{765546989404320}{2499347} \cdot \lambda_2 \lambda_3 \lambda_8 + \\
\frac{557590986922320}{2499347} \cdot \lambda_2 \lambda_4 \lambda_7
+ \frac{137157889680}{691} \cdot \lambda_2 \lambda_5 \lambda_6
- \frac{65228924160}{691} \cdot \lambda_3 \lambda_4 \lambda_6 + \\
\frac{222975391680}{3617} \cdot \lambda_5 \lambda_8 
+ \frac{619128268649280}{2499347} \cdot \lambda_6 \lambda_7\end{array}$ \\ \hline
$9$ & $9$ & $\begin{array}{l}
- \frac{14056721160829451064}{109638854849} \cdot \lambda_1\lambda_2\lambda_3\lambda_4\lambda_9
+ \frac{30701442391430652216}{109638854849} \cdot \lambda_1\lambda_2\lambda_3\lambda_5\lambda_8 + \\
\frac{210713452910496}{2499347} \cdot \lambda_1\lambda_2\lambda_3\lambda_6\lambda_7
- \frac{116628096888}{691} \cdot \lambda_1\lambda_2\lambda_4\lambda_5\lambda_7 \\
- \frac{413586638029774446840}{109638854849} \cdot \lambda_1\lambda_2\lambda_7\lambda_9
+ \frac{11814086520}{691} \cdot \lambda_1\lambda_3\lambda_4\lambda_5\lambda_6 + \\
\frac{122688355864219332120}{109638854849} \cdot \lambda_1\lambda_3\lambda_6\lambda_9
+ \frac{440718389668182709296}{109638854849} \cdot \lambda_1\lambda_3\lambda_7\lambda_8 + \\
\frac{13903521806747068560}{109638854849} \cdot \lambda_1\lambda_4\lambda_5\lambda_9
- \frac{4195640627816953248}{109638854849} \cdot \lambda_1\lambda_4\lambda_6\lambda_8 \\
- \frac{906642848204856}{2499347} \cdot \lambda_1\lambda_5\lambda_6\lambda_7
+ \frac{177727020316473479616}{109638854849} \cdot \lambda_2\lambda_3\lambda_5\lambda_9 \\
- \frac{376358074932048990192}{109638854849} \cdot \lambda_2\lambda_3\lambda_6\lambda_8
+ \frac{14559548122269223776}{109638854849} \cdot \lambda_2\lambda_4\lambda_5\lambda_8 + \\
\frac{2232095924101968}{2499347} \cdot \lambda_2\lambda_4\lambda_6\lambda_7
+ \frac{106421355357345064512}{109638854849} \cdot \lambda_2 \lambda_8\lambda_9 \\
- \frac{64224363456}{691} \cdot \lambda_3\lambda_4\lambda_5\lambda_7
+ \frac{140739095285792303616}{109638854849} \cdot \lambda_3\lambda_7\lambda_9 \\
-\frac{304644186553472153280}{109638854849} \cdot \lambda_4\lambda_6\lambda_9
- \frac{353040743985654576672}{109638854849} \cdot \lambda_4\lambda_7\lambda_8 + \\
\frac{132227309374047440064}{109638854849} \cdot \lambda_5\lambda_6\lambda_8
\end{array}$ \\ \hline
\end{longtable}

%%%%%%%%%%%%
% TABLE FOR 4 BRANCH POINTS
%%%%%%%%%%%%
\renewcommand{\arraystretch}{1.5} % increase row height locally
\begin{longtable}{| >{\centering\arraybackslash}p{0.5cm} | >{\centering\arraybackslash}p{0.5cm} | p{15cm} |}
% Heading for first page
%\caption{My long table}
%#\label{tab:my-long-table}
%\\
\caption{Tautological projections of Prym classes associated to double covers with $4$ branch points:} \label{table: 4 branch}
\\
\hline
\cellcolor{gray!20} $g$ & \cellcolor{gray!20} $\tilde{g}$ & \cellcolor{gray!20} $\taut \Prym_\star [\Rcalbar_{g,4}]$ \\
\thickhline
\endfirsthead

% Heading for subsequent pages
\multicolumn{3}{l}{Continued from previous page} \\
\hline
\cellcolor{gray!20} $g$ & \cellcolor{gray!20} $\tilde{g}$ & \cellcolor{gray!20} $\taut \Prym_\star [\Rcalbar_{g,4}]$ \\
\thickhline
\endhead

% Footer for non-final pages
\hline
\multicolumn{3}{r}{Continued on next page} \\
\endfoot

% Footer for final page
\hline
\endlastfoot

$3$ & $4$ & $\begin{array}{l} 72 \cdot \mathbbm{1} \end{array}$ \\ \hline
$4$ & $5$ & $\begin{array}{l} 1944 \cdot \lambda_2 \end{array}$ \\ \hline
$5$ & $6$ & $\begin{array}{l} -14472 \cdot \lambda_1 \lambda_4 + 20520 \cdot \lambda_2 \lambda_3 + \frac{1677240}{691} \cdot \lambda_5 \end{array}$ \\ \hline
$6$ & $7$ & $\begin{array}{l} \frac{715093272}{5461} \cdot \lambda_1 \lambda_2 \lambda_6
- 209520 \cdot \lambda_1 \lambda_3 \lambda_5
+ 67608 \cdot \lambda_2 \lambda_3 \lambda_4
- \frac{708158174688}{3773551} \cdot \lambda_2 \lambda_7 \\
- \frac{69143544}{29713} \cdot \lambda_3 \lambda_6
+ \frac{90099000}{691} \cdot \lambda_4 \lambda_5\end{array}$ \\ \hline
$7$ & $8$ & $\begin{array}{l}
- \frac{218342252266920}{642332179} \cdot \lambda_1 \lambda_2 \lambda_3 \lambda_8
+ \frac{492834528058248}{642332179} \cdot \lambda_1 \lambda_2 \lambda_4 \lambda_7
+ \frac{155732544}{691} \cdot \lambda_1 \lambda_2 \lambda_5 \lambda_6 \\
- \frac{331386552}{691} \cdot \lambda_1 \lambda_3 \lambda_4 \lambda_6
+ \frac{40809565630224}{642332179} \cdot \lambda_1 \lambda_5 \lambda_8
+ \frac{1010845833141816}{642332179} \cdot \lambda_1 \lambda_6 \lambda_7 + \\
\frac{36129240}{691} \cdot \lambda_2 \lambda_3 \lambda_4 \lambda_5
+ \frac{1310334629478624}{642332179} \cdot \lambda_2 \lambda_4 \lambda_8
- \frac{2282971654291608}{642332179} \cdot \lambda_2 \lambda_5 \lambda_7 \\
- \frac{173674020411504}{642332179} \cdot \lambda_3 \lambda_4 \lambda_7
+ \frac{529052472}{691} \cdot \lambda_3 \lambda_5 \lambda_6
- \frac{494098439831136}{642332179} \cdot \lambda_6 \lambda_8
\end{array}$ \\ \hline
\end{longtable}

%%%%%%%%%%%%
% TABLE FOR 6 BRANCH POINTS
%%%%%%%%%%%%
\renewcommand{\arraystretch}{1.5} % increase row height locally
\begin{longtable}{| >{\centering\arraybackslash}p{0.5cm} | >{\centering\arraybackslash}p{0.5cm} | p{15cm} |}
% Heading for first page
%\caption{My long table}
%#\label{tab:my-long-table}
%\\
\caption{Tautological projections of Prym classes associated to double covers with $6$ branch points:} \label{table: 6 branch}
\\
\hline
\cellcolor{gray!20} $g$ & \cellcolor{gray!20} $\tilde{g}$ & \cellcolor{gray!20} $\taut \Prym_\star [\Rcalbar_{g,6}]$ \\
\thickhline
\endfirsthead

% Heading for subsequent pages
\multicolumn{3}{l}{Continued from previous page} \\
\hline
\cellcolor{gray!20} $g$ & \cellcolor{gray!20} $\tilde{g}$ & \cellcolor{gray!20} $\taut \Prym_\star [\Rcalbar_{g,6}]$ \\
\thickhline
\endhead

% Footer for non-final pages
\hline
\multicolumn{3}{r}{Continued on next page} \\
\endfoot

% Footer for final page
\hline
\endlastfoot

$1$ & $3$ & $\begin{array}{l} 720 \cdot \mathbbm{1} \end{array}$ \\ \hline
$2$ & $4$ & $\begin{array}{l} 2160 \cdot \lambda_1 \end{array}$ \\ \hline
$3$ & $5$ & $\begin{array}{l} 6480 \cdot \lambda_1 \lambda_2 - 2160 \cdot \lambda_3 \end{array}$ \\ \hline
$4$ & $6$ & $\begin{array}{l} 10800 \cdot \lambda_1 \lambda_2 \lambda_3 + \frac{4482000}{691} \cdot \lambda_1 \lambda_5 - 21600 \cdot \lambda_2 \lambda_4 - \frac{1252800}{691} \cdot \lambda_6 \end{array}$ \\ \hline
$5$ & $7$ & $\begin{array}{l}
6480 \cdot \lambda_1 \lambda_2 \lambda_3 \lambda_4
- \frac{81295455600}{3773551} \cdot \lambda_1 \lambda_2 \lambda_7
+ \frac{226291689360}{3773551} \cdot \lambda_1 \lambda_3 \lambda_6
+ \frac{11184480}{691} \cdot \lambda_1 \lambda_4 \lambda_5 \\
- 45360 \cdot \lambda_2 \lambda_3 \lambda_5
+ \frac{7259900400}{3773551} \cdot \lambda_3 \lambda_7
- \frac{237819330720}{3773551} \cdot \lambda_4 \lambda_6
\end{array}$ \\ \hline
\end{longtable}
%\end{landscape}

\subsection{Hyperelliptic Jacobians} \label{sec: tables hyperelliptic taut} \label{appendix hyperelliptic Jacobians} Setting $g=0$, we have
\[ \Prym(f) = \Jac(C) \]
and so our algorithm computes the tautological projection of the closure of the locus in $\Acalbar_{\tilde{g}}$ parametrising Jacobians of hyperelliptic curves. Note that in this case we have $\Etilde = \EE_C$, which is consistent with \eqref{eqn: ses Hodge bundles} and the fact that $\EE_B=0$. Fixing $n \geqslant 4$ even, we have:
\[ g_C = \tilde{g} = n/2-1, \qquad n = 2 g_C + 2.\]
For $n=4$ and $n=6$ the class has codimension zero and the calculation is classical, see Propositions~\ref{thm: genus 0 with 4 branch points}~and~\ref{thm: genus 0 with 6 branch points}. Higher codimension cases are given in the following table, computed using the commands:
\begin{verbatim}
sage: load("_4_Execution.sage");
sage: taut_projection(0,n);
\end{verbatim}

%%%%%%%%%%%%
% TABLE FOR HYPERELLIPTIC PROJECTIONS
%%%%%%%%%%%%
\renewcommand{\arraystretch}{1.5} % increase row height locally
\begin{longtable}{| >{\centering\arraybackslash}p{0.5cm} | >{\centering\arraybackslash}p{0.75cm} | p{15cm} |}
% Heading for first page
%\caption{My long table}
%#\label{tab:my-long-table}
%\\
\caption{Tautological projections of classes of hyperelliptic Jacobians:} \label{table: hyperelliptic Jacobians}
\\
\hline
\cellcolor{gray!20} $n$ & \cellcolor{gray!20} $g_C$ & \cellcolor{gray!20} $\taut \Prym_\star [\Rcalbar_{0,n}]$ \\
\thickhline
\endfirsthead

% Heading for subsequent pages
\multicolumn{3}{l}{Continued from previous page} \\
\hline
\cellcolor{gray!20} $n$ & \cellcolor{gray!20} $g_C$ & \cellcolor{gray!20} $\taut \Prym_\star [\Rcalbar_{0,n}]$ \\
\thickhline
\endhead

% Footer for non-final pages
\hline
\multicolumn{3}{r}{Continued on next page} \\
\endfoot

% Footer for final page
\hline
\endlastfoot

$4$ & $1$ & $\begin{array}{l} 6 \cdot \mathbbm{1} \end{array}$ \\ \hline
$6$ & $2$ & $\begin{array}{l} 720 \cdot \mathbbm{1} \end{array}$ \\ \hline
$8$ & $3$ & $\begin{array}{l} 725760 \cdot \lambda_1 \end{array}$ \\ \hline
$10$ & $4$ & $\begin{array}{l} 1480550400 \cdot \lambda_1 \lambda_2 - 1219276800 \cdot \lambda_3\end{array}$ \\ \hline
$12$ & $5$ & $\begin{array}{l} 2851017523200 \cdot \lambda_1 \lambda_2 \lambda_3
+ 6851638886400 \cdot \lambda_1 \lambda_5
- 10024545484800 \cdot \lambda_2 \lambda_4
\end{array}$ \\ \hline
$14$ & $6$ & $\begin{array}{l} 3012881743872000 \cdot \lambda_1 \lambda_2 \lambda_3 \lambda_4
+ \frac{35089527230005248000}{691} \cdot \lambda_1 \lambda_3 \lambda_6
+ \frac{10697236631617536000}{691} \cdot \lambda_1 \lambda_4 \lambda_5 + \\
- \frac{21038953217458176000}{691} \cdot \lambda_2 \lambda_3 \lambda_5
- \frac{42873307215298560000}{691} \cdot \lambda_4 \lambda_6 \end{array}$ \\ \hline
\end{longtable}
%\end{landscape}

\subsection{Jacobian locus} \label{sec: tables Jac taut} \label{appendix Jacobians} Our package also includes code for computing the tautological projection of the Torelli class:
\[ \taut \operatorname{Tor}_*[\Mcalbar_g^\Sigma]. \]
This is used to prove \zcref{prop: Jacobian Prym not divisible} above. The code is run by entering:
\begin{verbatim}
sage: load("_5_Jacobians.sage");
sage: taut_projection_Jacobians(g);
\end{verbatim}
The output is collected in the following table. Setting $\lambda_g=0$ recovers the formulae in \cite[Section~1.3]{HomomorphismConjecture} for the tautological projections on the non-proper moduli space $\Acal_g$.

%%%%%%%%%%%%
% TABLE FOR JACOBIAN PROJECTION
%%%%%%%%%%%%
\renewcommand{\arraystretch}{1.5} % increase row height locally
\begin{longtable}{| >{\centering\arraybackslash}p{0.5cm} | p{16cm} |}
% Heading for first page
%\caption{My long table}
%#\label{tab:my-long-table}
%\\
\caption{Tautological projections of Torelli classes:} \label{table: Torelli}
\\
\hline
\cellcolor{gray!20} $g$ & \cellcolor{gray!20} $\taut \operatorname{Tor}_\star [\Mcalbar_g]$ \\
\thickhline
\endfirsthead

% Heading for subsequent pages
\multicolumn{2}{l}{Continued from previous page} \\
\hline
\cellcolor{gray!20} $g$ & \cellcolor{gray!20} $\taut \operatorname{Tor}_\star [\Mcalbar_g]$ \\
\thickhline
\endhead

% Footer for non-final pages
\hline
\multicolumn{2}{r}{Continued on next page} \\
\endfoot

% Footer for final page
\hline
\endlastfoot

$2$ & $\begin{array}{l} 1 \cdot \mathbbm{1} \end{array}$ \\ \hline
$3$ & $\begin{array}{l} 2 \cdot \mathbbm{1} \end{array}$ \\ \hline
$4$ & $\begin{array}{l} 16 \cdot \lambda_1 \end{array}$ \\ \hline
$5$ & $\begin{array}{l} 144 \cdot \lambda_1 \lambda_2 - 96 \cdot \lambda_3 \end{array}$ \\ \hline
$6$ & $\begin{array}{l} 768 \cdot \lambda_1 \lambda_2 \lambda_3 + \frac{948096}{691} \cdot \lambda_1 \lambda_5 - 2304 \cdot \lambda_2 \lambda_4 - \frac{496128}{691} \cdot \lambda_6 \end{array}$ \\ \hline
$7$ & $\begin{array}{l} 1536 \cdot \lambda_1 \lambda_2 \lambda_3 \lambda_4
- \frac{6553344}{691} \cdot \lambda_1 \lambda_2 \lambda_7
+ \frac{15044352}{691} \cdot \lambda_1 \lambda_3 \lambda_6
+ \frac{4418304}{691} \cdot \lambda_1 \lambda_4 \lambda_5
- 13824 \cdot \lambda_2 \lambda_3 \lambda_5 + \\
\frac{1937664}{691} \cdot \lambda_3 \lambda_7
- \frac{17685504}{691} \cdot \lambda_4 \lambda_6 \end{array}$ \\ \hline
$8$ & $\begin{array}{l} \frac{552960}{691} \cdot \lambda_1 \lambda_2 \lambda_3 \lambda_4 \lambda_5
- \frac{23136362496}{2499347} \cdot \lambda_1 \lambda_2 \lambda_4 \lambda_8
- \frac{203316609024}{2499347} \cdot \lambda_1 \lambda_2 \lambda_5 \lambda_7 
+ \frac{139564449792}{2499347} \cdot \lambda_1 \lambda_3 \lambda_4 \lambda_7 + \\
\frac{25731072}{691} \cdot \lambda_1 \lambda_3 \lambda_5 \lambda_6
- \frac{835083214848}{2499347} \cdot \lambda_1 \lambda_6 \lambda_8
- \frac{12533760}{691} \cdot \lambda_2 \lambda_3 \lambda_4 \lambda_6
+ \frac{305804943360}{2499347} \cdot \lambda_2 \lambda_5 \lambda_8 + \\
\frac{500106387456}{2499347} \cdot \lambda_2 \lambda_6 \lambda_7
+ \frac{74317307904}{2499347} \cdot \lambda_3 \lambda_4 \lambda_8
- \frac{311646117888}{2499347} \cdot \lambda_3 \lambda_5 \lambda_7
- \frac{14966784}{691} \cdot \lambda_4 \lambda_5 \lambda_6 + \\
\frac{317411647488}{2499347} \cdot \lambda_7 \lambda_8 \end{array}$ \\ \hline

$9$ & $\begin{array}{l} \frac{55296}{691} \cdot \lambda_1 \lambda_2 \lambda_3 \lambda_4 \lambda_5 \lambda_6
+ \frac{2937426750664704}{109638854849} \cdot \lambda_1 \lambda_2 \lambda_3 \lambda_6 \lambda_9
+ \frac{6325093578571776}{109638854849} \cdot \lambda_1 \lambda_2 \lambda_3 \lambda_7 \lambda_8 + \\
\frac{4381839876612096}{109638854849} \cdot \lambda_1 \lambda_2 \lambda_4 \lambda_5 \lambda_9
- \frac{15284050814287872}{109638854849} \cdot \lambda_1 \lambda_2 \lambda_4 \lambda_6 \lambda_8
- \frac{26725330944}{2499347} \cdot \lambda_1 \lambda_2 \lambda_5 \lambda_6 \lambda_7 + \\
\frac{2482981630377984}{109638854849} \cdot \lambda_1 \lambda_3 \lambda_4 \lambda_5 \lambda_8
+ \frac{64143802368}{2499347} \cdot \lambda_1 \lambda_3 \lambda_4 \lambda_6 \lambda_7
- \frac{128782037512912896}{109638854849} \cdot \lambda_1 \lambda_3 \lambda_8 \lambda_9 + \\
\frac{103668177851449344}{109638854849} \cdot \lambda_1 \lambda_4 \lambda_7 \lambda_9
+ \frac{12220301357899776}{109638854849} \cdot \lambda_1 \lambda_5 \lambda_6 \lambda_9
- \frac{71352160887840768}{109638854849} \cdot \lambda_1 \lambda_5 \lambda_7 \lambda_8 + \\
- \frac{2875392}{691} \cdot \lambda_2 \lambda_3 \lambda_4 \lambda_5 \lambda_7
+ \frac{5296779103666176}{109638854849} \cdot \lambda_2 \lambda_3 \lambda_7 \lambda_9
- \frac{61390307886907392}{109638854849} \cdot \lambda_2 \lambda_4 \lambda_6 \lambda_9 + \\
\frac{10984215671635968}{109638854849} \cdot \lambda_2 \lambda_4 \lambda_7 \lambda_8
+ \frac{53920539105214464}{109638854849} \cdot \lambda_2 \lambda_5 \lambda_6 \lambda_8
+ \frac{175624234795008}{109638854849} \cdot \lambda_3 \lambda_4 \lambda_5 \lambda_9 \\
- \frac{6747993604325376}{109638854849} \cdot \lambda_3 \lambda_4 \lambda_6 \lambda_8
- \frac{95704104960}{2499347} \cdot \lambda_3 \lambda_5 \lambda_6 \lambda_7
+ \frac{97418943174721536}{109638854849} \cdot \lambda_4 \lambda_8 \lambda_9 + \\
\frac{11390407624802304}{109638854849} \cdot \lambda_5 \lambda_7 \lambda_9
+ \frac{27077777026301952}{109638854849} \cdot \lambda_6 \lambda_7 \lambda_8 \end{array}$  \\ %\hline
\end{longtable}

\bibliographystyle{amsalpha}
\bibliography{Bibliography}
\footnotesize

\,

\noindent Yoav Len. \email{\href{mailto:yoav.len@st-andrews.ac.uk}{yoav.len@st-andrews.ac.uk}} \\
Mathematical Institute, University of St Andrews, St Andrews, KY16 9SS, UK. \medskip

\noindent Sam Molcho. \email{\href{samouil.molcho@uniroma1.it}{{samouil.molcho@uniroma1.it}}} \\
Dipartimento di Matematica, Sapienza Universit\`a di Roma, Piazzale Aldo Moro, 5, 00185 Roma RM, Italia. \medskip
 
\noindent Navid Nabijou. \email{\href{n.nabijou@qmul.ac.uk}{{n.nabijou@qmul.ac.uk}}} \\
School of Mathematical Sciences, Queen Mary University of London, E1 4NS, UK. 

\end{document}

%% file: IntroductionTables.tex
    
\renewcommand{\arraystretch}{1.5} % increase row height locally
\begin{longtable}{| >{\centering\arraybackslash}p{0.5cm} | p{16cm} |}
% Heading for first page
%\caption{My long table}
%#\label{tab:my-long-table}
%\\
\caption*{Prym classes associated to \'etale double covers (see \zcref{appendix Prym})} \label{table: no branch}
\\
\hline
\cellcolor{gray!20} $g$ & \cellcolor{gray!20} $\taut \Prym_\star [\Rcalbar_{g,0}] \in R^\star(\Acalbar_{g-1})$ \\
\thickhline
\endfirsthead

% Heading for subsequent pages
\multicolumn{2}{l}{Continued from previous page} \\
\hline
\cellcolor{gray!20} $g$ & \cellcolor{gray!20} $\taut \Prym_\star [\Rcalbar_{g,0}] \in R^\star(\Acalbar_{g-1})$ \\
\thickhline
\endhead

% Footer for non-final pages
\hline
\multicolumn{2}{r}{Continued on next page} \\
\endfoot

% Footer for final page
\hline
\endlastfoot

$6$ & $\begin{array}{l} 27 \cdot \mathbbm{1} \end{array}$ \\ \hline
$7$ & $\begin{array}{l} 1350 \cdot \lambda_1\lambda_2 + 135 \cdot \lambda_3 \end{array}$ \\ \hline
$8$ & $\begin{array}{l} 50490 \cdot \lambda_1\lambda_2\lambda_4 + 68175 \cdot \lambda_1\lambda_6 -128925 \cdot \lambda_2\lambda_5 - 29295 \cdot \lambda_3\lambda_4 - 3645 \cdot \lambda_7 \end{array}$ \\ \hline
$9$ & $\begin{array}{l} -263790 \cdot \lambda_1\lambda_2\lambda_3\lambda_6
    + 300510 \cdot \lambda_1\lambda_2\lambda_4\lambda_5
    - \frac{8482095920790}{2499347} \cdot \lambda_1 \lambda_3 \lambda_8 \\
    - \frac{949408594695}{2499347} \cdot \lambda_1\lambda_4\lambda_7
    + \frac{263527425}{691} \cdot \lambda_1\lambda_5\lambda_6
    + \frac{5371522346565}{2499347} \cdot \lambda_2\lambda_3\lambda_7 \\
    - \frac{753553530}{691} \cdot \lambda_2 \lambda_4 \lambda_6
    - \frac{160792695}{691} \cdot \lambda_3\lambda_4\lambda_5    
    + \frac{8638118044920}{2499347} \cdot \lambda_4 \lambda_8
    - \frac{31774940550}{2499347} \cdot \lambda_5\lambda_7 \end{array}$ \\ \hline
$10$ & $\begin{array}{l} \frac{20971589110578}{30312097} \cdot \lambda_1\lambda_2\lambda_3\lambda_4\lambda_8
- \frac{675652428}{691} \cdot \lambda_1\lambda_2\lambda_3\lambda_5\lambda_7 + \\
 \frac{186771582}{691} \cdot \lambda_1\lambda_2\lambda_4\lambda_5\lambda_6 
- \frac{465196290339779118}{109638854849} \cdot \lambda_1\lambda_2\lambda_6\lambda_9
+ \frac{389051197318493847}{109638854849} \cdot \lambda_1\lambda_2\lambda_7\lambda_8 + \\
 \frac{1498929178289742792}{109638854849} \cdot \lambda_1\lambda_3\lambda_5\lambda_9
- \frac{1424195655580021527}{109638854849} \cdot \lambda_1\lambda_3\lambda_6\lambda_8
- \frac{30893490658761246}{109638854849} \cdot \lambda_1 \lambda_4\lambda_5\lambda_8 \\
- \frac{6170817271599}{2499347} \cdot \lambda_1\lambda_4\lambda_6\lambda_7
- \frac{484213554135353340}{109638854849} \cdot \lambda_1\lambda_8\lambda_9 \\
- \frac{679566423852872739}{109638854849} \cdot \lambda_2\lambda_3\lambda_4\lambda_9
+ \frac{271123494575427027}{109638854849} \cdot \lambda_2\lambda_3\lambda_5\lambda_8
+ \frac{10068429339828}{2499347} \cdot \lambda_2\lambda_3\lambda_6\lambda_7 \\
- \frac{262802475}{691} \cdot \lambda_2\lambda_4\lambda_5\lambda_7
+ \frac{346150683823154301}{109638854849} \cdot \lambda_2\lambda_7\lambda_9
- \frac{160279965}{691} \cdot \lambda_3\lambda_4\lambda_5\lambda_6 \\
- \frac{49434478228593}{109638854849} \cdot \lambda_3\lambda_6\lambda_9
- \frac{48408736832864622}{109638854849} \cdot \lambda_3\lambda_7\lambda_8 \\
- \frac{1128469485770899110}{109638854849} \cdot \lambda_4\lambda_5\lambda_9
+ \frac{1619854602283211724}{109638854849} \cdot \lambda_4\lambda_6\lambda_8
+ \frac{2195291197989}{2499347} \cdot \lambda_5\lambda_6\lambda_7\end{array}$ \\ \hline
\end{longtable}

%%%%%%%%%%%%
% TABLE FOR HYPERELLIPTIC PROJECTIONS
%%%%%%%%%%%%
\renewcommand{\arraystretch}{1.5} % increase row height locally
\begin{longtable}{| >{\centering\arraybackslash}p{0.75cm} | p{15cm} |}
% Heading for first page
%\caption{My long table}
%#\label{tab:my-long-table}
%\\
\caption*{Classes of hyperelliptic Jacobians (see \zcref{appendix hyperelliptic Jacobians})}
\\
\hline
\cellcolor{gray!20} $g$ & \cellcolor{gray!20} $\taut \Prym_\star [\Rcalbar_{0,2g+2}] \in R^\star (\Acalbar_g)$ \\
\thickhline
\endfirsthead

% Heading for subsequent pages
\multicolumn{2}{l}{Continued from previous page} \\
\hline
\cellcolor{gray!20} $g$ & \cellcolor{gray!20} $\taut \Prym_\star [\Rcalbar_{0,2g+2}] \in R^\star(\Acalbar_g)$ \\
\thickhline
\endhead

% Footer for non-final pages
\hline
\multicolumn{2}{r}{Continued on next page} \\
\endfoot

% Footer for final page
\hline
\endlastfoot

$1$ & $\begin{array}{l} 6 \cdot \mathbbm{1} \end{array}$ \\ \hline
$2$ & $\begin{array}{l} 720 \cdot \mathbbm{1} \end{array}$ \\ \hline
$3$ & $\begin{array}{l} 725760 \cdot \lambda_1 \end{array}$ \\ \hline
$4$ & $\begin{array}{l} 1480550400 \cdot \lambda_1 \lambda_2 - 1219276800 \cdot \lambda_3\end{array}$ \\ \hline
$5$ & $\begin{array}{l} 2851017523200 \cdot \lambda_1 \lambda_2 \lambda_3
+ 6851638886400 \cdot \lambda_1 \lambda_5
- 10024545484800 \cdot \lambda_2 \lambda_4
\end{array}$ \\ \hline
$6$ & $\begin{array}{l} 3012881743872000 \cdot \lambda_1 \lambda_2 \lambda_3 \lambda_4
+ \frac{35089527230005248000}{691} \cdot \lambda_1 \lambda_3 \lambda_6
+ \frac{10697236631617536000}{691} \cdot \lambda_1 \lambda_4 \lambda_5 + \\
- \frac{21038953217458176000}{691} \cdot \lambda_2 \lambda_3 \lambda_5
- \frac{42873307215298560000}{691} \cdot \lambda_4 \lambda_6 \end{array}$ \\ \hline
\end{longtable}\